\documentclass{amsart}

\usepackage{tikz}
\usepackage{tikz-cd,mathtools}

\usepackage{amssymb, stmaryrd}

\usepackage{hyperref}
\usepackage{cleveref}

\newtheorem{theorem}{Theorem}[section]
\newtheorem{lemma}[theorem]{Lemma}
\newtheorem{proposition}{Proposition}

\theoremstyle{definition}
\newtheorem{definition}[theorem]{Definition}
\newtheorem{example}[theorem]{Example}

\theoremstyle{remark}
\newtheorem{remark}[theorem]{Remark}

\numberwithin{equation}{section}

\numberwithin{equation}{section}

\newcommand{\R}{\mathbb{R}}

\newcommand{\ttK}{\mathtt{K}}

\newcommand{\frk}{\mathfrak{k}}
\newcommand{\rN}{\mathrm{N}}

\newcommand{\fL}{\mathfrak{L}}

\newcommand{\rK}{\mathrm{K}}

\newcommand{\Yperp}{Y_{\perp}}
\newcommand{\sym}{\mathrm{sym}}
\newcommand{\asym}{\mathrm{skew}}

\newcommand{\Z}{\mathbb{Z}}

\newcommand{\cX}{\mathcal{X}}

\newcommand{\cB}{\mathcal{B}}

\newcommand{\cE}{\mathcal{E}}

\newcommand{\Herm}[1]{\mathrm{Sym}_{#1}}
\newcommand{\St}[2]{\mathrm{St}_{#1, #2}}

\newcommand{\cH}{\mathcal{H}}

\newcommand{\frj}{\mathfrak{j}}
\newcommand{\frJ}{\mathfrak{J}}
\newcommand{\frjH}{\mathfrak{j}_{\cH}}

\newcommand{\frb}{\mathfrak{b}}

\newcommand{\cF}{\mathcal{F}}

\newcommand{\cK}{\mathcal{K}}

\newcommand{\rM}{\mathcal{M}}
\newcommand{\cU}{\mathcal{U}}
\newcommand{\cT}{\mathcal{T}}
\newcommand{\cN}{\mathcal{N}}
\newcommand{\cQ}{\mathcal{Q}}

\newcommand{\lb}{\llbracket}
\newcommand{\rb}{\rrbracket}

\newcommand{\Gr}[2]{\mathrm{Gr}_{#1, #2}}

\newcommand{\cZ}{\mathcal{Z}}

\newcommand{\sfg}{\mathsf{g}}

\newcommand{\sfG}{\mathsf{G}}
\newcommand{\sfGQ}{\mathsf{G}_{\cQ}}
\newcommand{\txi}{\tilde{\xi}}
\newcommand{\qtxi}{q_{\tilde{\xi}}}
\newcommand{\rtxi}{r_{\tilde{\xi}}}

\newcommand{\teta}{\tilde{\eta}}
\newcommand{\qteta}{q_{\tilde{\eta}}}
\newcommand{\rteta}{r_{\tilde{\eta}}}

\newcommand{\tphi}{\tilde{\phi}}
\newcommand{\tomega}{\tilde{\omega}}
\newcommand{\hsfg}{\hat{\sfg}}

\newcommand{\pphi}{p_{\phi}}
\newcommand{\pxi}{p_{\xi}}
\newcommand{\peta}{p_{\eta}}

\newcommand{\cV}{\mathcal{V}}

\newcommand{\mrGamma}{\mathring{\Gamma}}

\newcommand{\Gammaa}{\Gamma^{\cQ}}
\newcommand{\sfT}{\mathsf{T}}

\newcommand{\qq}{\mathfrak{q}}

\newcommand{\ttX}{\mathtt{X}}
\newcommand{\ttY}{\mathtt{Y}}
\newcommand{\ttZ}{\mathtt{Z}}
\newcommand{\tts}{\mathtt{s}}
\newcommand{\rmh}{\mathrm{h}}
\newcommand{\rmv}{\mathrm{v}}
\newcommand{\rC}{\mathrm{C}}
\newcommand{\rmq}{\mathrm{h}}
\newcommand{\rmp}{\mathrm{v}}
\newcommand{\rmb}{\mathrm{b}}
\newcommand{\rCa}{\mathrm{C}^{\mathrm{\cQ}}}

\newcommand{\frm}{\mathfrak{m}}
\newcommand{\frc}{\mathfrak{c}}
\newcommand{\tDelta}{\tilde{\Delta}}
\newcommand{\tdelta}{\tilde{\delta}}

\newcommand{\frt}{\mathfrak{t}}

\newcommand{\tnabla}{\tilde{\nabla}}
\DeclareMathOperator{\rU}{\mathring{\mathtt{u}}}
\DeclareMathOperator{\rF}{F}

\DeclareMathOperator{\ad}{ad}
\DeclareMathOperator{\Ad}{Ad}
\DeclareMathOperator{\diag}{diag}
\DeclareMathOperator{\Tr}{Tr}

\DeclareMathOperator{\TrR}{Tr_{\R}}

\DeclareMathOperator{\Null}{Null}

\DeclareMathOperator{\OO}{O}

\DeclareMathOperator{\Exp}{Exp}

\DeclareMathOperator{\ft}{\mathsf{T}}
\DeclareMathOperator{\RcH}{R^{\cH}}
\DeclareMathOperator{\RcM}{R^{\rM}}
\DeclareMathOperator{\RcE}{R^{\cE}}

\DeclareMathOperator{\htK}{\hat{\ttK}}
\DeclareMathOperator{\hcR}{R}

\DeclareMathOperator{\ttH}{H}
\DeclareMathOperator{\ttQ}{Q}
\DeclareMathOperator{\ttV}{V}
\DeclareMathOperator{\rD}{D}
\DeclareMathOperator{\rR}{R}
\DeclareMathOperator{\rB}{B}
\DeclareMathOperator{\csr}{csr}
\DeclareMathOperator{\ssr}{ssr}
\DeclareMathOperator{\dI}{I}
\DeclareMathOperator{\frL}{L}

\DeclareMathOperator{\SOO}{SO}

\DeclareMathOperator{\GammaV}{\Gamma_{\cV}}
\DeclareMathOperator{\GammaH}{\Gamma_{\cH}}
\DeclareMathOperator{\GammaM}{\Gamma_{\rM}}
\DeclareMathOperator{\GammaE}{\Gamma_{\cE}}
\DeclareMathOperator{\Two}{I\!I}

\newcommand{\rA}{\mathrm{A}}
\newcommand{\rAd}{\mathrm{A}^{\dagger}}
\newcommand{\SOD}{\mathrm{S}(\mathrm{O}(d_0)\times \cdots\times\mathrm{O}(d_q))}
\newcommand{\sod}{\mathfrak{o}(d_0)\times \cdots\times\mathfrak{o}(d_q)}
\newcommand{\oo}{\mathfrak{o}}
\newcommand{\sbb}{\mathfrak{b}}
\newcommand{\cL}{\mathcal{L}}

\begin{document}
\title[Ambient space and metric operator]{Riemannian Geometry with differentiable ambient space and metric operator}
\author{Du Nguyen}

\email{nguyendu@post.harvard.edu}

\begin{abstract}We show Riemannian geometry could be studied by identifying the tangent bundle of a Riemannian manifold $\mathcal{M}$ with a subbundle of the trivial bundle $\mathcal{M}\times \mathcal{E}$, obtained by embedding $\mathcal{M}$ differentiably in a Euclidean space $\mathcal{E}$. Given such an embedding, we can extend the metric tensor on $\mathcal{M}$ to a (positive-definite) operator-valued function acting on $\mathcal{E}$, giving us an embedded ambient structure. The formulas for the Christoffel symbols and Riemannian curvature in local coordinates have simple generalizations to this setup. For a Riemannian submersion $\mathfrak{q}:\mathcal{M}\to \mathcal{B}$ from an embedded manifold $\mathcal{M}\subset \mathcal{E}$, we define a submersed ambient structure and obtain similar formulas, with the O'Neil tensor expressed in terms of the projection to the horizontal bundle $\mathcal{H}\mathcal{M}$. Using this framework, we provide the embedded and submersed ambient structures for the double tangent bundle $\mathcal{T}\mathcal{T}\mathcal{M}$ and the tangent of the horizontal bundle $\mathcal{T}\mathcal{H}\mathcal{M}$, describe the fibration of a horizontal bundle over the tangent bundle of the base manifold and extend the notion of a canonical flip to the submersion case. We obtain a formula for horizontal lifts of Jacobi fields, and a new closed-form formula for Jacobi fields of naturally reductive homogeneous spaces. We construct natural metrics on these double tangent bundles, in particular, extending Sasaki and other natural metrics to the submersion case. We illustrate by providing explicit calculations for several manifolds.
\end{abstract}
\maketitle

\section{Introduction} It is well-known that a Riemannian manifold could be embedded isometrically in a Euclidean space, via the Nash embedding theorems. However, this is difficult technically. On the other hand, manifolds are often encountered as a differentiable submersion of a differentiable embedding. It turns out it is relatively easy to do geometry in this context, where the metric on the manifold is not necessarily induced from the Euclidean space, but defined by an operator.

This approach to computation, initiated in \cite{Edelman_1999} in the optimization literature, has been very successful, leading to applications in optimization, statistics, and computer vision. The computation of the Levi-Civita connection either uses the calculus of variation or the well-known formulas for embedded or submersed manifolds. We attempted to suggest a simplified framework in \cite{Nguyen2020a}.

We show this approach is also fruitful in studying Riemannian geometry itself in this article. In a sense, the approach could be considered dual to the local chart approach. Its main advantage is all formulas are defined and computed globally. The local formulas involving Christoffel symbols, for example, the {\it curvature formula}, have very straightforward global\slash embedded counterparts which we will explain shortly. The applicability comes from the fact that we only need a differentiable embedding instead of a Riemannian embedding, and that there are also similarly simple formulas in the submersion case. Furthermore, we will show Jacobi fields and tangent\slash horizontal bundle metrics can be expressed and computed easily in this formulation, and obtain several new results. It is interesting to note that the global formulas are very similar to the local ones and are easy to use.

For a Riemannian manifold $\rM$ embedded (differentiably) in a Euclidean space $\cE$, at each point $x\in\rM$, the tangent space $\cT_x\rM$ is identified with a subspace of $\cE$, thus the tangent bundle $\cT\rM$ is a subbundle of $\rM\times \cE$. We show there exists a positive-definite operator-valued function $\sfg$ from $\rM$ into the space of linear operators $\fL(\cE, \cE)$ on $\cE$, inducing the original Riemannian metric on $\rM$. It extends the bundle metric from $\cT\rM$ to $\rM\times \cE$. The extension is not unique given an intrinsic metric on $\rM$. To use this approach to compute intrinsic Riemannian measures, we need to make a choice of $\sfg$, and the computational result will be independent of the choice. The operator $\sfg$ induces a {\it projection} $\Pi$ from $\rM\times\cE$ to the tangent bundle $\cT\rM$ of $\rM$, or in the case where we have a Riemannian submersion $\qq:\rM \to \cB$, a projection $\ttH$ from $\rM\times\cE$ to the horizontal subbundle $\cH\rM\subset \cT\rM$ associated with this submersion. These projections (considered as operator-valued functions) are pivotal in this approach.

First, the projection $\Pi$ to the tangent bundle defines a connection on $\rM$, defined simply as $\Pi(\rD_{\ttX} \ttY)$ for vector fields $\ttX$ and $\ttY$, where $\ttY$ is identified with an $\cE$-valued function on $\rM$, using the identification just discussed, $\rD_{\ttX}\ttY$ denotes the directional derivative (covariant derivative using the trivial connection on the trivial bundle defined by the embedding of $\rM$ in $\cE$). In general, this connection is not compatible with metric, but if the metric operator $\sfg$ is constant, it is identical to the Levi-Civita connection. Otherwise, it differs from the Levi-Civita connection by a tensor $\mrGamma$, evaluated on two tangent vectors $\xi$ and $\eta$ to $\rM$ as
$$\mrGamma(\xi, \eta) = \frac{1}{2}\Pi\sfg^{-1}((\rD_{\xi}\sfg)\eta + (\rD_{\eta}\sfg)\xi - \cX(\xi, \eta))
$$
($\cX$ is the index-raised term, see \cref{prop:Levi}). This is analogous to the usual formula for Christoffel symbols. Following \cite{Edelman_1999}, we define a concept of a Christoffel function $\Gamma$, that could be used to compute Levi-Civita covariant derivatives. It is a function from $\rM$ to the space of bilinear functions from $\cE\times\cE$ to $\cE$, such that $\nabla_{\ttX}\ttY = \rD_{\ttX}\ttY + \Gamma(\ttX, \ttY)$, where $\ttX, \ttY$ are vector fields and $\nabla$ is the Levi-Civita covariant derivative. On tangent vectors, $\Gamma(\xi, \eta) = -(\rD_{\xi}\Pi)\eta + \mrGamma(\xi, \eta)$ (again, $\rD_{\xi}\Pi$ is the directional derivative of the operator-valued function $\Pi$). Given a Christoffel function $\Gamma$, the curvature of $\rM$ could be computed by the familiar formula
\begin{equation}\label{eq:cur_local}
  \rR_{\xi, \eta}\phi = -(\rD_{\xi}\Gamma)(\eta, \phi)
  +(\rD_{\eta}\Gamma)(\xi, \phi) - \Gamma(\xi, \Gamma(\eta, \phi)) +\Gamma(\eta, \Gamma(\xi, \phi))
\end{equation}
for three tangent vectors $\xi, \eta, \phi$ at $x\in\rM$. Thus, for the textbook example of the sphere $\rM=S^n\subset \cE=\R^{n+1}$, $\Pi_x\omega = \omega - xx^{\sfT}\omega $ (with $x\in S^n, \omega\in \R^{n+1}$), where $\Pi_x$ denotes the projection to the tangent space $\cT_x\rM$ at $x$, $\Gamma(\xi, \eta) = -(\rD_{\xi}\Pi)\eta = x\xi^{\sfT}\eta$, and \cref{eq:cur_local} gives us the curvature. There is no need to convert to trigonometric coordinates. In this instance, \cref{eq:cur_local} is equivalent to a $(1,3)$-form of the Gauss-Codazzi equation, but it could be used for metric operators defined only on $\rM$.

The approach also works for a Riemannian submersion, in \cref{theo:rsub}, we provide a formula similar to \cref{eq:cur_local}. It is equivalent to the $(1, 3)$ form of the O'Neil formula \cite{ONeil1966}. In both the embedded and submersed cases, the projections $\Pi$ (to the tangent bundle) and $\ttH$ (to the horizontal bundle $\cH\rM$) allow us to extend a tangent or a horizontal vector $\xi$ at a point $x\in \rM$ to a vector field (or horizontal vector field) $p_{\xi}$ on $\rM$, defined as $\pxi(y)=\Pi_y\xi$ (or $\ttH_y\xi$) for $y\in \rM$. In the embedded case, we could show that $[p_\xi, p_\eta]$ evaluated at $x$ vanishes for two tangent vectors $\xi$ and $\eta$, or equivalently $(\rD_{\xi}\Pi)_x\eta = (\rD_{\eta}\Pi)_x\xi$. This is not the case with horizontal projections, and the difference $(\rD_{\xi}\ttH)_x\eta - (\rD_{\eta}\ttH)_x\xi$ is exactly $2\rA_{\xi}\eta$, where $\rA$ is the O'Neil tensor. Thus, knowing $\ttH$ and its directional derivative is sufficient to compute the lift of the curvature of a submersed manifold $\cB$ if the curvature of $\rM$ is known. We can derive easily the {\it curvature of flag manifolds} from this approach, obtaining an alternative form of the curvature formula for {\it naturally reductive homogeneous spaces} (\cite{KobNom}, chapter 10). In general, the curvature computed by this approach could produce rather complicated expressions if the underlying fibration or symmetries of the manifold is not apparent. However, it makes available a procedural approach for all metrics.

Our next goal is to study {\it Jacobi fields}, which could be considered as curves on the tangent bundle $\cT\rM$ of $\rM\subset \cE$, obtained by taking directional derivatives of geodesics, considered as a function of both time and initial conditions, by a change in initial conditions. The initial data of the Jacobi field equation could be identified as a point on the double tangent bundle $\cT\cT\rM$, considered as a submanifold of $\cE^4$. For $(x, v)\in \cT\rM\subset \cE^2$, a Jacobi field $\frJ(t)$ along a geodesic $\gamma(t)$, with initial condition $\gamma(0)=x, \dot{\gamma}(0) = v, \frJ(0) = (x, \Delta_{\frm})\in\cT\rM\subset\cE^2$ has initial time derivative of the form $\dot{\frJ}(0) = (x, \Delta_{\frm}, v, \Delta_{\frt})\in \cT\cT\rM\subset \cE^4$, where $(\Delta_{\frm}, \Delta_{\frt})\in\cE^2$ is a tangent vector in $\cT_{(x, v)}\cT\rM$, the tangent space of $\cT\rM$ at $(x, v)$. We describe $\cT\cT\rM$ as a submanifold of $\cE^4$, with constraints given in terms of the projection $\Pi$ and its directional derivative. To the best of our knowledge, the horizontal lift of a Jacobi field (from a curve on the tangent bundle $\cT\cB$ of a base manifold $\cB$ to a curve on the horizontal bundle $\cH\rM$ in a Riemannian submersion $\qq:\rM\to\cB$) has not been studied before, and it could be described quite explicitly in our framework. We show $d\qq_{|\cH\rM}:\cH\rM\to\cT\cB$ is a differentiable submersion, describe the vertical bundle $\cV\cH\rM$ of this submersion explicitly by a map $\rmb$ from the vertical bundle $\cV\rM$ to $\cV\cH\rM$, constructed by directional derivatives of $\ttH$. We also identify a subbundle $\cQ\cH\rM$ of $\cT\cH\rM$, which is transversal to the vertical bundle $\cV\cH\rM$, which will play the role of a horizontal bundle in the submersion $d\qq_{|\cH\rM}$. We define the {\it canonical flip} $\frjH$ on $\cQ\cH\rM$ (again, with the help of the O'Neil tensor) which corresponds to the {\it canonical flip} on $\cT\cT\cB$. The initial data of a lifted Jacobi field could be identified with a point of $\cQ\cH\rM$. Our main result for Jacobi fields is a horizontal lift formula in \cref{theo:jacsub}, using this canonical flip. We also obtain a formula for horizontal lifts of Jacobi fields of naturally reductive homogeneous spaces in \cref{theo:jacobi_submerse_nat}, further clarify the relationship between invariant vector fields and Jacobi fields. We also add partial results to a conjecture of Ziller characterizing symmetric spaces by zeros of Jacobi fields.

As we have a description of the double tangent bundle $\cT\cT\rM\subset\cE^4$, we can also explicitly construct {\it natural metrics} \cite{Sasaki,TriMuss,KowalSeki,Abbassi2005,GudKap} on $\cT\cT\rM$. We describe explicitly the connection map $\rC$, the metric operator $\sfG$, and its projection $\Pi_{\sfG}$ on $\cT\cT\rM$, for a family of metrics constructed based on two real-valued functions $\alpha, \beta$. The Sasaki metric \cite{Sasaki} ($\alpha=1, \beta=0$) and the Cheeger-Gromoll metric ($\alpha=\beta=(1+t)^{-1}$, constructed by Tricerri and Musso \cite{TriMuss}) are special cases of this family. Using the results in these cited works, we express the Christoffel function $\Gamma_{\sfG}$ corresponding to this metric in our framework.

Since $d\qq_{|\cH\rM}:\cH\rM\to\cT\cB$ is a differentiable submersion, naturally we wish to give $\cH\rM$ a metric $\sfGQ$ so $d\qq_{|\cH\rM}$ is a Riemannian submersion if $\cT\cB$ is equipped with the metric $\sfG$ above. The construction is similar to the embedded case, but the connection map $\rC$ is replaced by a modified counterpart $\rCa$, constructed with the help of both the Christoffel function and the map $\rmb$ defining the vertical bundle $\cV\cH\rM$. We provide formulas for the horizontal Christoffel function $\Gamma^{\cH}_{\sfGQ}$, allowing us to compute the horizontal lift of covariant derivatives on $\cT\cB$.

Besides the examples used to illustrate the concepts (including $\SOO(n)$, the tangent bundle of a sphere and flag manifolds), we present a detailed calculation for the Grassmann manifold $\Gr{p}{n}$, considered as the submersed image of the Stiefel manifold $\St{p}{n}$, providing explicit formulas for its Jacobi field, and the natural metric on the horizontal bundle $\cH\St{p}{n}$ corresponding to this submersion.
\section{Related works}The formulas for the Levi-Civita connections for Riemannian embedding and submersion are classical, for example in \cite{ONeil1983}, we make the observation that $\sfg$ is only required to be defined on $\rM$. Some results in \cref{sec:christ_func} overlap with \cite{Nguyen2020a}, but the focus of that paper is on numerical implementation and optimization, the proofs given here are also different. The foundational paper \cite{Edelman_1999} provided formulas in the style studied here for Grassmann and Stiefel manifolds, popularizing the method of Riemannian optimization from the earlier works of \cite{GaLu,Gabay1982}. A rather extensive collection of manifolds have been studied by this method, as quotients of products of Stiefel, Grassmann manifolds, Lie groups, and symmetric spaces and a few of their differential geometric measures have been implemented in computer codes, available in \cite{Manopt,Pymanopt,Geomstats}, among others. For the most part, these examples consider metrics that extend almost everywhere to the ambient space \cite{AMS_book, ONeil1983}, thus the role of the metric operator has not been emphasized, and the tensor $\mrGamma$ has not been studied. The treatment of the Gauss-Codazzi equation as a result about subbundles could be found in \cite{taylor2011partial}. We learn about the canonical flip from \cite{Michor}. The constant-coefficient differential equation for Jacobi fields of a naturally reductive space appeared in \cite{Rauch,Chavel,Ziller}. The closed-form formula for Jacobi fields for symmetric spaces, but not homogeneous spaces, is also well-known. The formula for Jacobi fields for $\SOO(n)$ traces back to \cite{GKR}. The idea of the horizontal lift of Jacobi fields in \cref{theo:jacsub} has been noted in Section 8 of \cite{Chavel}. Natural metrics were first studied in \cite{Sasaki} and subsequently extended by many authors. We provide the construction of families of natural metrics for horizontal bundles using our framework.

\section{Notations} By $\R^{m\times n}$, we denote the space of real matrices of size $m\times n$. The base inner product on an inner product space $\cE$ is denoted by $\langle,\rangle_{\cE}$. We use $\fL$ to denote the space of linear operators, for example, $\fL(\cE, \cE)$ is the space linear operator on $\cE$, $\fL(\cE\otimes\cE, \cE)$ is the space of $\cE$-valued bilinear operators on $\cE$. The metric operator on a manifold $\rM$ is often denoted by $\sfg$, $\Pi$ denotes the projection (under $\sfg$) from $\rM\times\cE$ to the tangent bundle $\cT\rM$ of $\rM$, considered as a function from $\rM$ to $\fL(\cE, \cE)$. If $\qq:\rM\to\cB$ is a Riemannian submersion, we denote by $\cV\rM$ and $\cH\rM$ the vertical and horizontal subbundles of $\cT\rM$, and $\ttV$, $\ttH$, the projections from $\cE$ to $\cV\rM$ and $\cH\rM$. For a tangent vector $\xi$, $\rD_{\xi}$ denotes the directional derivative in direction $\xi$ of a scalar, vector or operator-valued function. We use the same notation $\rD_{\ttX}$ for the derivatives by a vector field of scalar, vector, or operator-valued function. In particular, for two vector fields $\ttX$ and $\ttY$, $\rD_{\ttX}\ttY$ makes sense if we identify $\ttY$ with a function from $\rM$ to $\cE\supset \rM$. For a geodesic $\gamma(t)$, we denote by $\nabla_{d/dt}$ the covariant derivative in direction $\dot{\gamma}(t)$ along the curve. By $\Gamma$ and $\GammaM$ we denote the Christoffel function of $\rM$, $\GammaH$ and $\GammaV$ denote the horizontal and vertical Christoffel functions defined in \cref{sec:christ_func}. By $\nabla^{\cH}$ we denote $\ttH\nabla$, the horizontal component of the Levi-Civita connection. The O'Neil tensor is denoted by $\rA$ with adjoint $\rAd$, defined in \cref{subsec:Rsubmerse}. The tangent bundle projection map is typically denoted by $\pi$, $\rC$ is the connection map, $\frJ$ denotes the Jacobi field and $J$ denotes the tangent component of $\frJ$ in the embedding in $\cE^2$. By $\cQ$ we denote the horizontal subbundle of $\cT\cH\rM$, defined in \cref{sec:subm_tangent}.

We also use $\sfg_x$ to denote the valuation of $\sfg$ at $x$, this also applies for projections $\ttH_x$, vector fields $\ttY_x$, etc. Because of this, partial derivatives will always be denoted by $\partial$ to avoid confusion. The Riemannian exponential map is denoted by $\Exp$, with $\Exp_x v$ denotes the point $\gamma(1)$ for the geodesic $\gamma(t)$ with $\gamma(0) = x, \dot{\gamma}(0) = v$. The exponential map corresponding to a one-parameter subgroup of a Lie group is denoted by $\exp(tX)$, with $X$ an element of the corresponding Lie algebra, $t\in\R$.

For integers $n$ and $p$, we denote the orthogonal group in $\R^{n\times n}$ by $\OO(n)$, $\St{p}{n}$ the Stiefel submanifold of $\R^{n\times p}$, $\Gr{p}{n}$ the corresponding Grassmann manifold, $\Herm{p}$ the vector space of symmetric matrices in $\R^{p\times p}$, and $\oo(p)$ the space of antisymmetric matrices. We denote the symmetrize\slash asymmetrize operators by $\sym(A) = \frac{1}{2}(A + A^{\ft})$ and $\asym(A) = \frac{1}{2}(A - A^{\ft})$. By $\|\|_{\sfg}$ we denote the norm associated with the metric operator $\sfg$, and $\|\|_{\sfg, x}$ denotes the norm at a particular point.
\section{Embedded ambient structure: metric operator, projection, and the Levi-Civita connection}\label{sec:christ_func}
\subsection{Differentiable embedding and metric operator}\label{subsec:subman}
Let $\cE$ be an inner product space with the inner product of $\omega_1, \omega_2\in\cE$ denoted by $\langle\omega_1, \omega_2\rangle_{\cE}$. Assume $\rM$ is a differential submanifold of $\cE$. It is well-known, \cite{LeeSmooth} any (Hausdorff, $\sigma$-compact) differential manifold $\rM$ could be embedded to an inner product space $\cE$, for example, by the Whitney embedding theorem. We will assume $\rM$ is equipped with a Riemannian metric $\langle\rangle_R$, not necessarily the metric induced from the embedding in $\cE$. Let $\cT\rM$ be the tangent bundle of $\rM$, so at each point $x\in\rM$, the tangent space $\cT_x\rM$ is identified with a subspace of $\cE$. We define a metric operator as follows.
\begin{definition} Let $\rM$ be a Riemannian manifold, $\rM\subset \cE$ is a differentiable embedding of $\rM$ in a Euclidean space $(\cE, \langle\rangle_{\cE})$. A metric operator $\sfg$ on $\rM$ is a smooth operator-valued function $\sfg$ from $\rM$ to $\fL(\cE, \cE)$ such that $\sfg(x)$ is positive-definite for all $x\in\rM$ and $\langle \xi, \eta\rangle_{R, x} = \langle\xi, \sfg(x)\eta\rangle_{\cE}$ for all $x\in\rM, \xi, \eta\in \cT_x\rM$, where $\langle\rangle_{R, x}$ denotes the Riemannian metric evaluated at $x$. The triple $(\rM, \sfg, \cE)$ is called an embedded ambient structure (or simply ambient structure) of $\rM$.
\end{definition}
We will also write $\sfg_x$ for $\sfg(x)$. In general, a positive-definite operator $\sfg$ on an inner product space $\cE$ induces a new inner product  $\langle\rangle_{\sfg}$, defined as $\langle\omega_1, \sfg\omega_2\rangle_{\cE}$, and we write $\langle\rangle_{\sfg, x}$ for the inner product defined by $\sfg_x$. Recall if $\cE_1, \langle\rangle_{\cE_1}$ and $\cE_2, \langle\rangle_{\cE_2}$ are two inner product spaces, and $f:\cE_1\to\cE_2$ is a linear map, the adjoint $f^{\sfT}:\cE_2\to\cE_1$ of $f$ is the unique map satisfying $\langle \nu_2, f \omega_1\rangle_{\cE_2} = \langle f^{\sfT}\nu_2, \omega_1\rangle_{\cE_1}$, for $\nu_2\in\cE_2, \omega_1\in \cE_1$.

In \cref{prop:emb_exists}, we show metric operators always exist. The proof uses a standard result on projections stated below (which is also used in \cite{Nguyen2020a}).
\begin{lemma}\label{lem:projprop}Let $(\cE, \langle\rangle_{\cE})$ be an inner product space and $\cV\subset \cE$ be a subspace. Let $\sfg$ be a positive-definite operator on $\cE$. Assume there exists an inner product space $(\cE_{\rN},\langle\rangle_{\cE_{\rN}})$ and a linear map $\rN:\cE_{\rN}\to \cE$ such that $\rN(\cE_{\rN}) = \cV$ and $\rN$ is injective. Let $\rN^{\sfT}:\cE\to\cE_{\rN}$ be the adjoint of $\rN$ under the inner products $\langle\rangle_{\cE}$ and $\langle\rangle_{\cE_{\rN}}$, then $\rN^{\sfT}\sfg\rN$ is invertible and the projection $\Pi_{\sfg}$ from $\cE$ to $\cV$ under the metric induced by $\sfg$ is given by
  \begin{equation}\label{eq:Pig}
    \Pi_{\sfg}\omega = \rN(\rN^{\sfT}\sfg\rN)^{-1}\rN^{\sfT}\sfg\omega
  \end{equation}
  for $\omega\in\cE$. That means $\Pi_{\sfg}\omega\in \cV$, and $\langle \Pi_{\sfg}\omega, \sfg\eta\rangle_{\cE} = \langle \omega, \sfg\eta\rangle_{\cE}$ for any $\eta\in\cV$.
  Moreover, $\Pi_\sfg$ is idempotent and $\sfg\Pi_\sfg$ is self-adjoint under $\langle\rangle_{\cE}$.
\end{lemma}
\begin{proof} If $\rN^{\sfT}\sfg\rN\delta = 0$ for $\delta\in \cE_{\rN}$, then $\langle\sfg\rN\delta, \rN\delta \rangle_{\cE} = \langle \rN^{\sfT}\sfg\rN\delta, \delta\rangle_{\cE}=0$, this means $\rN\delta = 0$ since $\sfg$ is positive-definite, and so $\delta = 0$ since $\rN$ is injective. Thus $\rN^{\sfT}\sfg\rN$ is invertible. It is clear $\Pi_{\sfg}\omega\in \cV$. If $\eta\in \cV$, then $\eta = \rN\delta$ for $\delta\in \cE_{\rN}$, thus
  $$\langle\rN(\rN^{\sfT}\sfg\rN)^{-1}\rN^{\sfT}\sfg \omega, \sfg\eta\rangle_{\cE} =
\langle \omega,\sfg\rN (\rN^{\sfT}\sfg\rN)^{-1}\rN^{\sfT}\sfg\rN\delta\rangle_{\cE} = 
\langle \omega, \sfg\rN\delta\rangle_{\cE}=\langle \omega, \sfg\eta\rangle_{\cE}$$
Finally, it is clear from \cref{eq:Pig} that $\Pi_{\sfg}$ is idempotent and $\sfg\Pi_{\sfg}$ is self-adjoint.
\end{proof}
\begin{proposition}\label{prop:emb_exists}If a manifold $\rM$ is differentiably embedded in $\cE$, then $x\mapsto \Pi^{\cE}_x$ is a smooth operator-valued map from $\rM$ to $\fL(\cE, \cE)$, the space of linear operators in $\cE$, where $\Pi^{\cE}_x$ denotes the projection from $\cE$ to $\cT_x\rM$ under the inner product in $\cE$. Assume $\rM$ is equipped with a Riemannian metric $\langle\rangle_R$, then there exists a smooth operator-valued function $\sfg$ from $\rM$ to $\fL(\cE, \cE)$ such that $\sfg_x :=\sfg(x)$ is positive-definite for all $x\in\rM$ and for two tangent vectors  $\eta, \xi\in \cT_x\rM$ we have
  \begin{equation}\langle \eta, \xi\rangle_{R, x} = \langle\eta, \sfg_x\xi\rangle_{\cE}
  \end{equation}
If $\dI_{\cE}$ is the identity map of $\cE$, we can take $\sfg$ to be the operator defined by
\begin{equation}\label{eq:stdg}
  \sfg_x\omega = (\dI_{\cE} - \Pi^{\cE}_x)\omega + \sfg_{R, x}(\Pi^{\cE}_{x}\omega)
\end{equation}  
where $\sfg_{R, x}$ is the unique self-adjoint operator on $\cT_x\rM$ such that
  $$\langle \xi, \eta\rangle_R = \langle \xi, \sfg_{R, x}\eta\rangle_{\cE}$$
\end{proposition}
\begin{proof} As a projection is defined for any subspace, $\Pi_x^{\cE}$ is a well-defined map for all points $x\in \rM$. Let $m$ be the dimension of $\rM$, $\psi_{\cU}: \R^m\to \cU\subset\rM\subset \cE$ be a coordinate chart for $\cU\subset\rM$ near $x$ considered as a map from $\R^m$ to $\cE$. Then $d\psi_{\cU}(x)$ is a map from $\R^m$ to $\cE$, injective with image precisely $\cT_x\rM$. Applying \cref{lem:projprop} for the identity operator, $\Pi_x^{\cE} = d\psi_{\cU}(x)(d\psi_{\cU}^{\sfT}(x)d\psi_{\cU}(x))^{-1}d\psi_{\cU}^{\sfT}(x)$. Since $\psi$ is assumed to be smooth, $\Pi_x^{\cE}$ is smooth in $\cU$.
  
  To show $\sfg$ defined in \cref{eq:stdg} is self-adjoint, note that $\Pi_x^{\cE}$, thus $\dI_{\cE} - \Pi_x^{\cE}$ is self-adjoint. For $\omega_1, \omega_2\in \cE$ we have
  $$\langle \omega_1,\sfg_{R, x} \Pi^{\cE}_{x}\omega_2\rangle_{\cE}= \langle\Pi^{\cE}_{x} \omega_1,\sfg_{R, x} \Pi^{\cE}_{x}\omega_2\rangle_{\cE}
  = \langle\sfg_{R, x}\Pi^{\cE}_{x} \omega_1, \Pi^{\cE}_{x}\omega_2\rangle_{\cE}
  $$
  where the first equality is the defining property of a projection, and the second is because $\sfg_R$ is self-adjoint on $\cT_x\rM$. The last expression is $\langle\sfg_{R, x}\Pi^{\cE}_{x} \omega_1, \omega_2\rangle_{\cE}$ by property of a projection. Thus $\sfg_{R, x}\Pi^{\cE}_{x}$ is self-adjoint, hence $\sfg$ in \cref{eq:stdg} is self-adjoint. If $\langle\omega, \sfg_x\omega\rangle_{\cE} =0$ then $\langle\omega, (\dI_{\cE} - \Pi^{\cE}_x)\omega\rangle_{\cE} + \langle \omega, \sfg_{R, x}\Pi^{\cE}_x\omega\rangle_{\cE} = 0$, or
  $$\langle(\dI_{\cE} - \Pi^{\cE}_x)\omega, (\dI_{\cE} - \Pi^{\cE}_x)\omega\rangle_{\cE} + \langle\Pi^{\cE}_x \omega, \sfg_{R, x}\Pi^{\cE}_x\omega\rangle_{\cE}=0$$
using the fact that $(\dI_{\cE} - \Pi^{\cE}_x)$ is idempotent and self-adjoint on the first term and that $\Pi^{\cE}_x$ is a $\langle\rangle_{\cE}$-projection  on the second term (because $ \sfg_{R, x}\Pi^{\cE}_x\omega\in \cT_x\rM$). Both terms are nonnegative, therefore both have to be zero, hence $\omega=0$.
\end{proof}
It is easy to see $\sfg$ satisfying the proposition above is not unique: for example, for the sphere $\rM=S^n\subset\cE=\R^{n+1}$ with $\omega\in \cE$ and $x\in S^n$, the operator $\sfg(x)\omega = \omega +\beta xx^{\sfT}\omega$ for $\beta \geq 0$ is a positive-definite operator. While each $\beta$ defines a different operator on $\cE$, the induced metrics on $\cT_x\rM$ are given by the same expression $\eta^{\sfT}\eta$ for a tangent vector $\eta$.
\begin{proposition}\label{prop:Levi} Let $(\rM, \sfg, \cE)$ be an ambient structure of a Riemannian manifold $\rM$ with metric operator $\sfg$. Then at each $x\in \rM$, $\sfg_x$ defines an inner product $\langle \omega_1, \omega_2\rangle_{\sfg, x} =
  \langle \omega_1, \sfg_x\omega_2\rangle_{\cE}$ on $\cE$ for $\omega_1, \omega_2\in \cE$. Denote by $\Pi_{\sfg, x}$ the associated projection to $\cT_x\rM$. Then $\Pi_{\sfg}:x\mapsto \Pi_{\sfg, x}$ is a smooth map from $\rM$ to $\fL(\cE, \cE)$. There exists a smooth operator-valued function $\Gamma$ from $\rM$ to $\fL(\cE\otimes\cE, \cE)$, the space of $\cE$-valued bilinear forms on $\cE$, such that if $\ttX$ and $\ttY$ are vector fields on $\rM$, then
  \begin{equation}\label{eq:LvC}
(\nabla_{\ttX} \ttY)_x = (\rD_{\ttX}\ttY)_x + \Gamma(\ttX_x, \ttY_x)\end{equation}
  for all $x\in\rM$, where $\nabla_{\ttX} \ttY$ is the covariant derivative defined by the Levi-Civita connection of the metric induced by $\sfg$. Here, we identify $\ttY$ with a $\cE$-valued function and $(\rD_{\ttX}\ttY)_x$ is its directional derivative in the direction $\ttX_x$. We call such $\Gamma$ a {\bf Christoffel function} of $(\rM, \sfg, \cE)$. There could be more than one Christoffel function given a metric operator $\sfg$, however  for two tangent vector $\xi, \eta$, $\Gamma(\xi, \eta)$ is independent of the vector field extensions of $\xi$ and $\eta$, and is only dependent on the restriction of the metric operator $\sfg$ to the tangent bundle.

If $\cX$ is a smooth function from $\rM$ to the space $\fL(\cE\times\cE, \cE)$ of $\cE$-bilinear forms with value in $\cE$, such that at any point $x\in \rM$ and any triple $\xi, \eta, \phi$ of tangent vectors to $\rM$ at $x$,
\begin{equation}\label{eq:cX}
  \langle \cX(\xi, \eta)_x, \phi\rangle_{\cE} = \langle(\rD_{\phi}\sfg)_x\xi, \eta\rangle_{\cE}\end{equation}  
  Then for two tangent vectors $\xi, \eta$ at $x\in\rM$
  \begin{equation}\label{eq:GammaForm}\Gamma(\xi, \eta)_x = -(\rD_{\xi}\Pi)\eta + \frac{1}{2} \Pi_{\sfg, x}\sfg_x^{-1}(\rD_{\xi}\sfg\eta +\rD_{\eta}\sfg\xi-\cX(\xi, \eta))_x\end{equation}
\end{proposition}
\begin{proof} Let $m$ be the dimension of $\rM$ and $\psi_{\cU}:\R^m\to\cU\subset \cE$ be a coordinate function for an open subset $\cU$ of $\rM$, considered as a map from $\R^m$ to $\cE$. Then $\Pi_{\sfg} = d\psi_{\cU}(d\psi_{\cU}^{\sfT}\sfg d\psi_{\cU})^{-1}d\psi_{\cU}^\sfT\sfg$ in $\cU$, hence $\Pi_{\sfg}$ is smooth in $\cU$, thus, on $\rM$.
  
  To construct $\Gamma$, let $\cN\subset\cE$ be a tubular neighborhood of $\rM$ (see \cite{LeeRiemann}), $\cN$ is an open subset in $\cE$ and there is a retraction $r$, a smooth map from $\cN$ to $\rM$, such that $r$ is the identity map on $\rM$. We extend $\sfg$ from $\rM$ to $\cN$, setting $\sfg(x) = \sfg(r(x))$ for $x\in \cN$. This gives $\cN$ a metric. Since $\cN$ is an open subset of $\cE$, the Levi-Civita connection on $\cN$ is defined via the classical Christoffel symbols, which could be written as a bilinear form
\begin{equation}\label{eq:GammaN}\Gamma_{\cN}(\omega_1, \omega_2)=\frac{1}{2}\sfg^{-1}(\rD_{\omega_1}\sfg\omega_2 + \rD_{\omega_2}\sfg\omega_1 - \cX_0(\omega_1, \omega_2))\end{equation}
  for $\omega_1, \omega_2\in\cE$, with $\cX_0(\omega_1, \omega_2)$ satisfies $
\langle\cX_0(\omega_1, \omega_2), \omega_3\rangle_{\cE} = \langle \omega_2,\rD_{\omega_3}\omega_1\rangle_{\cE}$ for $\omega_3\in \cE$. Since $\rM$ is a Riemannian submanifold of $\cN$ with this metric, its Levi-Civita connection for two vector fields $\ttX, \ttY$ on $\rM$ is given by $\Pi_{\sfg}(\rD_{\ttX}\ttY + \Gamma_{\cN}(\ttX, \ttY))$ (\cite{ONeil1983}, Lemma 4.3). As $\Pi_{\sfg} \ttY = \ttY$, $\Pi_{\sfg}(\rD_{\ttX}\ttY) = \rD_{\ttX}(\Pi_{\sfg} \ttY) - (\rD_{\ttX}\Pi_{\sfg})\ttY = \rD_{\ttX} \ttY - (\rD_{\ttX}\Pi_{\sfg})\ttY$, so
$$\nabla_{\ttX}\ttY = \rD_\ttX \ttY -(\rD_\ttX\Pi_{\sfg})\ttY + \Pi_{\sfg}\Gamma_{\cN}(\ttX, \ttY)$$
Hence, we can define $\Gamma(\xi, \omega)_x := -(\rD_{\xi}\Pi_{\sfg})_x\Pi_{\sfg, x}\omega + \Pi_{\sfg, x}\Gamma_{\cN}(\xi, \omega)_x$ for $\xi\in \cT_x\rM, \omega\in\cE$, then $\Gamma(\ttX, \ttY)_x = -(\rD_{\ttX}\Pi_{\sfg})_x\ttY_x + \Pi_{\sfg, x}\Gamma_{\cN}(\ttX, \ttY)_x$ satisfies \cref{eq:LvC}. We can extend $\Gamma$ to an operator from $\cE\times\cE$ to $\cE$ by defining $\Gamma(\omega_1, \omega_2)_x = \Gamma(\Pi_{\sfg}\omega_1, \omega_2)_x$. Thus we have proved the existence of a Christoffel function $\Gamma$, and it is clearly smooth.

It is clear the condition in \cref{eq:LvC} implies $\Gamma(\ttX, \ttY)_x$ is only dependent on the value of $\ttX_x, \ttY_x$ at $x$. Since $(\rD_{\ttX}\ttY)_x$ is not dependent on the metric, while $\nabla_{\ttX}\ttY$ is only dependent on the restriction of the metric to the tangent bundle, $\Gamma$ evaluated on tangent vectors is only dependent on the restriction of a metric operator on the tangent bundle.

Let $\cX(\xi, \eta)$ be a function satisfying \cref{eq:cX}. $(\Pi_{\sfg}\sfg^{-1}\cX(\xi, \eta))_x$ is a vector in $\cT_x\rM$ satisfying $\langle (\Pi_{\sfg}\sfg^{-1}\cX(\xi, \eta))_x, \sfg_x\phi \rangle_{\cE}= \langle(\rD_{\phi}\sfg)_x\xi, \eta\rangle_{\cE}$ for all $\phi\in \cT_x\rM$. Such vector is unique as $\sfg_x$ is non degenerated in $\cT_x\rM$, so $(\Pi_{\sfg}\sfg^{-1}\cX(\xi, \eta))_x = (\Pi_{\sfg}\sfg^{-1}\cX_0(\xi, \eta))_x$ with $\cX_0$ as in \cref{eq:GammaN}, so the right-hand side of \cref{eq:GammaForm} evaluate to the same value as that of $\Gamma$ constructed from the existence part when restricted to tangent vectors, this proves the last statement.
\end{proof}
This last statement allows us to work with a more convenient $\cX$ in some cases, as we only need \cref{eq:cX} to be valid on tangent vectors.
\begin{remark}
  We will use the {\it derivative of projection trick} converting $\Pi_1(\rD_{\xi}\Pi_2)\omega$ to $(\rD_{\xi}(\Pi_1\Pi_2))\omega - (\rD_{\xi}\Pi_1)\Pi_2\omega$ often in this paper, for two projections $\Pi_1$ and $\Pi_2$. It is used in the following lemma, providing standard commuting vector fields extension of tangent vectors useful in tensor calculations.
\end{remark}
\begin{lemma}Let $\xi, \eta$ be two tangent vectors to $\rM$ at $x\in\rM$, where $(\rM, \sfg, \cE)$ is an embedded ambient structure. Define vector fields $\pxi, \peta$ on $\rM$ by setting $\pxi(y) = \Pi_{\sfg, y}\xi$ and $\peta(y) = \Pi_{\sfg, y}\eta$ for $y\in \rM$. Then $\pxi(x) = \xi, \peta(x) = \eta$ and
\begin{equation}\label{eq:proj_deriv}\begin{gathered}\Pi_{\sfg, x}(\rD_{\xi}\peta)_x = 0\\
    [\pxi,\peta]_x = 0\\
    (\rD_{\xi}\Pi_{\sfg})_x\eta = (\rD_{\eta}\Pi_{\sfg})_x\xi
    \end{gathered}
  \end{equation}
\end{lemma}
\begin{proof}It is clear $\pxi(x) = \xi, \peta(x) = \eta$. $\peta$ could be considered as an $\cE$-valued function on $\rM$, so $(\rD_{\xi}\peta)_x=(\rD_{\xi}(\Pi_{\sfg}\eta))_x$ is defined and is equal to $(\rD_{\xi}\Pi_{\sfg})_x\eta$. Note that at $x$, $\Pi_{\sfg, x}(\rD_{\eta}\pxi)_x$ expands to
$$\Pi_{\sfg, x}(\rD_{\eta}\Pi_{\sfg})_x\xi = (\rD_{\eta}(\Pi_{\sfg}^2)_x)\xi - (\rD_{\eta}\Pi_{\sfg})_x(\Pi_{\sfg, x}\xi) =
(\rD_{\eta}\Pi_{\sfg})_x\xi - (\rD_{\eta}\Pi_{\sfg})_x\xi = 0$$
  and similarly $\Pi_{\sfg, x}(\rD_{\xi}\peta)_x=0$, thus $\Pi_{\sfg, x}[\pxi, \peta]_x =
\Pi_{\sfg, x}(\rD_{\xi}\peta)_x - \Pi_{\sfg, x}(\rD_{\eta}\pxi)_x= 0$. But $[\pxi, \peta]_x$ is tangent to $\rM$ at $x$, hence $[\pxi, \peta]_x=\Pi_{\sfg, x}[\pxi, \peta]_x =0$. This implies  $(\rD_{\xi}(\Pi_{\sfg}\eta))_x = (\rD_{\eta}(\Pi_{\sfg}\xi))_x$.
\end{proof}
\begin{remark}Note that when the metric operator $\sfg$ is constant, $\Gamma(\xi, \eta) = -\rD_{\xi}\Pi_{\sfg}\eta$ for tangent vectors $\xi, \eta$.

  Another way to look at \cref{eq:GammaForm} is for vector fields $\ttX$ and $\ttY$ on $\rM$ considered as functions from $\rM$ to $\cE$, $\Pi_{\sfg}\rD_{\ttX}\ttY$ is a connection, although in general not compatible with metric. The Levi-Civita connection is another connection, thus the difference gives rise to a $(1, 2)$-tensor $\mrGamma$ described below.  
\end{remark}  
\begin{proposition}Let $(\rM, \sfg, \cE)$ be an embedded tangent structure with a Christoffel function $\Gamma$. Then for all tangent vectors $\xi, \eta$ at $x\in \rM$, we have
  \begin{equation}\label{eq:mrGamma0}
    \mrGamma(\xi, \eta)_x := (\rD_{\xi}\Pi_{\sfg})_x\eta + \Gamma(\xi, \eta)_x = \Pi_{\sfg, x}\Gamma(\xi, \eta)_x \in \cT_x\rM\end{equation}
Also, $\Gamma(\xi, \eta)_x=\Gamma(\eta,\xi)_x$ and $\mrGamma(\xi, \eta)_x= (\nabla_{\pxi}\peta)_x$.
\end{proposition}
\begin{proof} This follows from \cref{prop:emb_exists} and \cref{eq:GammaForm}. $\Gamma(\xi, \eta)_x = \Gamma(\eta, \xi)_x$ follows from torsion-freeness of the Levi-Civita connection.
\end{proof}
In the following examples, for an ambient matrix space $\cE$ we always use the trace inner product, $\langle \omega, \omega\rangle_{\cE} = \Tr\omega\omega^{\sfT}$ for $\omega\in\cE$ as the base inner product.

\begin{example}\label{ex:son}
  For the special orthogonal group $\SOO(n)$ of matrices satisfying $UU^{\sfT} = \dI_n$ in $\cE=\R^{n\times n}$, with determinant $1$, define the metric $\langle\omega, \omega\rangle_{\sfg} = \frac{1}{2}\Tr(\omega \omega^{\sfT})$ for $\omega\in\cE$, and the metric operator $\sfg = \frac{1}{2}\dI_n$. The tangent space at $U\subset\SOO(n)$ consists of matrices $\eta \in \cE$ satisfying $U^{\sfT}\eta + \eta^{\sfT}U= 0$. Applying \cref{lem:projprop} with $\cE_{\rN} =\oo(n)$, the space of antisymmetric matrices, and $\rN\delta = U\delta$ for $\delta\in\oo(n)$, then $\rN^{\sfT}\omega = \frac{1}{2}(U^{\sfT}\omega - \omega^{\sfT}U)=\asym(U^{\sfT}\omega)$, where $\asym$ is the antisymmetrize operator. From here the projection is $\Pi\omega = \frac{1}{2}(\omega - U\omega^{\sfT}U)$ and the Christoffel function is just $-(\rD_{\xi}\Pi)\eta = \frac{1}{2}(\xi \eta^{\sfT}U + U\eta^{\sfT}\xi)$.
\end{example}
\begin{example}\label{ex:sasaki}Let $\rM=\cT S^{n-1}$ be the tangent bundle of the unit sphere $S^{n-1}$, consisting of pairs of $\R^n$-vectors $(x, v)$ with $x^{\sfT}x = 1, x^{\sfT}v = 0$. The ambient space is $\cE = \R^{2n}$. Consider the Sasaki metric \cite{Sasaki}, given by the operator
  \begin{equation}\sfg_{(x, v)}\begin{bmatrix} \omega_{\frm} \\ \omega_{\frt}\end{bmatrix} =\begin{bmatrix}\dI_n & vx^{\sfT} \\0 & \dI_n
    \end{bmatrix}
    \begin{bmatrix}\dI_n & 0 \\xv^{\sfT} & \dI_n
    \end{bmatrix}\begin{bmatrix} \omega_{\frm} \\ \omega_{\frt}\end{bmatrix} = \begin{bmatrix} \dI_n +vv^{\sfT} & vx^{\sfT}\\xv^{\sfT} & \dI_n
      \end{bmatrix}\begin{bmatrix} \omega_{\frm} \\ \omega_{\frt}\end{bmatrix}
  \end{equation}
  for $(\omega_{\frm}^{\sfT}, \omega_{\frt}^{\sfT})^{\sfT}\in \R^{2n}$. To avoid repeated use of the adjoint $\sfT$, we will write a vector in $\R^{2n}$ as $(\omega_{\frm}, \omega_{\frt})$ only, but will let the matrices operate on column vectors. Then, $\sfg^{-1}$ is given by the matrix $\sfg^{-1} =\begin{bmatrix} \dI_n & -vx^{\sfT}\\-xv^{\sfT} & \dI_n + xv^{\sfT}vx^{\sfT} \end{bmatrix}$. The tangent space of $\rM$ consists of pairs of vectors $(\Delta_{\frm}, \Delta_{\frt})$ satisfying the conditions $x^{\frt}\Delta_{\frm} = 0, \Delta_{\frm}^{\sfT} v + x^{\sfT}\Delta_{\frt} = 0$. The normal space at $(x, v)$ consists of vectors $(\delta_{\frm}, \delta_{\frt})$ satisfying $\langle\sfg(\delta_{\frm}, \delta_{\frt}), (\Delta_{\frm}, \Delta_{\frt})\rangle_{\R^{2n}} = 0$. The constraints show that $\langle(x, 0), (\Delta_{\frm}, \Delta_{\frt})\rangle_{\R^{2n}} = 0$, $\langle(v, x), (\Delta_{\frm}, \Delta_{\frt})\rangle_{\R^{2n}} = 0$. Thus, a normal vector will have the form $\rN(a_0, a_1)$, with $\rN$ is the linear map from $\cE_{\rN} = \R^2$ to the normal space at $(x, v)$ such that
  $$\sfg\rN\begin{bmatrix}a_0 \\ a_1\end{bmatrix} = \begin{bmatrix}x & v\\0& x \end{bmatrix}\begin{bmatrix} a_0\\ a_1\end{bmatrix}$$
so $\rN$ is given by the matrix $\begin{bmatrix} x &  0\\ 0 &x\end{bmatrix}$, since $v^{\sfT}x =0$. By direct computation, $\rN^{\sfT}\sfg\rN= \dI_2$, so the projection to the normal space is given by $\rN(\dI_2)^{-1}\rN^{\sfT}\sfg = \begin{bmatrix} xx^{\sfT} & 0 \\xv^{\sfT} & xx^{\sfT}\end{bmatrix}$. The projection to the tangent space is given by
\begin{equation} \Pi_{(x, v)} =  \begin{bmatrix}\dI_n- xx^{\sfT} & 0 \\-xv^{\sfT} & \dI_n-xx^{\sfT}
        \end{bmatrix}\end{equation}
For three tangent vectors $\xi = (\xi_{\frm}, \xi_{\frt}), \eta = (\eta_{\frm}, \eta_{\frt}), \phi=(\phi_{\frm}, \phi_{\frt})$, as $\rD_{\phi}\sfg\xi =((\phi_{\frt} v^{\sfT} + v\phi_{\frt}^{\sfT})\xi_{\frm} + (\phi_{\frt} x^{\sfT} + v\phi_{\frm}^{\sfT})\xi_{\frt}, (\phi_{\frm}v^{\sfT} + x\phi_{\frt}^{\sfT})\xi_{\frm} )$, we can take $\cX(\xi, \eta) = ( \xi_{\frt}\eta_{\frm}^{\sfT}v +
\eta_{\frt}\xi_{\frm}^{\sfT}v, \eta_{\frm}(\xi_{\frm}^{\sfT}v + \xi_{\frt}^{\sfT}x) +  \xi_{\frm}(\eta_{\frm}^{\sfT}v + \eta_{\frt}^{\sfT}x))$. We can compute the Christoffel function by \cref{eq:GammaForm}. As is well-known, it is easier to compute covariant derivatives on lifts of vector fields from $S^{n-1}$, we will see the general result in \cref{sec:nat_metric}.
\end{example}  

An interesting example is that of a Stiefel manifold $\St{p}{n}$, consisting of matrices $Y\in \R^{n\times p}$ ($n >p$ are two positive integers) with $Y^{\sfT}Y = \dI_p$. In this case, we can define a metric operator of the form $\sfg\omega =  \omega +(\alpha-1)YY^{\sfT}\omega$ for $\omega\in\R^{n\times p}$. This is a metric operator if $\alpha > 0$. In \cite{Nguyen2020a}, we derived the Levi-Civita connection and projection for this metric using the metric operator approach. This example is interesting, as when $\alpha < 1$, $\sfg$ is not positive-definite for some values of $Y\in\cE$, thus the usual Riemannian embedding approach cannot be used on the full $\cE$. The metric operator approach avoids this difficulty, see details in \cite{Nguyen2020a}. See also \cite{ExtCurveStiefel}, where the authors use a pseudo-Riemannian metric on $\SOO(n)\times \SOO(k)$, which induces a Riemannian metric on $\St{p}{n}$ for $\alpha > 0$. Finally, we can realize this metric as a quotient metric on $\SOO(n)/\SOO(n-p)$ with a left-invariant metric on $\SOO(n)$.
\subsection{Submersed ambient structure}
Recall (\cite{ONeil1983}, Definition 7.44) a Riemannian submersion $\qq:\rM\to\cB$ between two manifolds $\rM$ and $\cB$ is a smooth, onto map, such that the differential $d\qq$ is onto at every point $x\in\rM$, the fiber $\qq^{-1}(\qq x)$ is a Riemannian submanifold of $\rM$ containing $x$, and $d\qq$ preserves scalar products of vectors normal to fibers. In particular, quotient space by free and proper action of a group of isometries is a Riemannian submersion. In a Riemannian submersion, the tangent bundle $\cT\rM$ of $\rM$ has a decomposition $\cT\rM = \cV\rM\oplus\cH\rM$, where $\cV\rM$ is the vertical bundle, defined to be the tangent space of the submanifold $\qq^{-1}(\qq x)$ at each $x\in\rM$, and $\cH\rM$ is its orthogonal complement. By the submersion assumption, each tangent vector of $\cB$ at $\qq x\in \cB$ has a unique inverse image in $\cT_x\rM$ that is orthogonal to $\cV_x\rM$, called the horizontal lift. Denoted by $\ttH$ the projection from $\cE$ to $\cH\rM$, we will call it the {\it horizontal projection}. The {\it vertical projection} $\ttV$ from $\cE$ to $\cV\rM$ is defined similarly, and we have $\ttH + \ttV = \Pi_{\sfg}$. Tangent vectors in $\cH\rM$ and $\cV\rM$ are called {\it horizontal and vertical vectors} respectively. By the submersion assumption, we have the following lemma
\begin{lemma}Let $\qq:\rM\to\cB$ be a Riemannian submersion, where $(\rM, \sfg, \cE)$ is an embedded ambient structure with the projection $\Pi_{\sfg}$ and a Christoffel function $\GammaM$. If $v_{\cB}$ and $w_{\cB}$ are vector fields on $\cB$, which lift to horizontal vector fields $v_{\cH}$ and $w_{\cH}$ then the horizontal lift of $\nabla^{\cB}_{v_B}w_{\cB}$ is $\nabla^{\cH}_{v_{\cH}} w_{\cH} :=\ttH\nabla_{v_{\cH}} w_{\cH}$, or
\begin{equation}\label{eq:gamma_sub}
  \begin{gathered}
\ttH\nabla_{v_{\cH}} w_{\cH}=  \ttH(\rD_{v_{\cH}}w_{\cH} + \GammaM(v_{\cH}, w_{\cH})) = \rD_{v_{\cH}}w_{\cH} +\GammaH(v_{\cH}, w_{\cH})\\
\GammaH(v_{\cH}, w_{\cH}) := -(\rD_{v_{\cH}}\ttH)w_{\cH} + \ttH\GammaM(v_{\cH}, w_{\cH})  \end{gathered}
\end{equation}
We also have $\GammaH(v_{\cH}, w_{\cH})=  -(\rD_{v_{\cH}}\ttH)w_{\cH} + \ttH\mrGamma_{\rM}(v_{\cH}, w_{\cH})$.
\end{lemma}
\begin{proof} The lift of $\nabla^{\cB}_{v_B}w_{\cB}$ is $\ttH\nabla_{v_{\cH}} w_{\cH}$ follows from \cite{ONeil1983}, Lemma 7.45. The formula for $\GammaH$ follows from the facts $\ttH\Pi = \ttH$ and $\ttH w_{\cH} = w_{\cH}$, so
$$\ttH\rD_{v_{\cH}}w_{\cH} = \ttH\rD_{v_{\cH}}(\Pi_{\sfg} w_{\cH}) = \rD_{v_{\cH}}(\ttH\Pi_{\sfg} w_{\cH}) - (\rD_{v_{\cH}}\ttH)\Pi_{\sfg} w_{\cH} 
  $$
or $\rD_{v_{\cH}} w_{\cH} - (\rD_{v_{\cH}}\ttH) w_{\cH}$. The alternative formula for $\GammaH$ follows from
$$\begin{gathered}
\ttH(\rD_{v_{\cH}} \Pi_{\sfg})w_{\cH} = (\rD_{v_{\cH}}(\ttH \Pi_{\sfg}))w_{\cH} -(\rD_{v_{\cH}}\ttH)\Pi_{\sfg} w_{\cH} =0
  \end{gathered}$$
\end{proof}
\begin{remark}
  In general, $\GammaH(\xi, \eta) \neq \GammaH(\eta, \xi)$, because $(\rD_{\xi}\ttH){\eta} \neq (\rD_{\eta}\ttH)\xi$. The difference $(\rD_{\xi}\ttH){\eta} - (\rD_{\eta}\ttH)\xi$ is a vertical tangent vector (we can prove using the derivative of projection trick that $\Pi((\rD_{\xi}\ttH){\eta} - (\rD_{\eta}\ttH)\xi) = (\rD_{\xi}\ttH){\eta} - (\rD_{\eta}\ttH)\xi$  and $\ttH((\rD_{\xi}\ttH){\eta} - (\rD_{\eta}\ttH)\xi) =0$). It is $2\rA_{\xi}\eta$, where $\rA_{\xi}\eta$ is the O'Neil tensor, which will be defined in \cref{subsec:Rsubmerse}.

We will define a {\it submersed ambient structure} as a tuple $(\rM, \qq, \cB, \sfg, \cE)$ where $(\rM, \sfg, \cE)$ is an embedded ambient structure, and $\qq:\rM\to\cB$ is a Riemannian submersion. The results in this section apply to submersed ambient structures.
\end{remark}
Expressions for $\GammaH$ could be lengthy, and there may exist a simpler expression for $\GammaH$ that is only valid for horizontal vectors. To calculate curvatures, we will need  $\GammaH$  to relate to $\GammaM$ as in \cref{eq:gamma_sub}, and we cannot use an expression that is valid only for horizontal vectors.

\begin{example} If $\sum_{j=0}^q d_i = n$, with $d_i > 0$ is an integer partition of $n$, 
consider the action of $\SOD$ (the group of block-diagonal orthogonal matrices of size $n\times n$ with determinant $1$) acting by right multiplication on $\SOO(n)$. In this case, $\rM = \SOO(n)$ with the metric $\sfg=\frac{1}{2}\dI_n$ as in \cref{ex:son}, while the flag manifold is the quotient $\cB = \SOO(n)/\SOD$ under this action, with the case $q=1$ corresponding to a Grassmann manifold. Let $\frk$ denote the Lie algebra $\sod$ of the stabilizer group $\SOD$. It is a subalgebra of $\oo(n)$ consisting of block-diagonal anti-symmetric matrices. Its orthogonal complement is the subspace $\sbb$ of antisymmetric matrices with diagonal blocks (corresponding to $\frk$) equal to zero.

The vertical space at $U\in \SOO(n)$ consists of matrices $U\diag(b_0,\cdots b_q)= UD$ with $U\in\SOO(n)$ and $D=\diag(b_0,\cdots b_q)\in\frk$ denotes the block diagonal matrix with diagonal blocks $b_i\in \oo(d_i)\subset \R^{d_i\times d_i}$,  for $i\in \{0,\cdots q\}$. Let $\cE_{\rN} := \frk$, define $\rN:\cE_{\rN}\to\cE =\R^{n\times n}$ by setting $\rN(b_0,\cdots, b_q) = U\diag(b_0,\cdots b_q)$, let $K_i := \diag(0_{d_0\times d_0},\cdots, \dI_{d_i},\cdots, 0_{d_q\times d_q})$ then $\rN^{\sfT}\omega = \frac{1}{2}\sum_{i=0}^q K_i(U^{\sfT}\omega - \omega^{\sfT}U)K_i\in \frk$ (which is the operator symmetrizing $U^{\sfT}\omega$ then taking the diagonal part). By \cref{lem:projprop}, the projection $\ttV_U$ to the vertical space at $U$ is $\frac{1}{2}\sum_{i=0}^q UK_i(U^{\sfT}\omega - \omega^{\sfT}U)K_i$, hence the horizontal projection is
\begin{equation}
  \ttH_U\omega = \frac{1}{2}(\omega - U\omega^{\sfT}U - \sum_{i=0}^q UK_i(U^{\sfT}\omega - \omega^{\sfT}U)K_i)
\end{equation}  
since $\Pi\omega = \frac{1}{2}(\omega - U\omega^{\sfT}U)$ from \cref{subsec:subman}. A matrix $\eta$ is horizontal at $U$ if and only if $U^{\sfT}\eta$ is antisymmetric and $K_iU^{\sfT}\eta K_i = 0$ for $i\in \{0\cdots q\}$. It is more intuitive to translate this to the picture at the identity. For any matrix $X$ in $\R^{n\times n}$, let $X_{\frk}$ denote the block diagonal
matrix in $\R^{n\times n}$ with the same diagonal blocks as $X$, (corresponding to the structure of $\frk$), and $X_{\frb} = X - X_{\frk}$, for example when $q=2$
$$
  X = \begin{bmatrix} * & * & *\\
    * & * & *\\
    * & * & *
    \end{bmatrix}\;\;\; X_{\frk} = \begin{bmatrix} * & 0 & 0\\
    0 & * & 0\\
    0 & 0 & *
  \end{bmatrix}\;\;\;
  X_{\frb} = \begin{bmatrix} 0 & * & *\\
    * & 0 & *\\
    * & * & 0
    \end{bmatrix}
$$
The tangent space of $\SOO(n)$ at $\dI_n$ is identified with $\oo(n)$, the vertical projection of an antisymmetric matrix $Z$ at $\dI_n$ is just $Z_{\frk}$, and the horizontal projection is  $Z_{\frb}$. The projection $\ttH_U$ is $U\asym(U^{\sfT}\omega)_{\frb}$, where $\asym$ is the antisymmetrize operator, $\asym X = 1/2(X-X^{\sfT})$. This picture could be generalized to other homogeneous spaces, ($\frk$ and $\frb$ are written as $\mathfrak{h}$ and $\mathfrak{m}$ in \cite{KobNom}, for example. We rename the subspaces to avoid notational conflict in this paper).

We have $\mrGamma = 0$, thus, a horizontal Christoffel function $\GammaH$ is simply $-(\rD_{\xi}\ttH)\omega$
\begin{equation}\label{eq:flag_gamma_omg}
  \GammaH(\xi,\omega) = \frac{1}{2}( \xi\omega^{\sfT}U + U\omega^{\sfT}\xi+ \sum_{i=0}^q \xi K_i(U^{\sfT}\omega - \omega^{\sfT}U)K_i) + UK_i(\xi^{\sfT}\omega - \omega^{\sfT}\xi)K_i)
\end{equation}  
If $\omega = \eta$ is a horizontal vector then $\eta^{\sfT}U = -U^{\sfT}\eta$ and $K_i U^{\sfT}\eta K_i = 0$, thus
\begin{equation}\label{eq:Axe}
  \rD_{\xi}\ttH\eta - \rD_{\eta}\ttH\xi = -U\sum_{i=0}^q  K_i (\xi^{\sfT}\eta - \eta^{\sfT}\xi) K_i = -U(\xi^{\sfT}\eta - \eta^{\sfT}\xi)_{\frk}
\end{equation}
Translating to the identity, \cref{eq:Axe} is $U[U^{\sfT}\xi, U^{\sfT}\eta]_{\frk}$ and 
for $A, B\in \R^{n\times n}$,
\begin{equation}\label{eq:UAB}\GammaH(UA, UB) = \frac{1}{2}U(AB^{\sfT} +B^{\sfT}A + A(B-B^{\sfT})_{\frk} + (A^{\sfT}B - B^{\sfT}A)_{\frk})
\end{equation}
Denote by $\kappa_A$ the invariant vector field $X\mapsto XA$ for $X\in\SOO(n), A\in \oo(n)$ with $A_{\frk}=0$, then for $A, B\in\oo(n)$ with $A_{\frk} = B_{\frk} = 0$, $(\nabla_{\kappa_A} \kappa_B)_U = (\rD_{\kappa_A}\kappa_B)_U + \GammaH(UA, UB) = UAB -\frac{1}{2}U(AB+BA+[A,B]_{\frk}) = \frac{1}{2}(\kappa_{[A,B]_{\frb}})_U$, (\cite{KobNom}, theo. II.10.3.3).
\end{example}
\section{Curvature formulas}
\subsection{Curvature formula for an embedded manifold}
As explained, Christoffel functions are not unique for an ambient structure $(\rM, \sfg, \cE)$, but any choice of Christoffel function could be used to evaluate the curvature tensor:
\begin{theorem}\label{theo:cur_embed} Let $\rM$ be a submanifold of $\cE$, with a metric operator $\sfg: \rM\to\fL(\cE,\cE)$. Let $\Gamma:\rM\to \fL(\cE\otimes \cE,\cE)$ be a Christoffel function of $\sfg$. Then the Riemannian curvature of $\rM$ is given by one of the following:
    \begin{equation}\label{eq:rc1}
      \RcM_{\xi,\eta}\phi = -(\rD_{\xi}\Gamma)(\eta, \phi) + (\rD_{\eta}\Gamma)(\xi, \phi)-\Gamma(\xi, \Gamma(\eta, \phi)) +\Gamma(\eta, \Gamma(\xi, \phi))
    \end{equation}
    \begin{equation}\label{eq:rc1a}
      \RcM_{\xi,\eta}\phi = -(\rD_{\xi}\Gamma)(\eta, \phi) + (\rD_{\eta}\Gamma)(\xi, \phi)-\Gamma(\Gamma(\phi, \eta), \xi)) +\Gamma(\Gamma(\phi, \xi), \eta)
    \end{equation}    
where $\xi, \eta, \phi$ are three tangent vectors to $\rM$ at $x\in \rM$, $(\rD_{\xi}\Gamma)(\eta, \phi), (\rD_{\eta}\Gamma)(\xi, \phi)$ denote the directional derivatives, and all expressions are evaluated at $x$. With $\mathring{\Gamma}(\xi, \phi) =  (\rD_{\xi}\Pi) \phi +\Gamma(\xi, \phi)$ as in \cref{eq:mrGamma0}, then 
    \begin{equation}\label{eq:rc2}
      \RcM_{\xi,\eta}\phi = -(\rD_{\xi}\mrGamma)(\eta, \phi) +\
      (\rD_{\eta}\mrGamma)(\xi, \phi) -\Gamma(\xi,\Gamma(\eta, \phi)) +
        \Gamma(\eta,\Gamma(\xi, \phi))
    \end{equation}
\end{theorem}
We use \cref{eq:rc2} when the directional derivative of $\mrGamma$ is easier to compute than that of $\Gamma$, for example, if $\sfg$ is constant then $\mrGamma=0$. Note that to use \cref{eq:rc2}, derivatives of $\mrGamma$ have to be computed with $\mrGamma$ satisfying \cref{eq:mrGamma0}, if we simplify $\mrGamma$, we must make sure $\Gamma$ is simplified accordingly, and vice-versa.
\begin{proof} Let $\xi, \eta, \phi$ be three tangent vectors at $x\in\rM$, identified with elements of $\cE$. Define three vector fields on $\rM$ by setting $\pxi(y) = \Pi(y)\xi, \peta(y)= \Pi(y)\eta, \pphi(y) = \Pi(y)\phi$ for $y\in\rM$. By \cref{eq:proj_deriv}, $[\pxi, \peta]_x = 0$ and
$$\begin{gathered}
    (\RcM_{\xi\eta }\phi)_x = (\nabla_{[\pxi, \peta]} \pphi - \nabla_{\pxi} \nabla_{\peta} \pphi + \nabla_{\peta} \nabla_{\pxi} \pphi)_x\\
    =-\{\rD_{\pxi}(\rD_{\peta}\pphi +\Gamma(\peta, \pphi))\}_{x} -\Gamma(\pxi, (\rD_{\peta}\pphi +\Gamma(\peta, \pphi)))_{x} +\\
    \{\rD_{\peta}(\rD_{\pxi}\pphi +\Gamma(\pxi, \pphi))\}_{x} + \Gamma(\peta, (\rD_{\pxi}\pphi +\Gamma(\pxi, \pphi)))_{x}\end{gathered}
$$
We have $(-\rD_{\pxi}\rD_{\peta}\pphi + \rD_{\peta}\rD_{\pxi}\pphi)_{x} = (\rD_{[\pxi, \peta]}\pphi)_{x} = 0$, thus
\begin{equation}\label{eq:rcm_alt}
  \begin{gathered}(\RcM_{\xi\eta }\phi)_x = -\{\rD_{\pxi}(\Gamma(\peta, \pphi))\}_{x} -\Gamma(\pxi, (\rD_{\peta}\pphi +\Gamma(\peta, \pphi)))_{x} +\\
\{\rD_{\peta}(\Gamma(\pxi, \pphi))\}_{x} + \Gamma(\peta, (\rD_{\pxi}\pphi +\Gamma(\pxi, \pphi)))_{x}\end{gathered}\end{equation}
As $\Gamma$ is linear in the $\cE$ variables, we have
$$\rD_{\pxi}(\Gamma(\peta, \pphi)) = (\rD_{\pxi}\Gamma)(\peta, \pphi) + \Gamma(\rD_{\pxi}\peta, \pphi) + \Gamma(\peta, \rD_{\pxi}\pphi)
$$
Substitute the above and the similar combination for $(\peta, \pxi, \pphi)$ to \cref{eq:rcm_alt}, the terms $-\Gamma(\rD_{\pxi}\peta, \pphi)_x$ and $\Gamma(\rD_{\peta}\pxi, \pphi)_x$ from the similar combination cancels as $(\rD_{\pxi}\peta)_x = (\rD_{\peta}\pxi)_x$ by \cref{eq:proj_deriv}. Thus
$$\begin{gathered}(\RcM_{\xi\eta }\phi)_x = -(\rD_{\pxi}\Gamma)(\peta, \pphi)_x - \Gamma(\peta, \rD_{\pxi}\pphi)_x -\Gamma(\pxi, (\rD_{\peta}\pphi +\Gamma(\peta, \pphi)))_{x}\\
+(\rD_{\peta}\Gamma)(\pxi, \pphi)_x + \Gamma(\pxi, \rD_{\peta}\pphi)_x
  + \Gamma(\peta, (\rD_{\pxi}\pphi +\Gamma(\pxi, \pphi)))_{x}=\\
 -(\rD_{\pxi}\Gamma)(\peta, \pphi)_x  -\Gamma(\pxi, \Gamma(\peta, \pphi))_{x} +
(\rD_{\peta}\Gamma)(\pxi, \pphi)_x 
  + \Gamma(\peta, \Gamma(\pxi, \pphi))_{x}
\end{gathered}
$$
which gives us \cref{eq:rc1}. In \cref{eq:rcm_alt},
$\Gamma(\pxi, (\rD_{\peta}\pphi +\Gamma(\peta, \pphi)))_{x} = \Gamma((\rD_{\peta}\pphi +\Gamma(\peta, \pphi)), \pxi)_{x}$ because each variable is a tangent vector, similarly
$\Gamma(\peta, (\rD_{\pxi}\pphi +\Gamma(\pxi, \pphi)))_{x} = \Gamma((\rD_{\pxi}\pphi +\Gamma(\pxi, \pphi)), \peta)_{x}$, and in the preceding calculation, the affected terms are 
$-\Gamma((\rD_{\peta}\pphi +\Gamma(\peta, \pphi)), \pxi)_{x} + \Gamma((\rD_{\pxi}\pphi +\Gamma(\pxi, \pphi)), \peta)_{x}$, or $-\Gamma(\Gamma(\peta, \pphi), \pxi)_{x} + \Gamma(\Gamma(\pxi, \pphi), \peta)_{x}$ which gives us \cref{eq:rc1a}.

From $(\rD_{\xi}\rD_{\eta}\Pi)\phi - (\rD_{\eta}\rD_{\xi}\Pi)\phi =0$, we deduce \cref{eq:rc2}.
\end{proof}
We now give a version of the Gauss-Codazzi equation for a metric operator.
\begin{proposition}\label{prop:gauss_cod} Let $(\rM, \sfg, \cE)$ be an embedded ambient structure and $\Pi=\Pi_{\sfg}$ the associated projection. Let $(\GammaE)_x$ be a bilinear map from $\cT_x\rM\times\cE$ to $\cE$ satisfying $\GammaE(\xi, \eta)_x = \GammaE(\eta, \xi)$ for all pair of tangent vectors $\xi, \eta$, and for all $\omega\in\cE$
  \begin{equation}\label{eq:metriccomp}
  \langle \eta, (\rD_{\xi}\sfg)_x \omega\rangle_{\cE} =\langle \GammaE(\xi, \eta)_x, \sfg_x\omega\rangle_{\cE} + \langle \eta, \sfg_x\GammaE(\xi, \omega)_x\rangle_{\cE}
\end{equation}
  and $\GammaE$ is smooth as a function from $\rM$ to $\cL(\cE\otimes\cE, \cE)$. Let $\ttX, \ttY$ be two vector fields on $\rM$ and $\tts$ be a $\cE$-valued function on $\rM$. Define a connection on $\rM\times \cE$ by $\nabla^{\cE}_{\ttX} \tts =   \rD_{\ttX} \tts + \GammaE(\ttX, \tts)$. Then $\nabla^{\cE}$ satisfies $\nabla^{\cE}_{\ttX} \ttY - \nabla^{\cE}_{\ttY}\ttX = [\ttX, \ttY]$ and
  \begin{equation}\label{eq:metriccomp2}
    \rD_{\ttX}\langle\ttY, \sfg \tts \rangle_{\cE} = \langle\nabla^{\cE}_{\ttX}\ttY, \sfg \tts\rangle_{\cE} + \langle \ttY, \sfg \nabla^{\cE}_{\ttX}\tts\rangle_{\cE}
  \end{equation}
Also, $\Gamma(\xi, \omega) := -(\rD_{\xi}\Pi)\omega + \Pi\GammaE(\xi, \omega)$ could be extended to a Christoffel function for $\sfg$, thus, for the Levi-Civita connection $\nabla$ of the induced metric, we have
  \begin{equation}\label{eq:proj_compab}
    \nabla_{\ttX}\ttY = \rD_{\ttX}\ttY  -(\rD_{\ttX}\Pi)\ttY + \Pi\GammaE(\ttX, \ttY)
\end{equation}    
  Let $\xi, \eta, \phi$ be tangent vectors on $\rM$ at $x$. Define the second fundamental form
  \begin{equation}\Two(\xi, \eta) = \GammaE(\xi, \eta) - \Gamma(\xi, \eta) = (\rD_{\xi}\Pi)\eta+(\dI_{\cE}-\Pi)\GammaE(\xi, \eta)\end{equation}
where $\dI_{\cE}$ denotes the identity map of $\cE$. Then
$\Two(\xi, \eta) =\Two(\eta, \xi)$ and $\Pi\Two(\xi, \eta) = 0$. Consider $\Two_{\xi}:\eta\mapsto \Two(\xi, \eta)$ as a linear map from $\cT_x\rM$ to $\cE$. Then its adjoint $\Two_{\xi}^{\dagger}$ as a map from $\cE$ to $\cT_x\rM$ under the inner products induced by $\sfg$ is given by:
  \begin{equation}\label{eq:Twodual}
\Two_{\xi}^{\dagger}\omega = \Two^{\dagger}(\xi, \omega)= (\rD_{\xi}\Pi)(\dI_{\cE}-\Pi) \omega - \Pi\GammaE(\xi,  (\dI_{\cE}-\Pi)\omega) =-\Gamma(\xi, (\dI_{\cE}-\Pi)\omega)
  \end{equation}
for $\omega\in\cE$. Define:
\begin{equation}
  \RcE_{\xi, \eta}\phi = -(\rD_{\xi}\GammaE)(\eta, \phi) + (\rD_{\eta}\GammaE)(\xi, \phi) - \GammaE(\xi, \GammaE(\eta, \phi)) + \GammaE(\eta, \GammaE(\xi, \phi)) 
\end{equation}
Then the Gauss-Codazzi equation holds:
\begin{equation}\label{eq:gaussco31}
    \RcM_{\xi \eta}\phi = \Pi\RcE_{\xi, \eta}\phi + \Two^{\dagger}(\eta,\Two(\xi, \phi)) - \Two^{\dagger}(\xi, \Two(\eta, \phi))
\end{equation}
\end{proposition}
\begin{proof}Equation \ref{eq:metriccomp2} is equivalent to
  $$\langle\ttY, (\rD_{\ttX}\sfg) \tts\rangle_{\cE} = \langle\GammaE(\ttX, \ttY), \sfg\tts \rangle_{\cE} + \langle\ttY, \sfg\GammaE(\ttX, \tts) \rangle_{\cE}
  $$
  which follows from  \ref{eq:metriccomp}. $\nabla^{\cE}_{\ttX}\ttY -\nabla^{\cE}_{\ttY}\ttX = [\ttX, \ttY]$ follows from the expression of $\nabla^{\cE}$ and the symmetry of $\Gamma^{\cE}$ when evaluated on vector fields. The right-hand side of \cref{eq:proj_compab} is a vector field, since
  $$\Pi(\rD_{\ttX}\ttY  -(\rD_{\ttX}\Pi)\ttY + \Pi\GammaE(\ttX, \ttY)) = \Pi\rD_{\ttX}\ttY  + \Pi\GammaE(\ttX, \ttY) =\rD_{\ttX}\ttY-
  (\rD_{\ttX}\Pi)\ttY + \Pi\GammaE(\ttX, \ttY) $$
  by the derivative of projection trick, thus we have a connection, and the right-hand side could be written as $\Pi\nabla^{\cE}_{\ttX}\ttY$. Compatibility with metric for $\nabla$ follows from compatibility of $\nabla^{\cE}$, as for three vector fields $\ttX, \ttY, \ttZ$
  $$\langle \Pi\nabla^{\cE}_{\ttX}\ttY, \sfg\ttZ\rangle_{\cE} +
  \langle \ttY, \sfg \Pi\nabla^{\cE}_{\ttX}\ttZ\rangle_{\cE}
=\langle \nabla^{\cE}_{\ttX}\ttY, \sfg\ttZ\rangle_{\cE} +
  \langle \ttY, \sfg \nabla^{\cE}_{\ttX}\ttZ\rangle_{\cE} = \rD_{\ttX}\langle \ttY, \sfg \ttZ\rangle_{\cE}
  $$
For an ambient vector $\omega$, $\langle (\dI_{\cE} - \Pi)\omega, \sfg \ttY\rangle_{\cE} = 0$ hence
  $$\begin{gathered}0 = \rD_{\ttX}\langle (\dI_{\cE} - \Pi)\omega, \sfg \ttY\rangle_{\cE} =
  \langle \nabla^{\cE}_{\ttX}(\dI_{\cE} - \Pi)\omega, \sfg \ttY\rangle_{\cE} + \langle (\dI_{\cE} - \Pi)\omega, \sfg  \nabla^{\cE}_{\ttX}\ttY\rangle_{\cE}=\\
  \langle-(\rD_{\ttX}\Pi)\omega +\GammaE((\dI_{\cE}-\Pi)\omega, \sfg\ttY)\rangle_{\cE} + 
  \langle \omega, \sfg (\dI_{\cE} - \Pi) \nabla^{\cE}_{\ttX}\ttY\rangle_{\cE}
\end{gathered}$$  
Note the last term is $\langle \omega, \sfg \Two(\ttX, \ttY)\rangle_{\cE}$, thus
$$\Two^{\dagger}(\ttX, \omega) = -\Pi(-(\rD_{\ttX}\Pi)\omega +\GammaE(\ttX, (\dI_{\cE}-\Pi)\omega)) = -\Gamma(\ttX, (\dI_{\cE}-\Pi)\omega)$$
which gives us \cref{eq:Twodual}. We note if $\nu$ is normal, $\Pi\nu = 0$ then $\Two^{\dagger}(\xi, \nu) = -\Gamma(\xi,\nu)$. Expanding the right-hand side of \cref{eq:gaussco31}
  $$\begin{gathered}\Pi(-\rD_{\xi}(\GammaE(\eta, \phi)) + \rD_{\eta}(\GammaE(\xi, \phi)) - \GammaE(\xi, \GammaE(\eta, \phi)) + \GammaE(\eta, \GammaE(\xi, \phi))) -\\
    \Gamma(\eta, \GammaE(\xi, \phi) - \Gamma(\xi, \phi)) +\Gamma(\xi, \GammaE(\eta, \phi)-\Gamma(\eta, \phi))\end{gathered}$$
Expand the first line, using $\Pi\rD_{\xi}(\GammaE(\eta, \phi)) = \rD_{\xi}(\Pi\GammaE(\eta, \phi)) - (\rD_{\xi}\Pi)\GammaE(\eta, \phi)$ and 
$\Pi\GammaE(\xi, \omega) = \Gamma(\xi, \omega) + (\rD_{\xi}\Pi)\omega$
then permute the roles of $\xi$ and $\eta$
$$\begin{gathered} -\rD_{\xi}(\Pi\GammaE(\eta, \phi)) + (\rD_{\xi}\Pi)\GammaE(\eta, \phi)
+\rD_{\eta}(\Pi\GammaE(\xi, \phi)) - (\rD_{\eta}\Pi)\GammaE(\xi, \phi)
\\- \Gamma(\xi, \GammaE(\eta, \phi))-(\rD_{\xi}\Pi) \GammaE(\eta, \phi))
+ \Gamma(\eta, \GammaE(\xi, \phi)) +(\rD_{\eta}\Pi) \GammaE(\xi, \phi))
-\\\Gamma(\eta, \GammaE(\xi, \phi) - \Gamma(\xi, \phi)) +\Gamma(\xi, \GammaE(\eta, \phi)-\Gamma(\eta, \phi))=\\
 -\rD_{\xi}(\Gamma(\eta, \phi)) + \rD_{\eta}(\Gamma(\xi, \phi)) -\Gamma(\xi,\Gamma(\eta, \phi)) +
\Gamma(\eta,\Gamma(\xi, \phi))
\end{gathered}$$
which is $\RcM_{\xi\eta}\phi$. We have used $-\rD_{\xi}(\Pi\GammaE(\eta, \phi))  +\rD_{\eta}(\Pi\GammaE(\xi, \phi)) =  -\rD_{\xi}(\Gamma(\eta, \phi)) -\rD_{\xi}(\rD_\eta\Pi\phi)+ \rD_{\eta}(\Gamma(\xi, \phi)) +\rD_{\eta}(\rD_{\xi}\Pi\phi)= -\rD_{\xi}(\Gamma(\eta, \phi)) + \rD_{\eta}(\Gamma(\xi, \phi))$.
\end{proof}
The equation is often given in the $(0, 4)$ form:
\begin{equation}
  \langle\RcM_{\xi \eta}\phi, \sfg\zeta\rangle_{\cE} = \langle\RcE_{\xi\eta}\phi, \sfg\zeta\rangle_{\cE} + \langle\Two(\xi, \phi), \sfg\Two(\eta, \zeta)\rangle_{\cE} - \langle\Two(\xi, \zeta),\sfg\Two(\eta, \phi)\rangle_{\cE}
\end{equation}
$\GammaE$ could be constructed by extending the metric operator $\sfg$ to a region near $\rM$ in $\cE$ then applying the usual Christoffel formula. $\RcE_{\xi\eta}\phi$ could also be calculated as $(\nabla_{[\ttX, \ttY]}^{\cE}\ttZ - \nabla_{\ttX}^{\cE}\nabla_{\ttY}^{\cE}\ttZ + \nabla_{\ttY}^{\cE}\nabla_{\ttX}^{\cE}\ttZ)_x$ for three vector fields $\ttX, \ttY, \ttZ$ such that $\ttX_x = \xi, \ttY_x = \eta, \ttZ_x = \phi$, and as before we can choose the vector fields to be $\pxi, \peta, p_{\phi}$.

For a Riemannian embedding in $\cE$, $\GammaE$ is zero and only the derivative of the projection needs to be evaluated. Otherwise, $\GammaE$ could be more complicated than $\Gamma$. As mentioned, the relationship between the Gauss-Codazzi equation and metric connections on subbundles is discussed in \cite{taylor2011partial}, appendix C. Results of \cref{subsec:Rsubmerse} below could also be considered from this point of view.

\begin{example}
Continue with $\rM=\SOO(n)$, consider $U\in\SOO(n)$ and let $\xi=UA, \eta=UB, \phi=UC$ be three tangent vectors at $U$, with $A, B, C\in\oo(n)$, and $\Gamma$ from \cref{ex:son}
$$\begin{gathered}\Gamma(\xi, \Gamma(\eta, \phi)) = \frac{1}{4}\{\xi(\eta\phi^{\sfT}U+U\phi^{\sfT}\eta)^{\sfT} U + U(\eta\phi^{\sfT}U+U\phi^{\sfT}\eta)^{\sfT}\xi\}=\\
  \frac{1}{4}\{\xi U^{\sfT}\phi\eta^{\sfT}U + \xi\eta^{\sfT}\phi + \phi\eta^{\sfT}\xi + U\eta^{\sfT}\phi U^{\sfT}\xi\}\\
  = \frac{1}{4}\{UA C(-B) + UA(-B)C + UC(-B)A + U(-B)CA\}
  \end{gathered}
  $$
Using \cref{eq:rc2}, the curvature is $\frac{1}{4}U[[A, B], C]$, as is well-known.
\end{example}
\subsection{Curvature formulas for a Riemannian submersion}\label{subsec:Rsubmerse}
The following lemma expresses the O'Neil tensor in \cite{ONeil1966} in terms of projections and Christoffel functions. We mostly follow the original paper. In a curvature calculation, we need to evaluate the Christoffel function on ambient, or not necessarily horizontal vectors, so the expression for the Christoffel function has to be valid on the whole tangent space, we cannot use simplified formulas that are valid for horizontal vectors only.
\begin{lemma}\label{lem:subdual}
  Let $\rM$ be a Riemannian manifold and $\cH\oplus\cV\subset \cT\rM$ be two orthogonal subbundles of $\cT\rM$. Let $\ttH$ and $\ttV$ be the projection operators from $\cT\rM$ to $\cH$ and $\cV$, respectively. Let $\nabla$ be the Levi-Civita connection of a Riemannian metric $\langle\rangle_{R}$, and $c_1, c_2$ be two vector fields. Then $\ttV\nabla_{c_1}(\ttH c_2)$ is a $(1, 2)$ tensor. For fixed tangent vectors $\xi, \eta$
  at $x\in\rM$ and two vector fields $c_1, c_2$ such that $c_1(x)=\xi, c_2(x)=\eta$, 
  the map $\rA_{\xi}:\eta\mapsto (\ttV\nabla_{c_1}(\ttH c_2))_{x}$ maps $\cT_x\rM$ to $\cT_x\rM$ for $x\in\rM$, and induce an operator on vector fields also denoted by $\rA_{c_1}$.
  Its adjoint $\rAd_{c_1}$ in Riemannian inner product is given by $-\ttH\nabla_{c_1}\ttV$, that is:
  \begin{equation}\label{eq:orthadj}
    \langle \ttH\nabla_{c_1}(\ttV c_3), c_2 \rangle_{R} = - \langle  c_3, \ttV\nabla_{c_1}(\ttH c_2) \rangle_{R}
  \end{equation}
  for all vector fields $c_2, c_3$. Further, if $(\rM, \sfg, \cE)$ is an embedded ambient structure inducing the Riemannian metric $\langle\rangle_R$, and $\Gamma$ is a Christoffel function, with $\mathring{\Gamma}(c_1, c_2 ) = \Gamma(c_1, c_2) + (\rD_{c_1}\Pi)c_2$. Then we have:
  \begin{equation}\label{eq:rAV}
    \rA_{c_1} c_2 = -(\rD_{c_1}\ttV)\ttH c_2 + \ttV\mrGamma(c_1, \ttH c_2)
  \end{equation}
   \begin{equation}
    \rA^{\dagger}_{c_1} c_3 = (\rD_{c_1}\ttH)\ttV c_3 - \ttH\mrGamma(c_1, \ttV c_3)
   \end{equation}
   For a subbundle $\cF$ of $\cT\rM$, let $\Pi_{\cF}$ be the projection to $\cF$. Set
   \begin{equation}\label{eq:gammaF}
     \Gamma_{\cF}(\xi, \omega) = -(\rD_{\xi}\Pi_{\cF})\omega + \Pi_{\cF}\mrGamma(\xi, \omega)
   \end{equation}
   In this notation, $\rA_{c_1}c_2 = \Gamma_{\cV}(c_1, \ttH c_2)$ and $\rA^{\dagger}_{c_1} c_2 = -\Gamma_{\cH}(c_1, \ttV c_2)$, or $\rA_{\xi}\eta = \Gamma_{\cV}(\xi, \ttH \eta)$ and $\rA^{\dagger}_{\xi} \eta = -\Gamma_{\cH}(\xi, \ttV \eta)$ for two tangent vectors $\xi, \eta$ at $x$.

If $\cH = \cH\rM$ is the horizontal bundle of a Riemannian submersion $\qq:\rM\to\cB$ and $c_1$ and $c_2$ are two horizontal vector fields with $c_1(x) = \xi, c_2(x)= \eta$ then
   \begin{equation}\label{eq:oneilLie}
     2\rA_{\xi}\eta = (\ttV[c_1, c_2])_x = (\rD_{\xi}\ttH)_x\eta - (\rD_{\eta}\ttH)_x\xi
     \end{equation}
\end{lemma}  
\begin{proof} We have
  $$\ttV\nabla_{c_1}(\ttH (fc_2)) = (\rD_{c_1}f) \ttV \ttH c_2 + f\ttV\nabla_{c_1}(\ttH c_2) = f\ttV\nabla_{c_1}(\ttH c_2)$$
  So the map $(c_1, c_2)\to \ttV\nabla_{c_1}(\ttH c_2)$ is a tensor. From compatibility with metric: $$0 = \rD_{c_1}\langle\ttH c_3, \ttV c_2 \rangle_R= \langle \nabla_{c_1}\ttH c_3, \ttV c_2 \rangle_R + \langle \ttH c_3, \nabla_{c_1}\ttV c_2 \rangle_R$$
  where $\langle\ttH c_2, \ttV c_3\rangle_R=0$ by orthogonality. This implies \cref{eq:orthadj}. When the Riemannian metric on $\rM$ is induced by a metric operator $\sfg$
  $$\ttV\nabla_{c_1}(\ttH c_2) = \ttV(\rD_{c_1}(\ttH c_2) -(\rD_{c_1}\Pi)(\ttH c_2) + \mrGamma(c_1, \ttH c_2))$$
  Expand the first two terms using the derivative of projection trick
  $$\begin{gathered}\ttV\rD_{c_1}(\ttH c_2) = \rD_{c_1}(\ttV\ttH c_2) - (\rD_{c_1}\ttV)\ttH c_2 = - (\rD_{c_1}\ttV)\ttH c_2 \\
    \ttV(\rD_{c_1}\Pi)(\ttH c_2) = (\rD_{c_1}\ttV ) \ttH c_2 - (\rD_{c_1}\ttV)\Pi\ttH c_2 =  0
  \end{gathered}$$
Hence $\ttV\nabla_{c_1}\ttH c_2 =  - (\rD_{c_1}\ttV)\ttH c_2 +\ttV\mrGamma(c_1, \ttH c_2)
  $, and we can switch the role of $\ttV$ and $\ttH$ for the dual formula.

For \cref{eq:oneilLie}, the first part follows from Lemma 2 of \cite{ONeil1966}, while the last equality follows by defining $c_1(y) = \ttH_y\xi$ and $c_2(y) = \ttH_y\eta$, in this case $[c_1, c_2]_x$ is vertical by the derivative of projection trick.
\end{proof}
To calculate the curvature in the following theorem, we repeat that $\GammaH$ must be evaluated from \cref{eq:gammaF} for ambient vectors, if $\GammaH$ is only valid for horizontal vectors the calculation may not be valid.
\begin{theorem}\label{theo:rsub} Assume $(\rM, \sfg, \cE)$ is an embedded ambient structure and there is a bundle decomposition $\cT\rM = \cH\rM\oplus\cV\rM$. For $x\in\rM$, if $\xi, \eta, \phi\in\cH_x\rM$ are three horizontal vectors at $x$), let $\rA_{\xi}$ be the operator defined by $\rA
  _{\xi}\omega=\ttV\nabla_{\xi}(\ttH\omega) = \GammaV(\xi, \ttH\omega)$ for $\omega\in\cE$ as in \cref{lem:subdual}, and $\rAd_{\xi}$ be its adjoint (thus if $\omega \in \cT_x \rM$ then $\rAd_{\xi}\omega=-\ttH\nabla_{\xi}\ttV\omega = -\GammaH(\xi, \ttV\omega)$), with all expressions evaluated at $x$. Set
    \begin{equation}\label{eq:cursubmer}
    \RcH_{\xi\eta}\phi := 2\rAd_{\phi}\rA_{\xi}\eta  -(\rD_{\xi}\GammaH)(\eta, \phi) + (\rD_{\eta}\GammaH)(\xi, \phi)-\GammaH(\xi, \GammaH(\eta, \phi)) +
    \GammaH(\eta, \GammaH(\xi, \phi))
    \end{equation}
Then $(\RcH_{\xi\eta}\phi)_x$ is in $\cH_x$. $\RcH$ satisfies the O'Neil's equations:
      \begin{equation}\label{eq:ONeil13}
      \RcH_{\xi\eta}\phi = \ttH \RcM_{\xi\eta}\phi + 2 \rAd_{\phi} \rA_{\xi}\eta - \rAd_{\xi} \rA_{\eta}\phi +  \rAd_{\eta} \rA_{\xi}\phi
      \end{equation}
where $\RcM$ is the curvature tensor of $\rM$. Thus, if $\cH=\cH\rM$ is the horizontal bundle from a Riemannian submersion $\qq:\rM\to\cB$, $\RcH$ given by \cref{eq:cursubmer} is the horizontal lift of the Riemannian curvature tensor on $\cB$.
Alternatively,
\begin{equation}\label{eq:cursubmer2}
  \begin{gathered}
    \RcH_{\xi\eta}\phi = 2\rAd_{\phi}\rA_{\xi}\eta
-(\rD_{\xi}(\ttH\mrGamma(\eta, \phi)) +\
      (\rD_{\eta}(\ttH\mrGamma(\xi, \phi))\\
-\GammaH(\xi, \GammaH(\eta, \phi)) +
\GammaH(\eta, \GammaH(\xi, \phi))
\end{gathered}
  \end{equation}
\end{theorem}
Equation \ref{eq:ONeil13} is the $(1, 3)$ form of the classical O'Neil's equation in \cite{ONeil1966}.
\begin{proof} We expand the right-hand side of \cref{eq:ONeil13}, every expression will be evaluated at $x\in\rM$
  \begin{equation}\label{eq:tmp1}
    \begin{gathered}\ttH(-(\rD_{\xi}(\mrGamma(\eta, \phi)) +
      \rD_{\eta}(\mrGamma(\xi, \phi)) -\Gamma(\xi,\Gamma(\eta, \phi)) +
    \Gamma(\eta,\Gamma(\xi, \phi)) ) -\\
    2\GammaH(\phi, \GammaV(\xi, \eta)) +\GammaH(\xi, \GammaV(\eta, \phi)) - \GammaH(\eta, \GammaV(\xi, \phi))
\end{gathered}
  \end{equation}
We will reduce it to \cref{eq:cursubmer}. We have, with $\omega = \Gamma(\eta, \phi)$, using \cref{eq:gammaF}
  $$\begin{gathered}\ttH\Gamma(\xi, \omega) = -\ttH(\rD_{\xi}\Pi_{\sfg})\omega +\ttH\mrGamma(\xi, \omega) =
  -\ttH(\rD_{\xi}\Pi_{\sfg})\omega +\GammaH(\xi, \omega) +(\rD_{\xi}\ttH)\omega =\\
  -(\rD_{\xi}(\ttH\Pi_{\sfg}))\omega +(\rD_{\xi}\ttH)\Pi_{\sfg}\omega+\GammaH(\xi, \omega) +(\rD_{\xi}\ttH)\omega = (\rD_{\xi}\ttH)\Pi_{\sfg}\omega+\GammaH(\xi, \omega)
  \end{gathered}
  $$
  Hence, $\ttH\Gamma(\xi, \Gamma(\eta, \phi)) = (\rD_{\xi}\ttH)\Pi_{\sfg}\Gamma(\eta, \phi)+\GammaH(\xi, \Gamma(\eta, \phi))$. By product rule
  $$\begin{gathered}
    \ttH\rD_{\xi}(\mrGamma(\eta, \phi))= \rD_{\xi}(\ttH\mrGamma(\eta, \phi)) - (\rD_{\xi}\ttH)\mrGamma(\eta, \phi)  \\
  \end{gathered}
  $$
and permuting the role of $\xi$ and $\eta$ the first line of \cref{eq:tmp1} is:
$$\begin{gathered}
  -\rD_{\xi}(\ttH\mrGamma(\eta, \phi)) + (\rD_{\xi}\ttH)\mrGamma(\eta, \phi)
  +\rD_{\eta}(\ttH\mrGamma(\xi, \phi)) - (\rD_{\eta}\ttH)\mrGamma(\xi, \phi)\\
   -(\rD_{\xi}\ttH)\Pi_{\sfg}\Gamma(\eta, \phi)-\GammaH(\xi, \Gamma(\eta, \phi)) 
    +(\rD_{\eta}\ttH)\Pi_{\sfg}\Gamma(\xi, \phi)+\GammaH(\eta, \Gamma(\xi, \phi))
  \\=
-\rD_{\xi}(\ttH\mrGamma(\eta, \phi)) -\GammaH(\xi, \Gamma(\eta, \phi))
+\rD_{\eta}(\ttH\mrGamma(\xi, \phi)) +\GammaH(\eta, \Gamma(\xi, \phi))
\end{gathered}
$$
Where we have used \cref{eq:mrGamma0}. Combine with the second line of \cref{eq:tmp1}:
$$\begin{gathered}-\rD_{\xi}(\ttH\mrGamma(\eta, \phi)) -\GammaH(\xi, \Gamma(\eta, \phi)-\GammaV(\eta, \phi)
+\rD_{\eta}(\ttH\mrGamma(\xi, \phi)) +\\
\GammaH(\eta, \Gamma(\xi, \phi)-\GammaV(\xi, \phi)) +2\GammaH(\phi, \GammaV(\xi, \eta)
\end{gathered}
$$
which reduces to the right-hand side of \cref{eq:cursubmer}, as $\Gamma(\eta, \phi) = \GammaV(\eta, \phi) + \GammaH(\eta, \phi)$ since $\Pi_{\sfg} =\ttH + \ttV$.

When $\cH$ is the horizontal lift of a Riemannian submersion, $\rA_{\xi}\eta$ is antisymmetric from Lemma 2 of \cite{ONeil1966} or \cref{eq:oneilLie}, $\rA_{\xi}\phi = -\rA_{\phi}\xi$. From Theorem 2 of \cite{ONeil1966}, the horizontal lift of the Riemannian curvature tensor on $\cB$ satisfies \cref{eq:ONeil13}.
\end{proof}
\begin{example}\label{ex:flag_curv}
  Continuing with our example of flag manifolds, as before, let $\frk = \sod$, $\frb$ is its orthogonal complement in $\oo(n)$, consisting of antisymmetric matrices with zero diagonal blocks. For any antisymmetric matrix $X$, let $X_{\frk}$ be the block-diagonal component of $X$ (block size determined by $\frk$), and $X_{\frb} = X - X_{\frk}$. Consider $U\in \SOO(n)$ and three horizontal tangent vectors $\xi = UA, \eta = UB, \phi = UC$ at $U$, with $A, B, C\in\frb$. Let us compute \cref{eq:ONeil13} first. From \cref{eq:Axe}, the term $2\rAd_{\phi}\rA_{\xi}\eta$ is $-\GammaH(UC, U[A, B]_{\frk})$. Using \cref{eq:UAB}, for an antisymmetric block diagonal matrix $D$, $\GammaH(UC, UD) = \frac{1}{2}U[C, D]_{\frb}$,
So $\rAd_{\phi}\rA_{\xi}\eta = (1/4)U[[A, B]_{\frk}, C]$ (also note $[[A, B]_\frk, C] = [[A, B]_{\frk}, C]_{\frb}$), and \cref{eq:ONeil13} gives the following formula for lift of the curvature of flag manifolds at $U\in\SOO(n)$
$$\RcH_{\xi, \eta}\phi = \frac{1}{4}U\{[[A, B], C]_{\frb} +2[[A, B]_{\frk}, C]  - [[B, C]_{\frk},  A] - [[C, A]_{\frk}, B]\}$$
or more formally, if $\cL_U$ is the left multiplication by $U$ and $d\cL_U$ its differential
\begin{equation}\label{eq:flag_curv}
\RcH_{\xi, \eta}\phi = \frac{1}{4}d\cL_U\{[[A, B], C]_{\frb} +2[[A, B]_{\frk}, C]  - [[B, C]_{\frk},  A] - [[C, A]_{\frk}, B]\}
\end{equation}
This formula has a generalization to naturally reductive homogeneous spaces, which we will review shortly. Alternatively, to use \cref{eq:cursubmer2}, from \cref{eq:UAB}
$$\begin{gathered}U^{\sfT}\GammaH(\xi, \GammaH(\eta, \phi)) = -\frac{1}{4}A(BC+CB-[B,C]_{\frk}) - \frac{1}{4}(BC+CB-[B,C]_{\frk})A \\
  -\frac{1}{2}\{A[B,C]_{\frk} +\frac{1}{4}[A, BC+CB+[B,C]_{\frk}]_{\frk}
\end{gathered}$$
and a lengthy but routine computation eventually gives us \cref{eq:flag_curv}.
\end{example}
\begin{remark}
We briefly review a few main facts about naturally reductive homogeneous spaces, used later in \cref{sec:jac}. Follow \cite{KobNom}, (where $\frm, \frk$ and $\frb$ are denoted by $\mathfrak{k}, \mathfrak{h}$ and $\frm$ respectively), we call a homogeneous space $\cB = \rM/\cK$ a naturally reductive homogeneous space where $\rM$ is a Lie group, $\cK$ is a closed subgroup with Lie algebras $\frm, \frk$ and $\frm = \frk\oplus\frb$, such that $[\frk, \frb]\subset \frb$, and the subspace $\frb$ is equipped with an $\Ad(\cK)$-invariant (positive-definite) non-degenerate symmetric bilinear form $\langle\rangle_{\frb}$ satisfying
\begin{equation}\label{eq:natural}\langle X, [Z, Y]_{\frb}\rangle_{\frb} + \langle [Z, X]_{\frb}, Y\rangle_{\frb} = 0\text{ for } X, Y, Z\in\frb
\end{equation}
here, for $W\in \frm$, $W_{\frk}$ and $W_{\frb}$ are components of $W = W_{\frk} +W_{\frb}$ in the decomposition $\frm = \frk\oplus\frb$. The form $\langle\rangle_{\frb}$ induces an invariant Riemannian metric on $\cB$. If $\langle\rangle_{\frb}$ is induced from a bi-invariant positive-definite inner product on $\frm$ and $\frb$ is orthogonal to $\frk$ then
\cref{eq:natural} is satisfied, in particular, flag manifolds are naturally reductive. For $A\in\frb$ and $U\in \rM$, if $\cL_U$ is the operator of left multiplication by $U$, $d\cL_UA$ is a tangent vector at $U$, which we will sometimes denote by $UA$. Denote by $\kappa_A$ the invariant vector field $X\mapsto d\cL_XA$ on $\rM$. The naturally reductive assumption implies (\cite{KobNom}, theorem 3.3)
\begin{equation}
  \nabla^{\cH}_{\kappa_A}\kappa_B := \ttH\nabla_{\kappa_A}\kappa_B = \frac{1}{2}\kappa_{[A,B]_{\frb}}
\end{equation}
From \cite{KobNom}, proposition II.10.3.4, we have (note the opposite sign convention)
\begin{equation}\label{eq:natret} 4\RcH_{\xi, \eta}\phi = d\cL_U\{4[[A, B]_{\frk}, C] -  [A, [B, C]_{\frb}]_{\frb} -\
            [B, [C, A]_{\frb}]_{\frb} + 2[[A, B]_{\frb}, C]_{\frb}\}
\end{equation}
We show it is equivalent to \cref{eq:flag_curv}. We have $4[[A, B]_{\frk}, C] + 2[[A, B]_{\frb}, C]_{\frb} = 2[[A, B]_{\frk}, C] + 2[[A, B], C]_{\frb}$, as $[[A, B]_{\frk}, C]\in\frb$. Expand $2[[A, B], C]_{\frb} = 2[[A, C], B]_{\frb} + 2[A,[B, C]]_{\frb}$, the curly bracket of the above becomes
$$\begin{gathered}2[[A, B]_{\frk}, C] + 2[[A, C], B]_{\frb} + 2[A,[B, C]]_{\frb} -[A, [B, C]_{\frb}]_{\frb} - [B, [C, A]_{\frb}]_{\frb}=\\
  2[[A, B]_{\frk}, C] + [B, [C, A]]_{\frb} + [A,[B, C]]_{\frb} -[[B, C]_{\frk}, A] - [[C, A]_{\frk}, B]=\\
    2[[A, B]_{\frk}, C] + [[A, B], C]]_{\frb} -[[B, C]_{\frk}, A] - [[C, A]_{\frk}, B]
\end{gathered}$$
using the Jacobi identity again in the last expression. Thus, \cref{eq:flag_curv} is an alternative formula for the curvature of naturally reductive homogeneous spaces.
\end{remark}
\section{Double tangent bundle}
\subsection{Tangent bundle of a tangent bundle in embedded ambient structure}
The embedding $\rM\subset \cE$ as differentiable manifolds allows us to identify the tangent bundle of $\rM$ with a subspace of $\cE^2 = \cE\oplus\cE$. If $\rM$ is defined by a system of equations, we can differentiate them to derive the defining equation for $\cT\rM$. We have seen in \cref{ex:sasaki}, if $\rM$ is the unit sphere with defining equation $x^{\sfT}x = 1$, then $\cT\rM$ is considered as a pair $(x, v)$ with $x^{\sfT}x = 1$ and $x^{\sfT}v = 0$, the second equation is linear in $v$, obtained by taking the directional derivative of $x^{\sfT}x - 1$.

In general, a tangent vector to $\cT\rM$ could be considered as an element in $\cE^2$. Corresponding to the $\mathfrak{m}$anifold and $\mathfrak{t}$angent components $x$ and $v$ of $\cT\rM$, a tangent vector $\tDelta$ to $\cT\rM$ at $(x, v)$ has two components $\Delta_{\frm}$ and $\Delta_{\frt}$, $\tDelta= (\Delta_{\frm}, \Delta_{\frt})\in \cE^2$. In the case of the sphere, the constraints on $\Delta_{\frm}$ and $\Delta_{\frt}$ are $x^{\sfT}\Delta_{\frm} = 0$ and $\Delta_{\frm}^{\sfT}v + x^{\sfT}\Delta_{\frt} = 0$.

Instead of working with specific constraints, our approach will be to define the double tangent space $\cT\cT\rM$ via the projection operators $\Pi_{\sfg}$.

\begin{proposition}\label{prop:TTM} Let $(\rM, \sfg, \cE)$ be an embedded ambient structure. The tangent bundle $\cT\rM$ of $\rM$ is a submanifold of $\cE^2$ consisting of pairs $(x, v)$ with $x\in \rM$, $v\in \cE$ such that $\Pi_{\sfg, x}v = v$. The tangent bundle $\cT\cT\rM$ of $\cT\rM$ is a submanifold of $\cE^4$ consisting of quadruples $(x, v, \Delta_{\frm}, \Delta_{\frt})\in\cE^4$ with $x\in \rM$ satisfying
  \begin{equation}\label{eq:TTM}
    \begin{gathered}
      \Pi_{\sfg, x} v = v\\
      \Pi_{\sfg, x} \Delta_{\frm} = \Delta_{\frm}\\
      (\rD_{\Delta_{\frm}}\Pi_{\sfg, x})v + \Pi_{\sfg, x}\Delta_{\frt} = \Delta_{\frt}
\end{gathered}    
  \end{equation}
In particular, $\Delta_{\frm}$ is a tangent vector at $x$. If $v = 0$ or $\Delta_{\frm}=0$ then $\Delta_{\frt}$ is also a tangent vector at $x$. If $(x, v, \Delta_{\frm}, \Delta_{\frt})\in\cT\cT\rM$ then $(x, \Delta_{\frm}, v, \Delta_{\frt})\in\cT\cT\rM$.
\end{proposition}
Let $\gamma(t)$ be the geodesic associated with the metric $\sfg$ on $\rM$. We will use the notation $\Exp$ to denote the exponential map, with $\Exp_x v = \gamma(1)$ where $\gamma$ is the geodesic with $\gamma(0) = x, \dot{\gamma}(0) = v$. If the manifold is not complete, $\gamma(1)$ may not exist, but below we look at $\Exp_x tv$, which exists if $t$ is small enough.
\begin{proof} The statement $v\in\cT_x\rM$ if and only if $\Pi_{\sfg, x}v = v$ is from the definition of the projection. The constraint on a tangent vector $(\Delta_{\frm}, \Delta_{\frt})\in\cT_{(x, v)}\cT\rM$ at $(x, v) \in \cT\rM$ follows by differentiating the constraints on $x$ and $v$. The condition $x\in\rM$ implies $\Delta_{\frm}\in \cT_x\rM$, or $\Pi_{\sfg, x} \Delta_{\frm} = \Delta_{\frm}$. The condition $\Pi_{\sfg, x} v = v$ implies $(\rD_{\Delta_{\frm}}\Pi_{\sfg, x})v + \Pi_{\sfg, x}\Delta_v = \Delta_v$. Conversely, assuming the pair $(\Delta_{\frm}, \Delta_v)$ satisfies the conditions of \cref{eq:TTM}. Consider the curve $c(t) = (\Exp_xt\Delta_{\frm}, \Pi_{\Exp_xt\Delta_{\frm}} (v+t\Delta_v))\in\cE^2$. It is clear that it is a curve on $\cT\rM$, with $c(0) = (x, v)$ and $\dot{c}(0) = (\Delta_{\frm}, (\rD_{\Delta_{\frm}}\Pi_{\sfg})_xv + \Pi_{\sfg, x}\Delta_v) = (\Delta_{\frm}, \Delta_v)$, thus $(\Delta_{\frm}, \Delta_v)$ is a tangent vector to $\cT\rM$. The last paragraph is clear, with the last statement follows from $(\rD_{\Delta_{\frm}}\Pi_{\sfg, x})v = (\rD_{v}\Pi_{\sfg, x})\Delta_{\frm}$.
\end{proof}
Define the map $\rU$ from $\cT\rM$ to $\cE^4$ by $\rU(x, v)= (x, v, 0, v)\in\cE^4$, then $\rU(x, v)$ satisfies \cref{eq:TTM}, so $\rU(x, v) \in\cT_{(x, v)}\cT\rM$. It is the familiar {\it canonical vector field}. The map $\frj:(x, v, \Delta_{\frm}, \Delta_{\frt})\mapsto (x, \Delta_{\frm}, v, \Delta_{\frt})\in\cT\cT\rM$ is the {\it canonical flip}.

We will write $\pi:\cT\rM\to\rM$ for the tangent bundle projection. We now introduce the {\it connection map} following \cite{Dombrowski1962,GudKap}, where it is defined via parallel transport. It is shown (Lemma 3.3 of \cite{GudKap} or section 3 of \cite{Dombrowski1962}) that it is a linear bundle map $\rC$ from $\cT\cT\rM$ to $\cT\rM$, satisfying, for all $x\in \rM$, $\Delta \in \cT_{x}\rM$
\begin{equation}\label{eq:conn_map} \rC((dZ)_{x}\Delta) = (\nabla_{\Delta}Z)_x = (\rD_{\Delta}z)_x + \Gamma(\Delta, z(x))_{x}
\end{equation}  
for all vector fields $Z: x\mapsto (x, z(x))$ defined on a geodesic $\gamma(t)$ from $x$ on $\rM$ with $\dot{\gamma}(0) = \Delta$.  The statement in \cite{GudKap} is for vector fields on $\rM$, but the proof only requires a vector field along a curve.
\begin{lemma}\label{lem:conn}
  The connection map $\rC:\cT\cT\rM\to\cT\rM$ at $(x, v) \in \cT\rM$ is given by
  \begin{equation}  (\Delta_{\frm}, \Delta_{\frt})\mapsto \Delta_{\frt} + \Gamma(\Delta_{\frm}, v)_x\end{equation}
    The map $(\Delta_{\frm}, \Delta_{\frt})\mapsto (\Delta_{\frm}, \rC_{(x, v)}(\Delta_{\frm}, \Delta_{\frt}))$ is a bijection between $\cT_{(x, v)}\cT\rM$ and $(\cT_x\rM)^2$. Alternatively, the map $(\Delta_{\frm}, \Delta_{\frt})\mapsto (\Delta_{\frm}, \Delta_{\frt} - (\rD_{\Delta_{\frm}}\ttH)_xv)$ is also a bijection between $\cT_{(x, v)}\cT\rM$ and $(\cT_x\rM)^2$.
\end{lemma}
\begin{proof} Write $\Pi_x$ for $\Pi_{\sfg, x}$. The curve $c(t) = (\Exp_xt\Delta_{\frm}, \Pi_{\Exp_xt\Delta_{\frm}}(v+t\Delta_{\frt}))$ on $\cT\rM$ gives us a vector field $Z: \Exp_xt\Delta_{\frm}\mapsto c(t)$ along the geodesic $\Exp_xt\Delta_{\frm}$. At $x=\pi c(0)$, $(dZ)_x\Delta_{\frm} =
  (\Delta_{\frm}, (\rD_{\Delta_{\frm}}\Pi)_x v + \Pi_x\Delta_{\frt}) = (\Delta_{\frm}, \Delta_{\frt})$, hence the left-hand side of \cref{eq:conn_map} is $\rC(\Delta_{\frm}, \Delta_{\frt})$ and the right-hand side is $\Delta_{\frt} + \Gamma(\Delta_{\frm}, v)_x$.

To show $f:(\Delta_{\frm}, \Delta_{\frt})\mapsto (\Delta_{\frm}, \rC_{(x, v)}(\Delta_{\frm}, \Delta_{\frt}))$ is injective, if $(\delta_{\frm}, \delta_{\frt})$ is such that $f(\delta_{\frm}, \delta_{\frt}) = f(\Delta_{\frm}, \Delta_{\frt})$ then $\delta_{\frm} = \Delta_{\frm}$, and it is clear from the affine format of $\rC$ that $\Delta_{\frt} = \delta_{\frt}$. To show it is onto, take $(\Delta_{\frm}, \Delta_{\frc})\in \cT_x\rM^2$. Define $\Delta_{\frt} := \Delta_{\frc} - \Gamma(v, \Delta_{\frm})_x$. We can verify the tangent relation
  $$(\rD_{\Delta_{\frm}}\Pi) v + \Pi(\Delta_{\frc} - \Gamma(v, \Delta_{\frm})_x) =
\Pi\Delta_{\frc} - \Gamma(v, \Delta_{\frm})_x = \Delta_{\frt}
  $$
from \cref{eq:mrGamma0}, and $\Pi\Delta_{\frc} = \Delta_{\frc}$. It is clear $\rC(\Delta_{\frm}, \Delta_{\frc} - \Gamma(v, \Delta_{\frm})_x) =\Delta_{\frc}$, so $f$ is onto.  The alternate identification with $(\cT_x\rM)^2$ is also clear, as the difference between them is $(\rD_{\Delta_{\frm}}\ttH)_x v + \Gamma(\Delta_{\frm}, v) = \mrGamma(\Delta_{\frm}, v)\in \cT_x\rM$.
\end{proof}
The connection map appears in the initial condition for Jacobi fields and plays a pivotal role in natural metrics on tangent bundles.

\begin{example}
For our $\rM=\SOO(n)\subset\cE$ example, recall if $U\in\rM$, $\eta \in \cT_U\rM$ if and only if $U^{\sfT}\eta + \eta^{\sfT}U= 0$, and the projection is $\Pi\omega = \frac{1}{2}(\omega - U\omega^{\sfT}U)$ for $\omega \in \cE$. The equation $\Pi\omega=\omega$ is equivalent to $U^{\sfT}\eta + \eta^{\sfT}U= 0$. The defining equations for $\tDelta = (\Delta_{\frm}, \Delta_{\frt})\in\cT\cT\rM$ at $(U, \eta)\in \cT\rM$ are $\Pi\Delta_{\frm} = \Delta_{\frm}$ and
$$-\frac{1}{2}\Delta_{\frm}\eta^{\sfT}U - \frac{1}{2}U\eta^{\sfT}\Delta_{\frm} + \Pi\Delta_{\frt} = \Delta_{\frt}$$
The expressions are simpler if we translate to the identity, with $\eta = UA, \Delta_{\frm} = UB$ where $A$ and $B$ are antisymmetric matrices. We can set $\Pi\Delta_{\frt} = UD$, and hence $\Delta_{\frt} = U\{\frac{1}{2}(BA+AB) + D\}$ for an antisymmetric matrix $D$.

The connection map is $(\Delta_{\frm}, \Delta_{\frt})\mapsto \Delta_{\frt} + \frac{1}{2}(\Delta_{\frm}\eta^{\sfT}U + U\eta^{\sfT}\Delta_{\frm}) = UD$.
\end{example}
\subsection{Tangent bundle of horizontal space in submersed ambient structures}\label{sec:subm_tangent}
We have an analogous result for submersed ambient structures. The main actors, to be introduced subsequently, have their relationship described in \cref{fig:HMB}.
\begin{figure}
  \centering
  \begin{tikzpicture}[node distance={17mm}, scale=.70]
    \node (1){$\cV\cH\rM$};
    \node[right of=1] (2){$\cQ\cH\rM$};
    \node[right of=2] (3){$\cT\cH\rM$};
    \node[right of=3] (4){$\cT\cT\cB$};
    \node[below of=3] (5){$\cH\rM$};
    \node[below of=1] (9){$\cV\rM$};
    \node[left of=9] (10){$\cT\rM$};
    \node[right of=5] (6){$\cT\cB$};    
    \node[below of=5] (7){$\rM$};
    \node[below of=6] (8){$\cB$};
    \fill (1) --node[]{$\oplus$}(2);
    \fill (2) --node[]{$=$}(3);
    \fill (10) --node[]{$=$}(9);
    \fill (9) --node[]{$\oplus$}(5);    
    \draw[->] (3) -- (4);
    \draw[->] (9) -- node[left]{$\rmb$}(1);
    \draw[->] (5) -- node[left]{$\pi_{|\cH\rM}$}(7);
    \draw[->] (6) -- node[left]{$\pi_{\cB}$}(8);
    \draw[->] (7) -- node[above]{$\qq$}(8);
    \draw[->] (5) -- node[above]{$d\qq$}(6);
    \draw[->] (3) -- node[above]{$d^2\qq$}(4);
    \draw[->] (5) to [bend right=10] node[above]{$\rmq$}(2);
    \draw[->] (5) to [bend left=10] node[below]{$\rmp$}(2);
    \draw[->] (2) to [bend right=30] node[below]{$\rCa$}(5);
    \draw[->] (3) -- node[left]{$d\pi$}(5);
    \draw[->] (4) -- node[left]{$d\pi_{\cB}$}(6);
    \draw[->] (10) -- node[left]{$\pi$}(7);    
  \end{tikzpicture}
  \caption{Relationship between bundles in a Riemannian submersion}
\label{fig:HMB}
\end{figure}
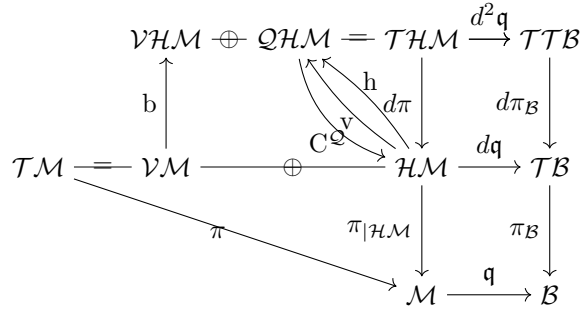
The main idea is if $\qq:\rM\to\cB$ is a Riemannian submersion, then $d\qq:\cH\rM\to\cT\cB$ is a differentiable submersion, with the vertical space $\cV\cH\rM$ having an explicit description via the map $\rmb$ which we will explain here. When the submersion is a quotient by a right action of a group of isometries $P$, then $P$ also acts on the tangent bundle. If $\psi$ belongs to the Lie algebra of $P$, with $(\rM, \sfg, \cE)$ is an embedded ambient structure and $\psi$ and $\exp t\psi$ act as operators on $\cE$ for $t\in \R$, the action of $\exp(t\psi)$ on the tangent bundle is given by $(x, v)\exp(t\psi) = (x\exp(t\psi), v\exp(t\psi))$ for $(x, v)\in\cT\rM$. The action maps the vertical space at $x$ to the vertical space at $x\exp(t\psi)$, hence the horizontal space at $x$ to that at $x\exp(t\psi)$. So if $v$ is horizontal, $(x\exp(t\psi), v\exp(t\psi))$ is a curve on $\cH\rM$ and differentiating, $(x\psi, v\psi)\in \cT_{(x, v)}\cH\rM$. The map $\rmb$ represents the correspondence $(x, x\psi)\mapsto (x, v, x\psi, v\psi)$ (note $v\psi$, like $\Delta_{\frt}$ in the previous section, does not belong to $\cT_x\rM$). We show this correspondence extends to submersions in general and could be defined using projections, thus providing an explicit decomposition of $\cT\cH\rM$ to vertical and horizontal spaces. Eventually, we will equip $\cT\cH\rM$ with a metric to make $d\qq_{|\cH\rM}:\cH\rM\to\cT\cB$ a Riemannian submersion, so the bundle $\cQ\cH\rM$ described below is the horizontal bundle of this submersion.
\begin{proposition}\label{prop:QHM}
Let  $(\rM, \qq, \cB, \sfg, \cE)$ be a submersed ambient structure with $\qq:\rM\rightarrow\cB$ is a Riemannian submersion. Let $\cT\cB$ be the tangent bundle of $\cB$ and $\cH\rM$ the horizontal bundle of $\rM$ in the submersion. Then $\cH\rM$ could be considered as a submanifold of $\cE^2$ consisting of pairs $(x, v)$ with $x\in\rM$ and $\ttH v = v$. The map $d\qq_{|\cH\rM}: \cH\rM\to \cT\cB$ is a differentiable submersion, with fiber at $(b, v_b)$ the submanifold $\{(x, v_{b,x})|x\in \qq^{-1}b\}$ where $v_{b, x}$ denotes the horizontal lift of $v_b$ at $x$. The tangent bundle $\cT\cH\rM$ of $\cH\rM$ is a submanifold of $\cE^4$ consisting of quadruples $(x, v, \delta_{\frm}, \delta_{\frt})\in\cE^4$ with $x\in\rM, v\in \cH_x\rM$ and
  \begin{equation}\label{eq:submerg}
    \begin{gathered}
      \Pi_x \Delta_{\frm} = \Delta_{\frm}\\
      (\rD_{\Delta_{\frm}}\ttH_x) v + \ttH_x\Delta_{\frt} = \Delta_{\frt}
\end{gathered}    
  \end{equation}
The bundle map $\ttQ:\cT\cH\rM\to\cT\cH\rM$ over $\cH\rM$, mapping $(\Delta_{\frm}, \Delta_{\frt})\in \cT_{(x, v)}\cH\rM$ to $(\ttH_x\Delta_{\frm}, (\rD_{\ttH_x\Delta_{\frm}}\ttH_x) v + \ttH_x\Delta_{\frt})$ is idempotent, $\ttQ^2=\ttQ$, with image a subbundle $\cQ\cH\rM$ of $\cT\cH\rM$ with fibers over $(x, v)$ vectors $(\delta_{\frm}, \delta_{\frt})$ satisfying
    \begin{equation}\label{eq:submerg_h}
    \begin{gathered}
      \ttH_x \delta_{\frm} = \delta_{\frm}\\
      (\rD_{\delta_{\frm}}\ttH)_x v + \ttH_x\delta_{\frt} = \delta_{\frt}
\end{gathered}    
    \end{equation}
    For each vertical vector $\epsilon_{\frm}\in \cV_x\rM$, there exists a unique vector $\epsilon_{\frt}\in\cE$ such that $(x, v, \epsilon_{\frm}, \epsilon_{\frt})\in\cV\cH\rM$, where $\cV\cH\rM$ is the vertical subbundle of the tangent bundle $\cT\cH\rM$ under the differentiable submersion $d\qq_{|\cH\rM}$. Its fiber $\cV_{(x, v)}\cH\rM$ at $(x, v) \in \cH\rM$ is the subspace of $\cT_{(x, v)}\cH\rM$ that $d^2\qq_{|\cH\rM} := d(d\qq_{|\cH\rM})$ maps to the zero tangent vector at $\cT_{d\qq(x, v)}\cT\cB$. We have
    \begin{equation}\label{eq:rmb}
      \epsilon_{\frt} = (\rD_{\epsilon_{\frm}}\ttH)_x v  - (\rD_{v}\ttH)_x\epsilon_{\frm} =  (\rD_{v}\ttV)_x\epsilon_{\frm} - (\rD_{\epsilon_{\frm}}\ttV)_x v =   \GammaH(v, \epsilon_{\frm})-\GammaH(\epsilon_{\frm}, v)
\end{equation}
Thus, $\rB_v:\epsilon_{\frm} \mapsto \epsilon_{\frt}$ defines a linear map from $\cV_x\rM$ to $\cE$. $(\rD_{v}\ttH)_x\epsilon_{\frm}$ is a horizontal tangent vector, $(\rD_{\epsilon_{\frm}}\ttV)_x v$ is a vertical tangent vector of $\rM$. For each $(x, v)\in \cH\rM$, the  map $\rmb$, mapping $\epsilon_{\frm}$ to $(\epsilon_{\frm}, \epsilon_{\frt})$ is a bijection between $\cV_x\rM$ and $\cV_{(x, v)}\cH\rM$.
      
  We have a differentiable bundle decomposition $\cT\cH\rM = \cQ\cH\rM \oplus \cV\cH\rM$, decomposing $(\Delta_{\frm}, \Delta_{\frt})\in \cT_{(x, v)}\cH\rM$ to
  \begin{equation}
\begin{gathered}
    (\Delta_{\frm}, \Delta_{\frt})= (\ttH_x\Delta_{\frm}, \Delta_{\frt} - (\rB_v\ttV_x\Delta_{\frm})) +  (\ttV_x\Delta_{\frm},  \rB_v(\ttV_x\Delta_{\frm}))
\end{gathered}
\end{equation}
  At each fiber, the bundle map $d(d\qq)_{|\cH\rM})_{|\cQ\cH\rM}:\cQ_{(x, v)}\cH\rM\mapsto \cT_{d\qq (x, v)}\cT\cB$ is a linear bijection. Both $\rmb$ and $\cQ\cH\rM$ (hence $\ttQ$) are intrinsic, they are only dependent on the submersion $\qq:\rM\to\cB$.
\end{proposition}
Note that we have not defined a metric on $\cH\rM$, so the decomposition $\cT\cH\rM = \cQ\cH\rM \oplus \cV\cH\rM$ is not yet an orthogonal decomposition. We note both equations
\cref{eq:submerg} and \cref{eq:submerg} have the property that if $(\delta_{\frm}, \delta_{\frt})$ satisfies them, and $\delta_{1}$ is a horizontal vector, then $(\delta_{\frm}, \delta_{\frt} +\delta_1)$ also satisfies them. We will use 
$\ttH_x(\rD_{\Delta_{\frm}}\ttH_x) v= (\rD_{\Delta_{\frm}}(\ttH_x)^2)v - (\rD_{\Delta_{\frm}}\ttH_x)\ttH_xv = 0$ for any tangent vector $\Delta_{\frm}$ and horizontal vector $v$ in the following.
\begin{proof} The descriptions of $\cH\cB$ and $\cT\cH\cB$ are similar to the tangent bundle case, the curve used to prove $(\Delta_{\frm}, \Delta_{\frt})$ satisfying \cref{eq:submerg} is horizontal is
  $$c(t) = (\Exp_xt\Delta_{\frm}, \ttH_{\Exp_xt\Delta_{\frm}} (v+t\Delta_{\frt}))\in\cE^2$$
  Since $\qq$ is a submersion, $d\qq:\cT\rM\to\cT\cB$ is surjective everywhere, $d\qq_{|\cH\rM}$ is also surjective. Let $(\cB, \sfg_{\cB}, \cE_{\cB})$ be an ambient space with a metric operator of $\cB$, thus $\qq$ could be considered a map from $\rM$ to $\cE_{\cB}$, and $d\qq_{|\cH\rM}$ a map from $\cH\rM$ to $\cE_{\cB}^2$ mapping $(x, v)$ to $(\qq(x), d_{\frt}\qq(x, v))$ where $d_{\frt}(x, v)$ denotes the tangent component of $d\qq_{|\cH\rM}$. For an element $(\gamma_{\mathfrak{b}}, \gamma_{\frt})\in\cT_{d\qq(x, v)}\cT\cB$, let $\Delta_{\frm}\in \cH_x\rM$ be the horizontal lift of $\gamma_{\frb}$. Then $(\Delta_{\frm}, (\rD_{\Delta_{\frm}}\ttH)_x v)$ satisfies the last equation of \cref{eq:submerg} hence belongs to $\cT_{(x, v)}\cH\rM$, the second component of its image under $d^2\qq_{|\cH\rM} := d(d\qq_{|\cH\rM})$ differs from $\gamma_{\frt}$ by a tangent vector in $\cT_{\qq(x)}\cB$, which lifts to a horizontal vector $\Delta_1\in\cH_x\rM$, and hence $(\Delta_{\frm}, (\rD_{\Delta_{\frm}}\ttH)_x v +\Delta_1)\in \cH_{x}\rM$ maps to $(\gamma_{\mathfrak{b}}, \gamma_{\frt})\in\cT_{d\qq(x, v)}\cT\cB$. Thus, $d^2\qq_{|\cH\rM}$ is surjective, and $d\qq_{|\cH\rM}$ is a differentiable submersion.

We can verify directly $\ttQ$ is idempotent, as $\ttH_x$ is idempotent and
  $$(\rD_{\ttH_x\Delta_{\frm}}\ttH_x) v + \ttH_x((\rD_{\ttH_x\Delta_{\frm}}\ttH_x) v + \ttH_x\Delta_{\frt}) = (\rD_{\ttH_x\Delta_{\frm}}\ttH_x) v + \ttH_x\Delta_{\frt}
  $$
The description in \cref{eq:submerg_h} of $\cQ\cH\rM$ is clear from idempotency.

Since $\qq$ maps $\rM$ to $\cB$, it could be considered a map from $\rM$ to $\cE_{\cB}$ (an ambient space of $\cB$). Extend $\qq$ to a smooth map on an open subset of $\cE$ near $\rM$, thus we have an extension of $d\qq$ to a map from $\rM\times\cE$ to $\cB\times \cE_{\cB}$, and the second derivative $d^2\qq$ would map $(x, v, \Delta_{\frm}, \Delta_{\frt})\in \cT\cT\rM$ to $(\qq(x), d\qq_{\frt, x} v, d\qq_{\frt, x} \Delta_{\frm},  d\qq_{\frt, x}\Delta_{\frt} + \mathsf{Hess}\qq(\Delta_{\frm}, v)$, the last component is the result of taking the directional derivative of $d_{\frt, x}\qq$, considered as a function from $\rM\times\cE$ to $\cE$, in direction $(\Delta_{\frm}, \Delta_{\frt})$, where $\mathsf{Hess}\qq(\Delta_{\frm}, v)$ denotes the Hessian $\cE_{\cB}$-valued bilinear form of $\qq$.

If $\epsilon_{\frm}\in \cV_x\rM$, consider the curve $e(t) = (\Exp_x tv, \ttV_{\Exp_x tv}\epsilon_{\frm})\in \cV\rM$. The velocity curve $\dot{e}(t)$ is a curve in $\cT\cV\rM$, which evaluates at $t=0$ to be $\dot{e}(0) = (x, \epsilon_{\frm}, v, (\rD_{v} \ttV)_x\epsilon_{\frm})\in\cE^4$. Via the differential of $d\qq$, $\dot{e}(0)$ maps to
$$d^2\qq \dot{e}(0) = (\qq{x}, 0,  d\qq_{\frt, x}v, d\qq_{\frt, x}(\rD_{v} \ttV)_x\epsilon_{\frm} + \mathsf{Hess}\qq(\epsilon_{\frm}, v))_x$$
On the other hand, the differential $d\qq$ maps $e(t)$ to the curve $(\Exp_{\qq(x)} td\qq_{\frt}(x, v), 0)\in \cT\cB$, and the velocity curve of $d\qq e$ at $t=0$ has components $(\qq(x), 0, d\qq_{\frt, x}v, 0)$. This gives us the equality
$$d\qq_{\frt, x}(\rD_{v} \ttV)_x\epsilon_{\frm} + \mathsf{Hess}\qq(\epsilon_{\frm}, v)_x=0$$
But $\ttV_x (\rD_{\epsilon_{\frm}}\ttV_x) v = (\rD_{\epsilon_{\frm}}\ttV)_x v - (\rD_{\epsilon_{\frm}}\ttV)_x\ttV_x v = \rD_{\epsilon_{\frm}}\ttV_x v$ using the derivative of projection trick, (as $v$ is a horizontal vector, we apologize for the possible confusion), so $\rD_{\epsilon_{\frm}}\ttV_x v$ is vertical, hence we have
\begin{equation}\label{eq:VHess}d\qq_{\frt, x}\{(\rD_{v} \ttV)\epsilon_{\frm}
- (\rD_{\epsilon_{\frm}} \ttV)v\}
+ \mathsf{Hess}\qq(\epsilon_{\frm}, v)=0\end{equation}
But $(\rD_{v} \ttV)\epsilon_{\frm} - (\rD_{\epsilon_{\frm}} \ttV)v =
(\rD_{\epsilon_{\frm}} \ttH)v - (\rD_{v} \ttH)\epsilon_{\frm} $ as $\Pi_{\sfg} = \ttV + \ttH$, and on $\rM$ $(\rD_{\epsilon_{\frm}}\Pi_{\sfg})_xv = (\rD_{v}\Pi_{\sfg})_x\epsilon_{\frm}$. This adjustment ensures $\hat{\epsilon} := (x, v, \epsilon_{\frm}, (\rD_{\epsilon_{\frm}} \ttH)_xv - (\rD_{v} \ttH)_x\epsilon_{\frm})$ is in $\cT\cH\rM$ by direct verification. Thus, \cref{eq:VHess} shows $d^2\qq_{|\cH\rM}\hat{\epsilon} =  0 \in \cT_{d\qq(x, v)}\cT\cB$, hence, it belongs to $\cV_{x, v}\cH\rM$. The last equality of \cref{eq:rmb} follows by expanding $\GammaH$ using its definition. The statement that $(\rD_{v}\ttH)_x\epsilon_{\frm} $ is horizontal is proved by verifying $\ttH_x(\rD_{v}\ttH)_x\epsilon_{\frm} =
(\rD_{v}\ttH^2)_x\epsilon_{\frm} - (\rD_{v}\ttH)_x(\ttH_x \epsilon_{\frm}) = (\rD_{v}\ttH)_x\epsilon_{\frm}$.

If another $\epsilon_{\frt}$ is with the same property that $(\epsilon_{\frm}, \epsilon_{\frt})$ maps to zero, then it differs from the constructed vector by a horizontal vector, which maps to zero. From the bijectivity of horizontal projections, we have the uniqueness of $\epsilon_{\frt}$.

The decomposition $\cT\cH\rM = \cQ\cH\rM \oplus \cV\cH\rM$ is now clear, the bijectivity on fibers of $\cQ\cH\rM$ to $\cT\cT\cB$ follows from the bijectivity of horizontal projections.

With $\pi:\cT\rM\to\rM$ is the tangent bundle projection, from \cref{fig:HMB}, the map $\rmb$ is intrinsic because $(\epsilon_{\frm}, \epsilon_{\frt})$ is described intrinsically as the only tangent vector on $\cT\cH\rM$ such that $d\pi(\epsilon_{\frm}, \epsilon_{\frt})=\epsilon_{\frm}$ and $d^2\qq(\epsilon_{\frm}, \epsilon_{\frt})$ is a zero vector in $\cT\cT\cB$, while $\cQ_{(x, v)}\cH\rM$ could be described as the space of vectors in $\cT_{(x, v)}\cH\rM$ mapped to $\cH_x\rM$ under $d\pi$.
\end{proof}
\begin{proposition}\label{prop:frjH} Let $\rA$ be the O'Neil tensor. The map $\frjH$ defined by
  \begin{equation}\frjH:(x, v, \Delta_{\frm}, \Delta_{\frt})\mapsto (x, \ttH_x\Delta_{\frm}, v, (\rD_v\ttH)_x\Delta_{\frm} +\ttH_x \Delta_{\frt})
= (x, \ttH_x\Delta_{\frm}, v, 2\rA_{v}\Delta_{\frm} + \Delta_{\frt})
  \end{equation}
  maps $\cT\cH\rM$ to $\cQ\cH\rM$. Restricting to $\cQ\cH\rM$, it is an involution. It corresponds to the canonical flip of $\cT\cT\cB$, that is, if $d^2\qq$ maps $(x, v, \delta_{\frm}, \delta_{\frt})\in \cQ\cH\rM$ to $(b, w, \delta_b, \delta_w)\in\cT\cT\cB$ then it maps $(x, \delta_{\frm}, v, (\rD_v\ttH)_x\delta_{\frm} +\ttH_x \delta_{\frt})$ to $(b, \delta_b, b, \delta_w)$.
\end{proposition}
\begin{proof}  First, we prove $(x, \ttH_x\Delta_{\frm}, v, (\rD_v\ttH)\Delta_{\frm} + \ttH\Delta_{\frt})$ is in $\cQ_{(x, \ttH_x\Delta_{\frm})}\cH\rM$, using \cref{eq:submerg}. Using the derivative of projection trick
  $$(\rD_{v}\ttH)_x \ttH_x\Delta_{\frm} + \ttH_x\{(\rD_v\ttH)_x\Delta_{\frm} + \ttH_x\Delta_{\frt}\} = 
  (\rD_v\ttH^2)_x\Delta_{\frm} + \ttH_x\Delta_{\frt}$$
  which verifies the last condition in \cref{eq:submerg}. By \cref{eq:oneilLie}
$$(\rD_v\ttH)_x\Delta_{\frm} + \ttH_x\Delta_{\frt} = 2\rA_v\Delta_{\frm} + (\rD_{\Delta_{\frm}}\ttH)_x +\ttH_x\Delta_{\frt} = 2\rA_v\Delta_{\frm} + \Delta_{\frt}$$
  It is an involution because $\rA_{v}\xi = -\rA_{\xi}v$ , or
  $$(\rD_{\Delta_{\frm}}\ttH)_x v + \ttH_x((\rD_v\ttH)_x\Delta_{\frm} +\ttH_x \Delta_{\frt}) = (\rD_{\Delta_{\frm}}\ttH)_x v +\ttH_x\Delta_{\frt} = \Delta_{\frt}
  $$
  To show it corresponds to the canonical flip of $\cT\cT\cB$, note, if $\cE_{\cB}$ is an embedded ambient space of $\cB$ and consider $d\qq$ as a map to $\cE_{\cB}^2$, let $d\qq_{\frt, x}$ be its tangent component, $d\qq(x, v) = (\qq(b), d\qq_{\frt, x} v)$, then for $(x, v, \delta_{\frm}, \delta_{\frt})\in\cQ\cH\rM$
  $$d^2\qq(x, v, \delta_{\frm}, \delta_{\frt}) = (\qq(x), d\qq_{\frt, x} v, d\qq_{\frt, x}\delta_{\frm},   d\qq_{\frt, x}\delta_{\frt} + \mathsf{Hess}\qq(\delta_{\frm}, v)) = (b, w, \delta_b, \delta_w)$$
The components of $d^2\qq(x, \delta_{\frm}, v, (\rD_v\ttH)_x\delta_{\frm} +\ttH_x\delta_{\frt})$ are $b, \delta_b, w$, and
$$\begin{gathered}
(d\qq_{\frt, x}\{(\rD_v\ttH)_x\delta_{\frm} +\ttH_x\delta_{\frt}\} + \mathsf{Hess}\qq(v, \delta_{\frm})) = \delta_w + d\qq_{\frt, x}\{(\rD_v\ttH)_x\delta_{\frm}+\ttH_x\delta_{\frt} -\delta_\frt\}
\end{gathered}
    $$
  But $(\rD_v\ttH)_x\delta_{\frm}+\ttH_x\delta_{\frt} - \delta_\frt= (\rD_v\ttH)_x\delta_{\frm}- (\rD_{\delta_{\frm}}\ttH)_xv=2\rA_v\delta_{\frm}$ is vertical, so it maps to zero in $\cT\cB$.
\end{proof}
We now define the horizontal connection map $\rCa$.
\begin{lemma}\label{lem:connHor} If $\tdelta = (\delta_{\frm}, \delta_{\frt})\in \cQ_{(x, v)}\cH\rM$, then 
  $\rCa_{(x, v)}\tdelta := \delta_{\frt} + \GammaH(\delta_{\frm}, v)$ is in $\cH_x\rM$. The map $(\delta_{\frm}, \delta_{\frt})\mapsto (\delta_{\frm}, \rCa_{(x, v)}(\delta_{\frm}, \delta_{\frt}))$ is a bijection between $\cQ_{(x, v)}\cH\rM$ and $(\cH_x\rM)^2$. The map $(\delta_{\frm}, \delta_{\frt})\mapsto (\delta_{\frm}, \delta_{\frt} - (\rD_{\delta_{\frm}}\ttH) v)$ is also a bijection between these spaces. We have the compatibility equation on $\cQ\cH\rM$
  \begin{equation}d\qq_{|\cH\rM}\rCa = \rC d^2\qq
    \end{equation}
\end{lemma}
\begin{proof} We have
  $$\ttH_x(\delta_{\frt} + \GammaH(\delta_{\frm}, v)) = \ttH_x\delta_{\frt} + \ttH_x\GammaH(\Delta_{\frm}, v) = \delta_{\frt} -(\rD_{\delta_{\frm}}\ttH)_x v+ \ttH_x\GammaH(\delta_{\frm}, v)$$
  The last expression reduces to $\delta_{\frt} +\GammaH(\delta_{\frm}, v)$, thus $\delta_{\frt} +\GammaH(\delta_{\frm}, v)$ is $\cH_x\rM$. The next two statements are proved similar to the embedded case. Since $\rCa$ maps to a horizontal vector, to prove compatibility, if $d^2\qq(x, v, \delta_\frm, \delta_\frt)$ maps to $(b, w, \delta_b, \delta_w)$, the vector field $(\Exp_b t\delta_b, \Pi^\cB_{\Exp_b t\delta_b}(w+t\delta_w))$ along the geodesic $\Exp_b t\delta_b$ lifts to $(\Exp_w t\delta_\frm, \ttH_{\Exp_v t\delta_\frm}(v+t\delta_\frt))$ along the geodesic $\Exp_x t\delta_\frm$, from here $\rC_{(b, w)} (\delta_b, \delta_w)$ lifts to $\rCa_{(x, v)}(\delta_\frt, \delta_\frm)$.
\end{proof}
\begin{example}
Continuing with our example of a flag manifold, with $\rM = \SOO(n)$ and $\cB = \SOO(n)/\SOD$, a vertical vector at $U\in\SOO(n)$ is of the form $\epsilon_{\frm} = U\diag(b_0, \cdots, b_q)$, and $\cH\rM$ consists of pairs $(U, \eta)$ where $U^{\sfT}\eta$ is antisymmetric, with zero diagonal blocks. As explained, we expect $\rB_{\eta} \epsilon_{\frm}$ to be $\eta \diag(b_0, \cdots, b_q)=\eta U^{\sfT}\epsilon_{\frm}$, as it is indeed invariant with respect to the induced action. Note by \cref{eq:UAB}, for two tangent (not necessarily horizontal) vectors to $\SOO(n)$ of form $UA$ and $UB$, with $A$ and $B$ antisymmetric matrices $(\rD_{UA}\ttH)_xUB = \frac{1}{2}(U(AB+BA-2AB_{\frk} + (AB-BA)_{\frk}))$ (we recall $X_{\frk}$ means taking the block diagonals of a matrix $X$).
From here, if $\eta = UA$ is horizontal ($A_{\frk} = 0$) and $\epsilon_{\frm} = UB$ with $B_{\frk} = B$, we have $(AB)_{\frk} = (BA)_{\frk} = 0$ and
$$\rB_{\eta}\epsilon_{\frm} =  (\rD_{\epsilon_{\frm}}\ttH)_U\eta - (\rD_{\eta}\ttH)_U\epsilon_{\frm}= UAB =\eta U^{\sfT}\epsilon_{\frm}
$$
as expected, and $\cV_{(U, \eta)}\cH\rM$ consists of vectors of the form $(\epsilon_{\frm}, \eta U^{\sfT}\epsilon_{\frm})$. The space $\cQ_{(U, \eta)}\cH\rM$ could be identified affinely with two copies of $\cH\rM$, for example $\delta_{\frm} = UC, \delta_1 = \ttH_x\delta_{\frt} = UD$ with two antisymmetric matrices $C$ and $D$ such that $C_{\frk} = D_{\frk} = 0$, then $\delta_{\frt} = U\{\frac{1}{2}(CA+AC + [C, A]_{\frk}) + D\}$.
The canonical flip would map $(U, UA, UC, U\{\frac{1}{2}(CA+AC + [C, A]_{\frk}) + D\})$ to $(U, UC, UA, U\{\frac{1}{2}(CA+AC + [A, C]_{\frk}) + D\})$ as $2\rA_{\eta}\delta_{\frm} = U[A, C]_{\frk}$.
\end{example}
\subsection{Jacobi fields}
\subsection{Jacobi fields of embedded spaces}
It is known Jacobi fields are derivatives of the exponential map, so if the exponential map is known explicitly, Jacobi fields should also be known explicitly. The following proposition assumes a simplified condition and shows how a Jacobi field and its time derivative, identified as $\cE$-valued functions, can be evaluated in our embedded manifold setup.

\begin{proposition} Let $(\rM, \sfg, \cE)$ be a Riemannian manifold with metric operator $\sfg$ on an ambient space $\cE$ with Christoffel function $\Gamma$ and $I\subset \R$ be an interval containing $0$. Let $\gamma: \cT\rM\times I\to \rM, (x, v,t) \mapsto \Exp_x tv$ be the geodesic family with initial condition $\gamma(x, v; 0) = x, \dot{\gamma}(x, v; 0) = v$, and assume $\gamma$ is defined on $\cT\rM\times I$. Then $\gamma$ is a smooth map from $\cT\rM\times I$ to $\rM$. We write $\gamma(t)$ for $\gamma(x, v;t)$ when $(x, v)$ is fixed and understood. Let $(x, v, \Delta_{\frm}, \Delta_{\frt})$ be a tangent vector to $\cT\rM$ at $(x, v) \in \cT\rM$, and for fixed $t$, let $J(t)= J_{(x,v, \Delta_\frm, \Delta_\frt)}(t)$ be the tangent component of $d\gamma$, that is $d\gamma: (x, v, \Delta_{\frm}, \Delta_{\frt}, t) \mapsto (\gamma(t), J(t))_{(x, v, \Delta_{\frm}, \Delta_{\frt})}\in \cT\rM\subset \cE^2$, thus, $J(t) = (\partial^{\cT\rM}_{\Delta_{\frm}, \Delta_{\frt}}\gamma)(t)_{x, v}$, the directional derivative of $\gamma$ in direction $(\Delta_{\frm}, \Delta_{\frt})$ at $(x, v)$. Then $\frJ(t):=(\gamma(t), J(t))$ is a vector field along the curve $\gamma(x, v; t)$ satisfying:
  \begin{equation}\begin{gathered}
      J(0) = \Delta_{\frm}\\
      \dot{J}(0) = \Delta_{\frt}\\
      (\nabla_{d/dt})^2J(t) = \rR_{J(t), \dot{\gamma}(t)} \dot{\gamma}(t)
    \end{gathered}
  \end{equation}
  Thus $\frJ$ is the Jacobi field with the given initial conditions. For any $t$, $\dot{\frJ}(t) =(\gamma(t), J(t), \dot{\gamma}(t), \dot{J}(t))$ belongs to $\cT\cT\rM$, in particular $\dot{\frJ}(0) = (x, \Delta_{\frm}, v, \Delta_{\frt})$, the canonical flip of $(x, v,\Delta_{\frm}, \Delta_{\frt})$. Alternatively, $\Delta_{\frc}:=\Delta_{\frt} + \Gamma(\Delta_{\frm}, v)_x= \rC_{(x, v)}(\Delta_{\frm}, \Delta_{\frt})$ is tangent to $\rM$ and the initial condition could be written as
    \begin{equation}\begin{gathered}
      J(0) = \Delta_{\frm}\\
      \nabla_{d/dt}J(0) = \Delta_{\frc}
      \end{gathered}
    \end{equation}
    for two tangent vectors $\Delta_{\frm}$ and $\Delta_{\frc}$ to $\rM$ at $x$.
\end{proposition}
Here, $\rC$ is the connection map. The two formulations of the initial conditions are equivalent by \cref{lem:conn}. The Jacobi field $J_0(t) = \dot{\gamma}(t)$ corresponds to the initial data $\Delta_{\frm} = v, \Delta_{\frt} = - \Gamma(v, v), \Delta_{\frc} = 0$ (as $\ddot{\gamma} + \Gamma(\dot{\gamma}, \dot{\gamma}) = 0$). The Jacobi field $J_1(t) = t\dot{\gamma}(t)$ corresponds to the initial data $\Delta_{\frm} = 0, \Delta_{\frt} = v = \Delta_{\frc}$.

The theorem should still work with some modifications in the situation where $\gamma$ is not defined on the whole $\cT\rM\times I$, but the initial data for geodesics belong to a subset of $\cT\rM$, satisfying conditions as in the setup of a geodesic variation.

\begin{proof} That $\gamma$ is a smooth map when it is defined follows from the Gr{\"o}nwall inequality as is standard in the theory of differential equations.
  
  Let $\alpha(s, t) = \gamma({\Exp_{x, s\Delta_{\frm}}}, \Pi_{\Exp_{x, s\Delta_{\frm}}}(v + s\Delta_{\frt});t)$. Then $\alpha(0, t)= \gamma(x, v; t)$ is a geodesic, thus $\alpha(s, t)$ is a geodesic variation. Hence $\frac{\partial}{\partial s}\alpha(s, t)|_{s=0}$ is a Jacobi field, which is $\partial^{T\rM}_{d/ds\Exp_{x, s\Delta_{\frm}}(s=0),d/ds(\Pi_{\Exp_{x, s\Delta_{\frm}}}(v + s\Delta_{\frt}))(s=0)}\gamma_{|(x, v; t)}$, and simplifies to $\partial^{T\rM}_{\Delta_{\frm},(\rD_{\Delta_{\frm}}\Pi)_xv + \Pi_x\Delta_{\frt}}\gamma(x, v; t)  = \partial^{T\rM}_{\Delta_{\frm},\Delta_{\frt}}\gamma(x, v; t)$. For initial conditions, first, $J(0) = \partial^{T\rM}_{\Delta_{\frm},\Delta_{\frt}}\gamma(x, v; 0)=\Delta_{\frm}$, as $\gamma(x, v, 0) = x$. As  $\dot{\gamma}(x, \Delta; 0)=\Delta$ for any tangent vector $\Delta$, $\dot{J}(0)= \partial^{T\rM}_{\Delta_{\frm},\Delta_{\frt}}\dot{\gamma}(x, v; 0) =
  \partial^{T\rM}_{\Delta_{\frm},\Delta_{\frt}}((x, \Delta)\mapsto \Delta) = \Delta_{\frt}$. It follows $(\nabla_{\dot{\gamma}(0)})J(0) = \dot{J}(0) + \Gamma(\dot{\gamma}(x, v, 0), J(0)) = \Delta_{\frt} + \Gamma(v, \Delta_{\frm})_{x}=\Delta_{\frc}$.

The remaining statements about $\frJ$ and $\dot{\frJ}$ are just standard statements about differentials and velocities of curves on a manifold.
\end{proof}
\begin{example}
Continuing with the example $\rM := \SOO(n)\subset \cE := \R^{n\times n}$. The geodesic for this metric is $\gamma(U, \eta; t) = U\exp(tU^{\sfT}\eta)$, or $\Exp_U t\eta=U\exp(tU^{\sfT}\eta)$. We have already characterized $\cT\cT\rM$ previously.

The above proposition states that the Jacobi field $J(t)$ with $J(0)=\Delta_{\frm}, \dot{J}(0) = \Delta_{\frt}$ is just the directional derivative of $U\exp(U^{\sfT}\eta)$ in the tangent direction $(\Delta_{\frm}, \Delta_{\frt})$, by the chain rule it is
\begin{equation}\label{eq:jac_son}
J(t) = J(U, \eta, \Delta_{\frm}, \Delta_{\frt}; t) = \Delta_{\frm} \exp(tU^{\sfT}\eta) + tU \frL_{\exp}(tU^{\sfT}\eta, \Delta_{\frm}^{\sfT}\eta + U^{\sfT}\Delta_{\frt})
\end{equation}
where for two square matrices $A$ and $E$, $\frL_{\exp}(A, E)$ denotes the Fr{\'e}chet derivative of $\exp$ at $A$ in direction $E$, as reviewed in \cref{sec:frechet}. For the exponential function, it is well-known $\frL_{\exp}(A, E) = \exp A \sum_{n=0}^{\infty} \frac{(-1)^n}{(n+1)!}\ad_A^n E$. However, as will be reviewed, $\frL_{\exp}(A, E)$ could be evaluated more efficiently by Pad{\'e} approximant.
If $\eta=UA, \Delta_{\frm} = UB, \Pi\Delta_{\frt} = UD$ then $\Delta_{\frm}^{\sfT}\eta + U^{\sfT}\Delta_{\frt} = \frac{1}{2}(AB-BA)+ D$
\begin{equation}\begin{gathered}
    J(t) = U\exp(tA)\{\exp(-tA)B\exp(tA) + t\sum_{n=0}^{\infty} \frac{(-1)^n}{(n+1)!}t^{n}\ad_A^n (\frac{1}{2}[A, B]+ D)\}\\
    = UB\exp(tA) + tU\frL_{\exp}(tA, \frac{1}{2}[A, B] +D)
\end{gathered}    
\end{equation}  
It is clear $J(t)$ is tangent to $\SOO(n)$ at $U\exp(tA)$. The expression on the first line extends to any compact Lie group with bi-invariant metric.
\end{example}
\subsection{Jacobi fields and Riemannian submersion}\label{sec:jac}
Let $(\rM, \qq, \cB, \sfg, \cE)$ be a submersed ambient structure of the submersion $\qq: \rM\to\cB$. Denote by $\nabla$ and $\nabla^{\cB}$ the Levi-Civita connections on $\rM$ and $\cB$, respectively. We recall for two vector fields $\ttX$ and $\ttY$ on $\cB$, the horizontal lift of $\nabla^{\cB}_{\ttX}\ttY$ is $\ttH\nabla_{\bar{\ttX}}\bar{\ttY}$, where $\bar{\ttX}$ and $\bar{\ttY}$ are horizontal lifts of $\ttX$ and $\ttY$. We will use the notation $\nabla^{\cH}_{\bar{\ttX}}\bar{\ttY}:=\ttH\nabla_{\bar{\ttX}}\bar{\ttY}$.
Recall $\frj^{\cH}: (x, v, \Delta_{\frm}, \Delta_{\frt}) = (x, \ttH_x\Delta_{\frm}, v, 2\rA_v\Delta_{\frm} + \Delta_{\frt})$ from $\cT\cH\rM$ to $\cQ\cH\rM$ defined in \cref{prop:frjH} is the canonical flips when restricted to $\cQ\cH\rM$. The construction below is independent of the embedding $\cB\subset \cE_{\cB}$, it allows us to lift Jacobi fields on $\cT\cB$ to curves on $\cH\rM$. See \cite{ONeil1983,LeeRiemann,Gallier} for background materials.

\begin{theorem}\label{theo:jacsub} Let $(\rM, \qq, \cB, \sfg, \cE)$ be a submersed ambient structure of the Riemannian submersion $\qq:\rM\mapsto\cB$, with horizontal bundles $\cH\rM$ and horizontal projection $\ttH$. Let $\cE_{\cB}$ be an inner product space containing $\cB$, so $\cT\cB$ and $\cT\cT\cB$ are considered as subspaces of $\cE_{\cB}^2$ and $\cE_{\cB}^4$. For $(x, v, \Delta_{\frm}, \Delta_{\frt})\in \cT\cH\rM\subset\cE^4$, let $\gamma(t) = \gamma(x,v, ;t) = \Exp_x tv$ be the geodesic family on $\rM$ with initial conditions $\gamma(0) = x, \dot{\gamma}(0) = v$ and let
  $J^{\cH}(t) = \ttH_{\gamma(t)}(\partial^{\cT\rM}_{x, v, \Delta_{\frm}, \Delta_{\frt}}\gamma)(t)$, $\nu_{\frm} = \ttH_x\Delta_{\frm}$, $\nu_{\frt} = (\rD_v\ttH)\Delta_{\frm} + \ttH\Delta_{\frt}$ and $\RcH$ be the horizontal lift of the curvature tensor, then $\frJ^{\cH}(t) := (\gamma(t), J^{\cH}(t))$ is a curve in $\cH\rM\subset\cE^2$ satisfying the Jacobi field equation
  \begin{equation}\label{eq:JacoH}
    \begin{gathered}
      J^{\cH}(0) = \nu_{\frm}\\
      \dot{J}^{\cH}(0) =  \nu_{\frt}\\
      (\nabla^{\cH}_{d/dt})^2 J^{\cH}(t) = \rR^{\cH}_{J^{\cH}(t), \dot{\gamma}(t)}\dot{\gamma}(t)
    \end{gathered}
  \end{equation}
$\dot{\frJ}^{\cH}(t) = (\gamma(t), J^{\cH}(t), \dot{\gamma}(t), \dot{J}^{\cH}(t))$ is a curve in $\cQ\cH\rM$ with
\begin{equation}\dot{\frJ}^{\cH}(0) = (x, \nu_{\frm}, v, \nu_{\frt}) = \frj^{\cH}(x, v, \Delta_{\frm}, \Delta_{\frt})
\end{equation}
Thus, $d\qq$ maps $\frJ^{\cH}$ to the Jacobi field $\frJ^{\cB}$on $\cT\cB$ with $\dot{\frJ}^\cB(0) = d^2\qq(\frjH(x, v, \Delta_{\frm}, \Delta_{\frt}))$. Conversely, for any Jacobi field $\frJ^{\cB}$ on $\cT\cB$ along the geodesic $\gamma^{\cB}(t) = \Exp^{\cB}_b tw$ on $\cB$, with $\dot{\frJ}^\cB(0) = (b, \Delta_b, w, \Delta_w)\in \cT\cT\cB$, let $(x, \nu_{\frm}, v, \nu_{\frt})\in \cQ\cH\rM$ be the unique vector in $\cQ\cH\rM$ such that $d^2\qq (x, \nu_{\frm}, v, \nu_{\frt}) = (b, \Delta_b, v, \Delta_w)$, then
\begin{equation}J^{\cH}(t) = \ttH_{\gamma(t)}(\partial^{\cT\rM}_{x, v, \nu_{\frm},\nu_{\frt} +2\rA_{\nu_{\frm}}v}\gamma^{\cH})(t)
\end{equation}
  satisfies \cref{eq:JacoH} and thus $\frJ^{\cH}$ is the lift of the Jacobi field $\frJ^{\cB}$ from $\cT\cB$ to $\cH\rM$ with the given initial condition. The initial conditions could also be stated as
\begin{equation}\label{eq:submerse_Jac_C}
  \begin{gathered}
      J^{\cH}(0) = \nu_{\frm}\\
      (\ttH_x\nabla_{d/dt}J^{\cH})(0) = \nu_{\frt}+\GammaH(v, \nu_{\frm})_x = \rCa_{(x, v)}(\nu_{\frm}, \nu_{\frt})=:\nu_\frc
      \end{gathered}
\end{equation}
\end{theorem}
\begin{proof}It is clear $J^{\cH}(0) = \ttH_x\Delta_{\frm}$, set $J(t) =\partial^{T\rM}_{x, v, \Delta_{\frm}, \Delta_{\frt}}\gamma(t)$ then
  $$\dot{J}^{\cH}(0) = (\rD_{\dot{\gamma}(0)}\ttH)_x J(0) + \ttH_x J(0) = (\rD_v\ttH)_x\Delta_{\frm} + \ttH_x\Delta_t=\Delta_{\frt}$$
For the alternate initial condition \cref{eq:submerse_Jac_C}, if $\Gamma$ is the Christoffel function on $\rM$
  $$\begin{gathered}\ttH_x(\nabla_{\dot{\gamma}(0)}{J}^{\cH})(0) = \ttH_x (\dot{J}^{\cH}(0) + \Gamma(v, J^{\cH}(0))_x) =\\ \ttH_x\{(\rD_v\ttH)_x\Delta_{\frm} + \ttH_x\Delta_{\frt} + \Gamma(v, \ttH_x\Delta_{\frm})_x\}\\
=   (\rD_v\ttH)_x\Delta_{\frm} - (\rD_v\ttH)_x\ttH_x\Delta_{\frm}+ \ttH_x\Delta_{\frt} +\ttH_x\Gamma(v, \ttH_x\Delta_{\frm})_x
    =\\ \ttH_x\Delta_{\frt} + (\rD_v\ttH)_x\Delta_{\frm} + \GammaH(v, \ttH_x\Delta_{\frm})_x = \nu_{\frt} + \GammaH(v, \nu_{\frm})_x=\nu_{\frc}
    \end{gathered}
  $$
  For the differential equation, let $(\Delta_{\frm}, \Delta_{\frt})$ be a tangent vector to $\cH\rM$ at  $(x, v)\in \cH\rM$ and set $d\qq(x, v) = (b, w)\in \cT\cB$, $d^2\qq(x, v, \Delta_{\frm}, \Delta_{\frt}) = (b, w, \Delta_b, \Delta_w)$. We have $\gamma^{\cB} = \qq\gamma$, hence, the chain rule gives:
  $$(\gamma^{\cB}(t), \partial^{\cT\cB}_{b, w, \Delta_{b}, \Delta_{w}}\gamma^{\cB}(t)) =  d\qq (\gamma(t), \partial^{\cT\rM}_{x, v, \Delta_{\frm}, \Delta_{\frt}}\gamma(t))$$
  Thus, $d\qq$ maps the tangent vector $\partial^{\cT\rM}_{x, v, \Delta_{\frm}, \Delta_{\frt}}\gamma(t)$ at $\gamma(t)$  to $\partial^{\cT\cB}_{b, w, \Delta_{b}, \Delta_{w}}\gamma^{\cB}(t)$ at $\gamma^{\cB}(t)$, hence $J^{\cB}(t):=\partial^{\cT\cB}_{b, w, \Delta_{b}, \Delta_{w}}\gamma^{\cB}$ lifts horizontally to $J^{\cH}(t) :=\ttH_x\partial^{\cT\rM}_{x, v, \Delta_{\frm}, \Delta_{\frt}}\gamma$. Both $(\nabla^{\cB}_{\dot{\gamma}^{\cB}})^2{J}^{\cB}(t)$ and $\rR^{\cB}_{J^{\cB}(t), \dot{\gamma}^{\cB}(t)}\dot{\gamma}^{\cB}(t)$ are in $\cT\cB$, with lifts $(\nabla^{\cH}_{\dot{\gamma}^{\cH}})^2{J}^{\cH}(t)$ and $\rR^{\cH}_{J^{\cH}(t), \dot{\gamma}(t)}\dot{\gamma}(t)$ respectively. By linearity of the lift, we have the differential equation \cref{eq:submerse_Jac_C}. The argument shows $\frJ^{\cH}$ is the lift of $\frJ^{\cB}$. The remaining statements on the initial condition follows from the correspondence between $\frjH$ with the canonical flip on $\cT\cT\cB$ in \cref{prop:frjH}.
\end{proof}
\begin{example}
Continuing with flag manifolds, recall geodesics on $\SOO(n)$ are of the form $\gamma(U, \eta; t) = U\exp(tU^{\sfT}\eta)$ for $(U, \eta) \in \cT\SOO(n)$. With the initial data given by $\eta = UA, \nu_{\frm} = UC$, $\nu_{\frt} =U(\frac{1}{2}(AC + CA + [A, C]_{\frk}) + E)$ (for $UE = \ttH_x\nu_{\frt}=\nu_{\frc}$ in this case), then $\delta_{\frm} = UC, \delta_{t} = U(\frac{1}{2}(AC + CA + [C, A]_{\frk}) + E)$ with $A, C, E$ are horizontal, hence $J^{\cH}(t) = \ttH_{\gamma(t)}J(t)$ is
\begin{equation}\begin{gathered}
    U\exp(tA)\{\exp(-tA)C\exp(tA) + t\exp(-tA) \frL_{\exp}(tA, \frac{1}{2}[A, C]_{\frb} + E)\}_{\frb}=\\
    U\exp(tA)\{C+ \sum_{n=1}^{\infty} \frac{(-1)^{n-1}}{n!}t^n(\ad_A^{n-1} (\frac{1}{2}[A, C]_{\frb}+ E-[A, C]) \}_{\frb}
\end{gathered}    
\end{equation}
\end{example}
This could be generalized to naturally reductive homogeneous spaces in the next theorem. Let $\rM$ be a Lie group with Lie algebra $\frm$, identified with the tangent space at the identity of $\rM$. For $U\in \rM$, let $\cL_U$ be the left multiplication by $U$. From \cite{KobNom}, chapter 10, section 2, geodesics on a naturally reductive homogeneous space lifts to a one-parameter exponential family $U\exp(tA)$ for $U\in \rM, A\in \frb$.
\begin{theorem}\label{theo:jacobi_submerse_nat} If $\cB=\rM/\cK$ is a naturally reductive homogeneous space, with $\rM$ and $\cK$ are Lie groups with Lie algebras $\frm$ and $\frk$, with a decomposition $\frm = \frk \oplus \frb$ such that $[\frb, \frk]\subset \frb$ and an $\ad(\frk)$-invariant and naturally reductive metric $\langle\rangle_{\frb}$ on $\frb$. Thus, any element $S\in\frm$ has a decomposition $S = S_{\frb} + S_{\frk}$. Let $\gamma(t) = U\exp(tA)$ be a horizontal geodesic on $\rM$, the lift of a geodesic $\gamma^{\cB}(t)$ on $\cB$, with $U=\gamma(0)\in\rM$ and $A\in\frb$. A Jacobi field $\frJ^{\cB}$ along $\gamma^{\cB}$ lifts to a horizontal vector field $\frJ^{\cH}$ along $\gamma(t)$ on $\rM$. Assume $\frJ^{\cH}(0) = d\cL_UC$, $\ttH\nabla^{\cH}_{d/dt}\frJ^{\cH}(0)=d\cL_UE$ for $C, E\in \frb$, let  $F(t)$ be the function from $\R$ to $\frb$ such that $\frJ^{\cH}(t) = d\cL_{\gamma(t)}F(t)$, then $F(t)$ satisfies
    \begin{equation}\label{eq:FJacobi}
      \ddot{F}(t) + [A, \dot{F}(t)]_{\frb} - [A,[A,F(t)]_{\frk}] = 0
\end{equation}
and the lift of the Jacobi field is
\begin{equation}\label{eq:jafieldH} \frJ^{\cH}(t) =  d\cL_{\gamma(t)}\{C+ t\cZ(t\ad_A) (\frac{1}{2}[A, C]_{\frb}+ E-[A, C]) \}_{\frb}
\end{equation}
with $\cZ(x) = \frac{1-\exp(-x)}{x}=\sum_{n=0}^{\infty} \frac{(-1)^{n}}{(n+1)!}x^n$.
\end{theorem}
The case $C=A, E=0$ corresponds to the Jacobi field $d\cL_{\gamma(t)}A=\dot{\gamma}(t)$, the case $C=0, E= A$ corresponds to the Jacobi field $td\cL_{\gamma(t)}A = t\dot{\gamma}(t)$. Equation \ref{eq:FJacobi} appeared in \cite{Rauch,Chavel,Ziller}.  We will use the notations $\nabla_{d/dt}^{\cH}$ for $\ttH\nabla_{d/dt}$ and $\cZ_{x=P}=\cZ(P)$ for an operator $P$.
\begin{proof}  Let $F(t)$ be a smooth function from $\R$ to $\frb$. Then we have for $A \in \frb$
  \begin{equation}\label{eq:nablaF}
    \nabla^{\cH}_{d/dt}d\cL_{\gamma(t)}F(t) = d\cL_{\gamma(t)}\{\dot{F}(t) + \frac{1}{2}[A, F(t)]_{\frb}\}
  \end{equation}
  this follows from the fact that for a fixed $t$, $d\cL_{\gamma(s)} F(t)$ and $\dot{\gamma}(s)$ are invariant vector fields along $\gamma$ (in the variable $s$), hence, by the naturally reductive assumption $(\nabla_{d/ds}^{\cH}d\cL_{\gamma(s)}F(t))_{s=t} = \frac{1}{2}d\cL_{\gamma(t)}[A, F(t)]_{\frb}$. On the other hand, and at $s=t$, $d\cL_{\gamma(s)}(F(s)-F(t))$ is the zero tangent vector, thus $(\nabla^{\cH}_{d/ds}d\cL_{ \gamma(s)}(F(s)-F(t))_{s=t}  = \lim_{s\to t}d\cL_{ \gamma(s)}\frac{1}{s-t}(F(s)-F(t)) = d\cL_{\gamma(t)}\dot{F}(t)$, and together we have \cref{eq:nablaF}. Repeating this, we have
  \begin{equation}\label{eq:nbl2}
    (\nabla^{\cH})^2_{d/dt}\gamma(t)F(t) = d\cL_{\gamma(t)}\{\ddot{F}(t) + [A, \dot{F}(t)]_{\frb} + \frac{1}{4}[A, [A, F(t)]_{\frb}]_{\frb}\}
\end{equation}    
Now we use $F(t) = C + (t\cZ_{x=t\ad_A}G)_{\frb}$ where $G := \frac{1}{2}[A, C]_{\frb}+ E-[A, C]$. Differentiate $1-\exp(-x) = x\cZ(x)$ we have $\exp(-x) = \cZ(x) + x \cZ'(x)$ and
  $$\dot{F}(t) = (\cZ_{x=t\ad_A}G +t\ad_A\cZ'_{x=t\ad_A}G)_{\frb} = (\exp(-t\ad_A)G)_{\frb}$$
hence  $\ddot{F}(x) = -(\ad_A\exp(-t\ad_A)G)_{\frb}= - [A, (\exp(-t\ad_A)G)_\frb + (\exp(-t\ad_A)G)_\frk]_\frb$. Noting $G_{\frk} = -[A, C]_{\frk}$, $\exp(-tx) = 1-tx\cZ(tx)$
$$\begin{gathered}\ddot{F}(t) + [A, \dot{F}(t)]_{\frb} = -[A, (\exp(-t\ad_A)G)_{\frk}]_{\frb} =
  -[A, G_{\frk}-t(\ad_A\cZ_{x=t\ad_A}G)_{\frk}]_{\frb}\\=  [A[A, t\cZ_{x=t\ad_A}G+C]_{\frk}]
  = [A[A, (t\cZ_{x=t\ad_A}G)_{\frb}+C]_{\frk}]
\end{gathered}$$
which is \cref{eq:FJacobi} since $[A, t\cZ_{x=t\ad_A}G+C]_{\frk} = [A, (t\cZ_{x=t\ad_A}G)_{\frb}+C]_{\frk}$, because $[\frb, \frk]\subset\frb$ implies $[A, (t\cZ_{x=t\ad_A}G)_{\frk}]_{\frk}=0$.
From here, the inside of the curly brackets on the right-hand side of \cref{eq:nbl2} is
$$\begin{gathered}
  [A[A, (t\cZ_{x=t\ad_A}G)_{\frb}+ C]_{\frk}]_{\frb} + \frac{1}{4}[A, [A, (t\cZ_{x=t\ad_A}G)_{\frb} +C]_{\frb}]_{\frb}=\\
  \frac{3}{4}[A[A, (t\cZ_{x=t\ad_A}G)_{\frb}+ C]_{\frk}]_{\frb}   + \frac{1}{4}[A[A, (t\cZ_{x=t\ad_A}G)_{\frb} +C]]_{\frb}
\end{gathered}
$$
which is $\frac{3}{4}[A, [A, F(t)]_{\frk}] +\frac{1}{4}[A[A,F(t)]]_{\frb}$. On the other hand, from \cref{eq:flag_curv}, which, as mentioned, is equivalent to the curvature formula for naturally reductive homogeneous space in \cref{eq:natret}
$$\RcH_{J^{\cH}(t), \dot{\gamma}(t)}\dot{\gamma}(t) = \frac{1}{4}d\cL_{\gamma(t)}\{[[F(t), A], A]_{\frb} +2[[F(t), A]_{\frk}, A] - [[A, F(t)]_{\frk}, A]\}
$$
and the result follows from anticommutativity of Lie brackets.
\end{proof}
We finish this section with the well-known connection between Killing fields and Jacobi fields. For naturally reductive homogeneous spaces, some results in the next lemma are well-known (\cite{Ziller}) but the formula for the Jacobi field helps clarify them.

We assume the setup of \cref{theo:jacobi_submerse_nat}. If $X\in\frm$, then the Killing field or {\it isotropic Jacobi field} associated with $X$ is $d\cL_{\gamma(t)}(\exp(-t\ad_A)X)_{\frb}$, which is also a Jacobi field, if we substitute $C = X_{\frb}, E = -\frac{1}{2}[A, X_{\frb}]_{\frb} - [A, X_{\frk}]$ in \cref{eq:jafieldH} and note $\exp(-x) = 1 - x\cZ(x)$. A Jacobi field is called isotropic if it arises from $X\in\frm$ in this manner. We will use the notation $\frJ^{\cH}_{U, C, A, E}$ to denote the (lift of the) Jacobi field with initial condition $\frJ^{\cH}(0) = d\cL_U C$, $\nabla^{\cH}_{d/dt}\frJ^{\cH}(0) = d\cL_U E$ along the geodesic $\gamma(t) = \Exp_U (td\cL_UA)= U\exp(tA)$ for $A, C, E\in\frb$ and $U\in\rM$.
\begin{lemma}\label{lem:tech_zil} Let $\cB = \rM/\cK$ be a connected naturally reductive homogeneous space as in \cref{theo:jacobi_submerse_nat}.\hfill\break
  1) A Jacobi field $\frJ^{\cH}_{U, C, A, E}$ along the geodesic $\gamma(t) = \Exp_U (d\cL_UtA)$ is isotropic if and only if there is $X=X_{\frb} + X_\frk\in \frm$ such that $C = X_{\frb}, E =-\frac{1}{2}[A, C]_{\frb}-[A, X_{\frk}]$. In particular, $\frJ^{\cH}_{U, 0, A, E}$ for $C\in\frb$ is isotropic if and only if $E=-[A, D]$ for $D\in\frk$. Also, $\frJ^{\cH}_{U, C, A, -\frac{1}{2}[A, C]_\frb}$ is isotropic.
  \hfill\break 
  2) For all $A\in \frb$, the map $\ad_A^{\frb}: X\mapsto [A, X]_{\frb}$ from $\frb$ to itself is anti-self-adjoint under $\langle\rangle_{\frb}$, thus its eigenvalues are either zero or purely imaginary. If $\lambda\sqrt{-1}$ is a purely imaginary eigenvalue of $\ad_A^{\frb}$ as an operator on $\frb$, with eigenvector $V + \sqrt{-1}V_*$ with $V, V_*\in\frb$, then $[A,V] =-\lambda V_* +D_*, [A, V_*] = \lambda V + D$ for $D, D_*\in \frk$. If $[A,D_*] = [A, D] = 0$ then
    \begin{equation}\label{eq:zjac}
    (\cZ(t\ad_A) V)_{\frb} = \frac{\sin t\lambda}{t\lambda} V - t\frac{1- \cos t\lambda }{(t\lambda)^2}[A, V]_\frb
    \end{equation}
    Thus, $t(\cZ(t\ad_A) V)_{\frb}=0$ for $t = \frac{2k\pi}{\lambda}$, $k\in \Z$ and the Jacobi field $\frJ^{\cH}_{U, 0, A, V}$ vanishes at those points. In particular, if $[A, \frk] = 0$, this formula holds.
\end{lemma}
\begin{proof} For 1), we compare the values and first time-derivatives of two functions 
  $\{C+ t\cZ_{x=t\ad_A} (\frac{1}{2}[A, C]_{\frb}+ E-[A, C]) \}_{\frb}$ and $(\exp(-t\ad_A)X)_{\frb}$ at $t= 0$ and get $C = X_{\frb}$ and $\frac{1}{2}[A, C]_{\frb}+ E-[A, C] =-[A, X]_{\frb}$, which gives us the relation between $X$, $C$ and $E$, and conversely, a direct substitution proves that the first function reduces to the second for $C$ and $E$ satisfying the conditions of 1. From here, when $C =0$, $E=-[A, X_\frk]$ and when $X_\frk=0$ we get $E=-\frac{1}{2}[A, C]_\frb$.

  For 2)For $A, B, C\in \frb$, from the naturally reductive assumption,
$$\langle [A, B]_{\frb}, C\rangle_{\frb} + \langle B, [A, C]_{\frb} \rangle_{\frb} = 0$$
  hence $\ad_A^{\frb}$ is anti-self-adjoint.
  
 If $\ad_A^{\frb}(V + \sqrt{-1}V_*) = \sqrt{-1}\lambda(V + \sqrt{-1}V_*)$, we have $[A, V]_{\frb} = - \lambda V_*$ and $[A, V_*]_{\frb} = \lambda V$, hence $[A, V] = -\lambda V_* + D_*$, and $[A, V_*] = \lambda V + D$, where $V, V_*\in \frb$, for some $D_*, D\in \frk$. We have
  $$[A, [A, V]] = -\lambda [A, V_*] + [A, D_*] = -\lambda^2 V + \lambda D$$
by assumption. By induction, $\ad_A^{2n}V = (-1)^n \lambda^{2n} V + c_nD$ for $c_n\in\R$, $\ad_A^{2n+1}V = (-1)^n \lambda^{2n}[A, V]$. Thus
$$(\cZ(t\ad_A) V)_{\frb} =\sum_{n=0}^{\infty}\frac{(-1)^n}{(2n+1)!}t^{2n}\lambda^{2n}V
- \sum_{n=0}^{\infty} \frac{(-1)^n}{(2n+2)!}t^{2n+1}\lambda^{2n}[A, V]_\frb 
$$
which gives us \cref{eq:zjac}.
\end{proof}
\begin{remark} In \cite{Ziller}, the author suggested that manifolds satisfying the condition
\begin{itemize}
\item[$(ZC)$] $\cB=\rM/\cK$ is a naturally reductive homogeneous space and all twice vanishing Jacobi fields of $\cB$ are isotropic.
\end{itemize}  
are locally symmetric spaces. The case of 3-symmetric naturally reductive manifolds was settled in \cite{Gonzalez}, where \cref{eq:zjac} appeared. Lemma \ref{lem:tech_zil} helps to show, for example, if $\frb$ contains an element $A$ such that $[A, \frk] = 0$ and $[A, \frb]_{\frb}\neq 0$ then $(ZC)$ is not satisfied, as in that case, $\ad_A^{\frb}$ must have an imaginary eigenvalue with eigenvector $V+\sqrt{-1}V_*$ satisfies \cref{eq:zjac}, and $\frJ^{\cH}_{U, 0, A, V}$ is isotropic by $(ZC)$. Hence $V = [A, D]=0$ for some $D\in\frk$ by 1). In particular, the Stiefel manifold $\SOO(n)/\SOO(n-p)$ with the bi-invariant metric on $\SOO(n)$ for integers $n>p$ does not satisfy $(ZC)$. In this case, $\cK=\SOO(n-p)\subset\rM=\SOO(n)$, with $\cK$ identified with the bottom right diagonal $(n-p)\times (n-p)$ block. Here, $\frm = \oo(n) = \frk\oplus\frb_0\oplus\frb_1$ where $\frk$ is formed by the bottom right $(n-p)\times(n-p)$ blocks, $\frb_0\subset\oo(n)$ formed by the top left $p\times p$ diagonal block, $\frb_1\subset \oo(n)$ is the space with those diagonal blocks vanishes and $\frb = \frb_0\oplus\frb_1$. Then $[\frb_0, \frk] = 0$ but $[\frb_0, \frb]_{\frb} \neq 0$.
\end{remark}  
\section{Natural metrics on tangent bundles}\label{sec:nat_metric}
In \cite{Sasaki}, Sasaki introduced a metric on the tangent bundle of a manifold which makes the bundle projection a submersion. Later works, including \cite{Dombrowski1962,KowalSeki}, clarify our understanding of this metric. In \cite{CheegerGromoll} Cheeger and Gromoll proposed a method to construct complete, non-negative metrics on vector bundles on compact homogeneous spaces (thus satisfies the condition of the {\it soul theorem}). This method was extended in \cite{TriMuss} to give a complete metric on tangent bundles of complete manifolds, which the authors named the Cheeger-Gromoll metric. Sasaki and Cheeger-Gromoll metrics are examples of natural metrics, their Levi-Civita connections and curvatures could be computed from those of the base manifold metric \cite{GudKap}. We will start with a general setup, then focus on a metric family inspired by \cite{Abbassi2005,BLW}, which includes both the Sasaki and Cheeger-Gromoll metrics.

The following proposition defines the vertical and horizontal spaces of the fibration $\pi: \cT\rM\to\rM$.
\begin{proposition} With the same notation as \cref{prop:TTM}, let $\pi:\cT\rM\to \rM$, $\pi: (x, v)\mapsto x$ with $(x, v)\in \cT\rM\subset\cE^2$ be the tangent bundle projection map. Its differential $d\pi:\cT\cT\rM\to \cT\rM$ maps $(x, v, \Delta_{\frm}, \Delta_{\frt})$ to $(x, \Delta_{\frm})$, thus its kernel, the {\bf $\pi$-vertical space} at $(x, v)$ consists of elements $(x, v, 0, \Delta_{\frt})\in\cE^4$, where $\Delta_{\frt}$ is tangent to $\rM$ at $x$ (i.e.\ $\Pi_x\Delta_{\frt} = \Delta_{\frt}$). Let $\rC:\cT\cT\rM\to \cT\rM$ be the connection map $(x, v, \Delta_{\frm}, \Delta_{\frt})\mapsto (x, \Delta_{\frt} + \Gamma(v, \Delta_{\frm})_x$ in \cref{lem:conn}. The {\bf $\pi$-horizontal space} at $(x, v)$, defined as the kernel of $\rC_{(x, v)}$ consists of vectors satisfying $\Delta_{\frt} + \Gamma(v, \Delta_{\frm})_x=0$. $\cT_{(x, v)}\cT\rM$ decomposes to vertical and horizontal components as follows:
  \begin{equation}\label{eq:TTM_decomp}
(\Delta_{\frm}, \Delta_{\frt}) = (0, \Delta_{\frt} + \Gamma(\Delta_{\frm}, v)_{x}) + (\Delta_{\frm}, -\Gamma(\Delta_{\frm}, v)_x)
  \end{equation}
\end{proposition}
\begin{proof}The statements for $\pi$ and $d\pi$ are straightforward. It is clear the first component of \cref{eq:TTM_decomp} is vertical and the second is horizontal.
\end{proof}
\begin{remark}\label{rem:rmv_rmh}
Consider $(x, v)\in \cT\rM$. If $\xi \in\cT_{x}\rM$, define the horizontal lift $\xi^{\rmh}\in \cT_{(x, v)}\cT\rM$ by $\xi^{\rmh}:=(\xi, -\Gamma(\xi, v)_x)$, which is in the horizontal space, the kernel of $\rC_{(x, v)}$. Define the vertical lift $\xi^\rmv$ by $\xi^\rmv:=(0, \xi)\in\cT_{(x, v)}\cT\rM$. Let $\ttX$ and $\ttY$ be two vector fields on $\rM$, we use the same notations $\rmh, \rmv$ to denote the corresponding lifts of vector fields. Applying derivative rules:
\begin{equation} (\rD_{\ttX^{\rmh}}\ttY^{\rmh})_{(x, v)} = ((\rD_{\ttX}\ttY)_{x}, (-\rD_{\ttX}\Gamma)(\ttY)_{x}+\Gamma(\ttY, \Gamma(\ttX, v))_{x} -\Gamma((\rD_{\ttX}\ttY))_{x}, v)
  \end{equation}
  \begin{equation} \rD_{\ttX^{\rmh}}\ttY^{\rmv} = (0, \rD_{\ttX}\ttY)
  \end{equation}
  \begin{equation} \rD_{\ttX^{\rmv}}\ttY^{\rmh} = (0, -\Gamma(\ttX, \ttY))
    \end{equation}
  \begin{equation} \rD_{\ttX^{\rmv}}\ttY^{\rmv} = 0
  \end{equation}
From here we have the Dombrowski Lie bracket relations, using our global formula for curvature
  \begin{equation} [\ttX^{\rmh}, \ttY^{\rmh}]_{(x, v)} = ([\ttX, \ttY]_{(x, v)}, \rR_{\ttX_x, \ttY_x}v -\Gamma([\ttX, \ttY], v)_{x})
  \end{equation}
  \begin{equation} [\ttX^{\rmh}, \ttY^{\rmv}] = (\nabla_{\ttX}\ttY)^{\rmv}
    \end{equation}
  \begin{equation} [\ttX^{\rmv}, \ttY^{\rmv}] = 0
    \end{equation}
\end{remark}
We now state a purely linear algebra lemma for metric and projection, which gives us the main idea of the natural metric construction
\begin{lemma}\label{lem:two_func} Let $\cE, \langle\rangle_{\cE}$ be an inner product space, and $\cT\subset\cE$, $\cT_2\subset\cE^2$ be a pair of vector subspaces of $\cE$ and $\cE^2$. Assume $f_1:\cE^2\to\cE, f_2:\cE^2\to\cE$ are two linear maps such that $f_1(\cT_2)=\cT, f_2(\cT_2) = \cT$  and $f_1\oplus f_2=[f_1^{\sfT}, f_2^{\sfT}]^{\sfT}$ is invertible in $\cE^2$ and restricts to a bijection between $\cT_2$ and $\cT^2$. Let $\sfg_1, \sfg_2$ be two positive-definite operators on $\cE$, and $\Pi_1, \Pi_2$ be the projections of $\cE$ to $\cT$ with respect to the inner products defined by these operators. Then the operator
  \begin{equation}\sfG:= \begin{bmatrix}f_1^{\sfT}, f_2^{\sfT}\end{bmatrix}\begin{bmatrix}\sfg_1 & 0 \\ 0 &\sfg_2\end{bmatrix}\begin{bmatrix}f_1\\ f_2\end{bmatrix}
\end{equation}        
  which maps $\tomega\in\cE^2$ to $f_1^{\sfT}\sfg_1 f_1\tomega + f_2^{\sfT}\sfg_2f_2\tomega$ is a positive definite operator, which defines an inner product on $\cE^2$. The projection to $\cT_2$ under $\sfG$ is given by
  \begin{equation}\Pi_{\sfG} = \begin{bmatrix}f_1\\ f_2\end{bmatrix}^{-1}\begin{bmatrix}\Pi_1 & 0 \\ 0 &\Pi_2\end{bmatrix}\begin{bmatrix}f_1\\ f_2\end{bmatrix}
  \end{equation}
or for $\tomega\in \cE^2$, $\Pi_{\sfG}\tomega = (f_1\oplus f_2)^{-1}(\Pi_1 f_1 \tomega, \Pi_2 f_2 \tomega)$. Under this metric, the kernels of $f_1$ and $f_2$ are orthogonal complement subspaces.
\end{lemma}  
\begin{proof}It is clear from the bijective assumption of $(f_1\oplus f_2)$ that $\sfG$ is positive-definite, and it is clear from construction that the image of $\Pi_{\sfG}$ is in $\cT_2$. To show $\Pi_{\sfG}$ is the projection, note for $\tomega\in \cE^2$ and $\teta\in \cT_2$
  $$\begin{gathered}\langle \Pi_{\sfG}\omega, \sfG \teta\rangle_{\cE^2} = \langle\sfG \Pi_{\sfG}\omega,  \teta\rangle_{\cE^2} = \langle \begin{bmatrix}f_1^{\sfT}, f_2^{\sfT}\end{bmatrix}\begin{bmatrix}\sfg_1\Pi_1 & 0 \\ 0 &\sfg_2\Pi_2\end{bmatrix}\begin{bmatrix}f_1\\ f_2\end{bmatrix}\tomega, \teta\rangle_{\cE^2} = \\
       \langle \sfg_1\Pi_1f_1\tomega, f_1\teta \rangle_{\cE} + \langle \sfg_2\Pi_2f_2\tomega, f_2\teta\rangle_{\cE}=
      \langle f_1\tomega, \sfg_1\Pi_1f_1\teta \rangle_{\cE} + \langle f_2\tomega, \sfg_2\Pi_2 f_2\teta\rangle_{\cE}\\=      \langle f_1\tomega, \sfg_1 f_1\teta \rangle_{\cE} + \langle f_2\tomega, \sfg_2 f_2\teta\rangle_{\cE}       
\end{gathered}        
  $$
  where aside from the abuse of the notation to write operator expressions in matrix notation, we use the facts that $\sfG$, $\sfg_1\Pi_1, \sfg_2\Pi_2$ are self-adjoint together with  $f_i\teta\in \cT$ for $i=1,2$. The last expression is $\langle \tomega, \sfG\teta\rangle_{\cE^2}$. Since $f_1\oplus f_2$ is a bijection, $f_1$ and $f_2$ are surjective onto $\cE$. Thus, by the rank-nullity theorem, the kernels of $f_1$ and $f_2$ each has dimension $\dim\cE$. If $\tomega_1\in \Null(f_1), \tomega_2\in \Null(f_2)$ then
$$\langle \tomega_1,\sfG \tomega_2\rangle_{\cE^2}
=\langle f_1\tomega_1, \sfg_1 f_1\tomega_2 \rangle_{\cE} + \langle f_2\tomega_1,  \sfg_2f_2\tomega_2\rangle_{\cE}=0
$$
Thus $\Null(f_1)$ is orthogonal to $\Null(f_2)$, therefore $\cE^2=\Null(f_1)\oplus \Null(f_2)$
\end{proof}  
If $\omega\in \cE$, we will use the notation $\Gamma[\omega]$ to denote the operator from $\rM$ to $\fL(\cE, \cE)$, mapping $x\in \rM$ to the operator $\eta\mapsto \Gamma[\omega]_{x}\eta :=\Gamma(\eta, \omega)_{x}$. For each $x$, $\Gamma[\omega]_{x}$ is a linear map from $\cE$ to itself so we can define the adjoint $\Gamma^{\sfT}[\omega]_{x}$ under $\langle\rangle_{\cE}$, and we define the operator-valued function $\Gamma^{\sfT}[\omega]: x\to \Gamma^{\sfT}[\omega]_x$ from $\rM$ to $\fL(\cE, \cE)$. For a vector field $\ttX$ we use the notations $\Gamma[\ttX]$ and $\Gamma[\ttX]_{|x}$ to define the operator-valued function evaluated as $\Gamma(\nu, \ttX_{x})_{x}$ for $x\in \rM, \nu\in \cE$. We will use block matrix notation to define operators on $\cE^2$ to $\cE^2$, where an operator block $\begin{bmatrix} A & B \\ C & D\end{bmatrix}$ acts on $\tDelta = (\Delta_{\frm}, \Delta_{\frt})$ by $(A \Delta_{\frm} + B \Delta_{\frt}, C\Delta_{\frm} + D\Delta_{\frt})$ (thus we think of $\tDelta$ as a vertical vector for the operator action but write it in horizontal form for convenience). We use the subscript $(x, v)$ to denote the value of a function at point $(x, v)\in \cE^2$, for example, $\sfG_{(x, v)}$ or $\hsfg_{(x, v)}$ below.      
  \begin{theorem}\label{theo:SCGMT}Let $(\rM, \sfg, \cE)$ be an embedded ambient structure of a manifold $\rM$. Let $\hsfg$ be a positive-definite operator-valued function from $\cT\rM$ to $\cE$. Let $\sfG$ be the operator-valued function from $\cT\rM$ to $\fL(\cE^2, \cE^2)$ defined by
  \begin{equation}\label{eq:sfG_metric}
    \begin{gathered}\sfG_{(x, v)} =  \begin{bmatrix} \dI  & \Gamma^{\sfT}[v]_{x}\\
0 & \dI
  \end{bmatrix}\begin{bmatrix} \sfg_x & 0 \\ 0 & \hsfg_{(x, v)}\end{bmatrix} \begin{bmatrix} \dI  & 0\\
\Gamma[v]_{x} & \dI
  \end{bmatrix}
  \end{gathered}
  \end{equation}
for $(x, v) \in \cT\rM$. Then $(\cT\rM, \sfG, \cE^2)$ is an embedded ambient structure, and
\begin{equation}\begin{gathered}\sfG^{-1}_{(x, v)} =
    \begin{bmatrix} \dI  & 0\\
-\Gamma[v]_{x} & \dI
    \end{bmatrix}
    \begin{bmatrix}\sfg^{-1}_x & 0 \\
      0 & \hsfg^{-1}_{(x, v)}
      \end{bmatrix}
    \begin{bmatrix} \dI  & -\Gamma^{\sfT}[v]_{x}\\
0 & \dI
  \end{bmatrix}
  \end{gathered}
\end{equation}
The metric induced by $\sfG$ is natural, this means for $(\Delta_{\frm}, \Delta_{\frt}) \in \cT_{(x, v)}\cT\rM$, the vertical component $(0, \Delta_{\frt} + \Gamma(\Delta_{\frm}, v)_x$ and horizontal component $(\Delta_{\frm}, -\Gamma(\Delta_{\frm}, v)_x$ are orthogonal and $\pi:\cT\rM\to\rM$ is a Riemannian submersion. Let $\Pi_{\sfg}$ and $\Pi_{\hsfg}$ be the projections corresponding to $\sfg$ and $\hsfg$, then the projection $\Pi_{\sfG}$ from $\cE^2$ to $\cT_{(x, v)}\cT\rM$ corresponding to $\sfG$ is
    \begin{equation}\label{eq:bundl_proj}
      \Pi_{\sfG,(x, v)} = \begin{gathered}
        \begin{bmatrix} \dI  & 0\\
-\Gamma[v]_{x} & \dI
  \end{bmatrix}\begin{bmatrix} \Pi_{\sfg, x} & 0 \\ 0 & \Pi_{\hsfg, x, v}\end{bmatrix}
        \begin{bmatrix} \dI  & 0\\
\Gamma[v]_x & \dI
        \end{bmatrix} 
        \end{gathered}
    \end{equation}
In other words, for ${\tomega} = (\omega_x, \omega_v)\in \cE^2$, its projection to $\cT_{(x, v)}\cT\rM$ is
    \begin{equation}\label{eq:bundl_proj2}
      \Pi_{\sfG, (x, v)}(\omega_{\frm}, \omega_{\frt})  = (\Pi_{\sfg, x}\omega_{\frm},
        (-\Gamma[v]_{x}\Pi_{\sfg, x} + \Pi_{\hsfg, x, v}\Gamma[v]_{x})\omega_{\frm} + \Pi_{\hsfg, x, v}\omega_{\frt})
\end{equation}
A Christoffel function $\Gamma_{\sfG}$, evaluated at $\txi = (\xi_x, \xi_v)$, $\teta = (\xi_x, \xi_v)$ in $\cE^2$ is given by $-\rD_{\txi}\Pi_{\sfG}\teta + \mrGamma_{\sfG}(\txi, \teta)$ with
    \begin{equation}
      \begin{gathered}
      \mrGamma_{\sfG}(\txi, \teta) = \Pi_{\sfG}\sfG^{-1}\rK_{\sfG}(\txi, \teta)\\
\rK_{\sfG}(\txi, \teta) = \frac{1}{2}
  ((\rD_{\txi}\sfG)\teta + (\rD_{\teta}\sfG)\txi -\cX_{\sfG}(\txi, \teta))
    \end{gathered}
    \end{equation}
    where $\cX_{\sfG}(\txi, \teta)$ satisfies $\langle \tphi, \cX_{\sfG}(\txi, \teta)\rangle_{\cE^2} = \langle\txi, \rD_{\tphi}\sfG\teta\rangle_{\cE^2}$ for all tangent vectors $\txi, \teta, \tphi$ of $\cT\rM$.
\end{theorem}
Thus, $\sfG$ is diagonal after being transformed by the combination of $\pi$ and the connection map. Direct calculations show $\Pi_{\sfG}\sfG^{-1}$, $\rD_{\txi}\Pi_{\sfG}$ and $\Pi_{\sfG}\sfG^{-1}\rD_{\txi}\sfG$ also have simpler forms after that transformation. For computational purposes, $\rD_{\txi}\sfG$ can be evaluated by block, and $\cX_{\sfG}$ can be evaluated from index-raising expressions of derivatives of $\sfg, \hsfg$ and $\Gamma$.

\begin{proof} Applying \cref{lem:two_func} with $\cT = \cT_x\rM$, $\cT_2= \cT_{(x, v)}\cT\rM$, $f_1 = (d\pi)_{(x, v)}$ and $f_2$ the connection map $\rC_{(x, v)}$, we get the first four equations in the theorem, the kernels of $d\pi$ and of the connection map are orthogonal, and it is clear $\pi$ is a submersion from the expression of $\sfG$. The remaining statements about the Christoffel function follows from \cref{prop:Levi}.
\end{proof}
There are choices of $\hsfg$ such that the projection $\Pi_{\hsfg}$ is the same as $\Pi_{\sfg}$. While this could be done more generally, we focus on a subfamily of the $g$-natural metrics in \cite{KowalSeki,Abbassi2005}, that contains both the Sasaki and the Cheeger-Gromoll metric.

We denote $\alpha(t) = \alpha_t, \beta(t) =\beta_t$ for scalar functions $\alpha$ and $\beta$ in the following.
\begin{proposition}\label{prop:nat_TTM2} Let $\alpha$ be a smooth, positive scalar function from $\R_{\geq 0}$ to $\R_{> 0}$ and $\beta$ is a non-negative smooth function from $\R_{\geq 0}$ to $\R_{\geq 0}$ . Define the operator-valued function $\hsfg:\cT\rM\mapsto \fL(\cE, \cE)$ for $\omega\in\cE$ by:
\begin{equation}\label{eq:BLW}
\hsfg_{(x, v)}\omega =  \alpha_{\|v\|^2_{\sfg, x}}\sfg_x\omega + \beta_{\|v\|^2_{\sfg, x}} \langle v,\omega\rangle_{\sfg, x}\ \sfg_x v
\end{equation}
then  $\Pi_{\sfg, x}$ is the projection of $\hsfg_{(x, v)}$. The inverse of $\hsfg$ is given by
\begin{equation}\label{eq:hsfginv}
  \hsfg_{(x, v)}^{-1}\omega = \frac{1}{\alpha_{\|v\|^2_{\sfg, x}}}\sfg_x^{-1}\omega -\frac{\beta_{\|v\|^2_{\sfg, x}}}{\alpha_{\|v\|^2_{\sfg, x}}(\alpha_{\|v\|^2_{\sfg, x}} +\beta_{\|v\|^2_{\sfg, x}}\|v\|_{\sfg, x}^2)} \langle v, \omega\rangle_{\cE}v
\end{equation}
In that case, if $\Gamma$ is a Christoffel function of $\sfg$, and $\nabla$ is the Levi-Civita connection associated with $\sfg$, for three vector fields $\ttX, \ttY, \ttZ$ of $\rM$ we have
\begin{equation}\label{eq:hsfg_compa}
  \rD_{\ttX^{\rmh}}\langle \ttY\circ\pi, \hsfg(\ttZ\circ\pi)\rangle_{\cE} =
\langle \nabla_{\ttX}\ttY, \hsfg\ttZ\rangle_{\cE} + \langle \ttY, \hsfg\nabla_{\ttY}\ttZ\rangle_{\cE}
\end{equation}
Recall the horizontal lift of a tangent vector $\delta$ at $x$ to $(x, v)$ is given by $\delta^{\rmh}=(\delta, -\Gamma(\delta, v)_x)$, the vertical lift is $\delta^{\rmv} = (0, \delta)$, and the connection map is $\rC_{(x, v)}\txi = \xi_{\frt} +\Gamma(\xi_{\frm}, v)_x$ for a tangent vector $\txi\in \cT_{(x, v)}\cT\rM$. A Christoffel function $\Gamma_{\sfG}$ of $\sfG$ could be evaluated at two tangent vectors $\txi =(\xi_{\frm}, \xi_{\frt}), \teta = (\eta_{\frm}, \eta_{\frt})$ at $(x, v)$ as
  \begin{equation}\label{eq:GammasfGDec}
    \Gamma_{\sfG}(\txi, \teta) =
    \Gamma_{\sfG}(\xi_{\frm}^{\rmh}, \eta_{\frm}^{\rmh}) +
    \Gamma_{\sfG}(\xi_{\frm}^{\rmh}, (\rC\teta)^{\rmv}) +
\Gamma_{\sfG}((\rC\txi)^{\rmv}, \eta_{\frm}^{\rmh}) +
\Gamma_{\sfG}((\rC\txi)^{\rmv}, (\rC\teta)^{\rmv})
  \end{equation}
  \begin{equation}\label{eq:GammasfGparts}
 \begin{gathered}
       \Gamma_{\sfG}(\xi_{\frm}^{\rmh}, \eta_{\frm}^{\rmh}) = (\Gamma(\xi_{\frm}, \eta_{\frm}), -\Gamma(\Gamma(\xi_{\frm},\eta_{\frm}), v) + (\rD_{\xi_{\frm}}\Gamma)(\eta_{\frm}, v)-\\\Gamma( \eta_{\frm}, \Gamma(\xi_{\frm}, v)) + \frac{1}{2}\rR_{\xi_{\frm}, \eta_{\frm}}v)\\
\Gamma_{\sfG}(\xi_{\frm}^{\rmh}, (\rC\teta)^{\rmv}) =
(-\frac{\alpha}{2}\rR_{v, \rC\teta}\xi_{\frm}, \frac{\alpha}{2}\Gamma(\rR_{v, \rC\teta}\xi_{\frm}, v) +
\Gamma(\xi_{\frm}, \rC\teta))\\
\Gamma_{\sfG}((\rC\txi)^{\rmv}, \eta_{\frm}^{\rmh}) = (-\frac{\alpha}{2}\rR_{v,\rC\txi}\eta_{\frm}, \frac{\alpha}{2}\Gamma(\rR_{v,\rC\txi}\eta_{\frm}, v)
+\Gamma(\rC\txi, \eta_{\frm}))\\
\Gamma_{\sfG}((\rC\txi)^{\rmv}, (\rC\teta)^{\rmv}) = (0,\frac{\alpha'}{\alpha}(\|v\|_{\sfg}^2) \{\langle v, \sfg\rC\teta\rangle_{\cE} \rC\txi + \langle v,\sfg \txi\rangle_{\cE}\rC\teta\} + \rF v)
    \end{gathered}
  \end{equation}
where the operator-valued functions $\Gamma$, its directional derivatives, $\rR$ and $\sfg$ are evaluated at $x$, $\rC$ is evaluated at $(x, v)$ and the functions $\alpha, \alpha', \beta, \beta'$ are evaluated at $\|v\|_{\sfg}^2$ and $\rF := (\alpha + \|v\|_{\sfg}^2\beta)^{-1}\{(\beta-\alpha')\langle\rC\txi, \rC\teta\rangle_{\sfg} + (\beta'-2\frac{\alpha'}{\alpha}\beta)\langle v, \rC\txi)\rangle_{\sfg}\langle v, \rC\teta)\rangle_{\sfg}\}$.
\end{proposition}
It is known the case $\alpha = 1$, $\beta=0$ corresponds to the Sasaki metric and
$\Gamma_{\sfG}((\rC\txi)^{\rmv}, (\rC\teta)^{\rmv}) = (0,0)$ in this case. The case $\alpha(t) = \beta(t) = (1 +t)^{-1}$ corresponds to the Cheeger-Gromoll metric \cite{TriMuss}. In that case, the coefficient $\rF$ in $\Gamma_{\sfG}((\rC\txi)^{\rmv}, (\rC\teta)^{\rmv})_{\frt}$ equals
$$\frac{2+\|v\|_{\sfg}^2}{(1+\|v\|_{\sfg}^2)^2}\langle\rC\txi, \rC\teta\rangle_{\sfg} + \frac{1}{(1+\|v\|_{\sfg}^2)^2}\langle v, \rC\txi)\rangle_{\sfg}\langle v, \rC\teta)\rangle_{\sfg}$$
Parametrizing $\alpha, \beta$ by additional parameters provides subfamilies of metrics, for example, those considered in \cite{BLW}. In the following, recall the canonical vector field $\rU$ on $\cT\rM$ is the vector field defined by $\rU(x, v) = (x, v, 0, v)$ for $(x, v)\in\cT\rM$.

\begin{proof} Unless stated otherwise, we will evaluate expressions at $(x, v)$, thus, avoid showing variables, if possible, to shorten expressions. Direct computation shows
  $$\begin{gathered}
    \langle\hsfg\omega_1,\Pi_{\sfg} \omega_2\rangle =     \langle\Pi_{\sfg}^{\sfT}\hsfg\omega_1, \omega_2\rangle = \alpha_{\|v\|^2_{\sfg}}\langle\Pi^{\sfT}_{\sfg} \sfg\omega_1, \omega_2\rangle_{\cE} + \beta_{\|v\|^2_{\sfg}} \langle v,\Pi_{\sfg}^{\sfT}\sfg\omega_1\rangle_{\cE}\langle\sfg v, \omega_2\rangle_{\cE}\\
 = \alpha_{\|v\|^2_{\sfg}}\langle \sfg\Pi_{\sfg}\omega_1, \omega_2\rangle_{\cE} + \beta_{\|v\|^2_{\sfg}} \langle v,\sfg\omega_1\rangle_{\cE}\langle\sfg v, \omega_2\rangle_{\cE}
  \end{gathered}$$
  where we use self-adjointness of $\sfg\Pi_{\sfg}$, and the fact that $v\in \cT_x\rM$. The last expression is $\langle\hsfg\Pi_{\sfg}\omega_1, \omega_2\rangle$, so $\hsfg\Pi_{\sfg}$ is self-adjoint. The formula for $\hsfg^{-1}$ is the operator form of the Sherman-Morrison matrix identity, proved by direct substitution.

We can verify \cref{eq:hsfg_compa} by a direct calculation, for vector fields $\ttX, \ttY, \ttZ$, noting
  $$(\rD_{(\ttX, -\Gamma(\ttX, \rU))}\|\rU\|_{\sfg}^2)_{(x, v)} = 2\langle v, \sfg_x \Gamma(\ttX, v)_x\rangle_{\cE} + 2\langle v, \sfg_x(-\Gamma(\ttX_x, v)_x)\rangle_{\cE}=0$$
Hence $(\rD_{(\ttX, -\Gamma(\ttX, \rU))}\alpha(\|\rU\|_{\sfg}^2))_{(x, v)}=(\rD_{(\ttX, -\Gamma(\ttX, \rU))}\beta(\|\rU\|_{\sfg}^2))_{(x, v)}=0$, thus
$$\begin{gathered}(\rD_{(\ttX, -\Gamma(\ttX, \rU))}\{\alpha(\|\rU\|_{\sfg}^2)\langle \ttY, \sfg\ttZ\rangle_{\cE} +\beta(\|\rU\|_{\sfg}^2)\langle \rU, \sfg\ttY\rangle_{\cE}\langle \rU, \sfg\ttZ\rangle_{\cE}\})_{(x, v)}=\\
  \alpha(\|v\|_{\sfg, x}^2)\{\langle \nabla_{\ttX}\ttY, \sfg\ttZ\rangle_{\cE} +
  \langle \ttY, \sfg\nabla_{\ttX}\ttZ\rangle_{\cE}\}_x +\\
  \beta(\|v\|_{\sfg, x}^2)\{\langle v, \sfg\nabla_{\ttX}\ttY\rangle_{\cE}\langle v,\sfg \ttZ\rangle_{\cE}+\langle \Gamma(v, \ttX), \sfg\ttY\rangle_{\cE}\langle v,\sfg\ttZ\rangle_{\cE}
  +\langle -\Gamma(v, \ttX), \sfg\ttY\rangle_{\cE}\langle v,\sfg\ttZ\rangle_{\cE}\\
  +\langle v, \sfg\ttY\rangle_{\cE}\langle v, \sfg\nabla_{\ttX}\ttZ\rangle_{\cE}
  +\langle v, \sfg\ttY\rangle_{\cE}\langle \Gamma(\ttX, v), \sfg\ttZ\rangle_{\cE}
  +\langle v, \sfg\ttY\rangle_{\cE}\langle -\Gamma(\ttX, v), \sfg\ttZ\rangle_{\cE}  \}_x
\end{gathered}$$
which, writing $\alpha$ and $\beta$ for their values at $\|v\|_{\sfg, x}^2$, could be rearranged to
$$\{\langle \nabla_{\ttX}\ttY,\alpha\sfg\ttZ +\beta\langle v, \sfg\ttZ\rangle_{\cE}v \rangle_{\cE} + \langle \alpha\sfg \ttY +\beta\langle v, \sfg\ttY\rangle_{\cE} , \nabla_{\ttX}\ttZ\rangle_{\cE}\}_x
$$
which is the right-hand side of \cref{eq:hsfg_compa}.

For the Christoffel function, we will formulate a more general result in \cref{theo:NaturalSubmerse} and will provide the rest of the proof.
\end{proof}
For a Riemannian submersion, we have the following
\begin{proposition}\label{prop:projH} Let $(\rM, \qq, \cB, \sfg, \cE)$ be a submersed ambient structure of $\cB$. By \cref{prop:QHM}, $d\qq_{|\cH\rM}\to\cT\cB$ is a differentiable submersion, with the vertical bundle $\cV\cH\rM$ and at $(x, v)\in \cH\rM$, $\cV_{(x, v)}\cH\rM$ consists of vectors of form $(\xi, (\rB_v\xi)_x)\in \cT\cH\rM$, with $\xi\in \cV_x\rM$ and $\rB_v$ defined in \cref{prop:QHM}. Set $\rB(\phi, v)_x := (\rD_{\ttV\phi}\ttH) v - (\rD_v\ttH)\ttV\phi$ for $\phi\in \cT_x\rM$, we have $\rB(\epsilon, v)_x = \rB_v\epsilon$ for $\epsilon\in \cV_x\rM$ and $\rB(\xi, v)_x = 0$ if $\xi\in \cH\rM$. Extend $\rB$ to a smooth bilinear map from $\rM$ to $\fL(\cE\otimes \cE$, $\cE)$. Let $\hsfg$ be a positive-definite operator-valued function from $\cH\rM$ to $\cL(\cE, \cE)$, such that $\hsfg_{(x_1, v_1)} = \hsfg_{(x_2, v_2)}$ if $d\qq(x_1, v_1) = d\qq(x_2, v_2)$. Define
  \begin{equation}\begin{gathered}
(\Gammaa)_{(x, v)}\omega =  \Gammaa[v]_x\omega := \GammaH(\ttH\omega, v)_x - \rB(\omega, v)_x
      \end{gathered}
    \end{equation}
  \begin{equation}\label{eq:Hproj}
    \sfG_{\cQ, (x, v)} :=  \begin{bmatrix} \dI  & (\Gammaa)^{\sfT}[v]_{x}\\
0 & \dI
  \end{bmatrix}\begin{bmatrix} \sfg_x & 0 \\ 0 & \hsfg_{(x, v)}\end{bmatrix} \begin{bmatrix} \dI  & 0\\
\Gammaa[v]_{x} & \dI
  \end{bmatrix}
  \end{equation}
Then $\Gammaa$ is an operator-valued function from $\cH\rM$ to $\fL(\cE, \cE)$ and $\sfGQ$ is a metric operator from $\cH\rM$ to $\cE^2$, defining a Riemannian metric on $\cH\rM$.

Let $\sfg_{\cB}$ be the metric on $\cB$ in the Riemannian submersion. For $(x, v)\in \cH\rM$, let $(b, w) = d\qq(x, v)
\in \cT\cB$, $\hsfg$ induces an inner product $\hsfg_{\cB, (b, w)}$ on $\cT_b\cB$, evaluated on two tangent vectors $\xi^{\cB}$ and $\xi^{\cB}\in\cT_b\cB$ as $\langle \xi^{\rM}, \hsfg_{(x, v)}\eta^{\rM}\rangle_{\cE}$ where $\xi^{\rM}$ and $\eta^{\rM}$ are horizontal lifts of $\xi^{\cB}$ and $\eta^{\cB}$. The inner products $\hsfg_{\cB, (b, w)}$ is well-defined, independent of the lifts. With the associated metric tensors $\sfg_{\cB}$ and $\hsfg_{\cB}$, $\cT\cB$ could be equipped with the metric $\sfG_\cB$ as defined in \cref{theo:SCGMT}. Under the metrics $\sfGQ$ and $\sfG_{\cB}$, the bundle projection $d\qq_{|\cH\rM}:\cH\rM\to\cT\cB$ is a Riemannian submersion.

Let $\ttH_{\sfg, x}$ and $\ttH_{\hsfg, (x, v)}$ be the projections of $\cE$ to $\cH_x\rM$ under the inner products $\sfg_x$ and $\hsfg_{(x, v)}$, respectively. Under $\sfGQ$, the projection of $\cE^2$ to $\cQ_x\cH\rM$ is
    \begin{equation}\label{eq:bundl_projH}
      \ttH_{\sfG,(x, v)} = \begin{gathered}
        \begin{bmatrix} \dI  & 0\\
-\Gammaa[v]_{x} & \dI
  \end{bmatrix}\begin{bmatrix} \ttH_{\sfg, x} & 0 \\ 0 & \ttH_{\hsfg, x}\end{bmatrix}
        \begin{bmatrix} \dI  & 0\\
\Gammaa[v]_x & \dI
        \end{bmatrix} 
        \end{gathered}
    \end{equation}
As a Riemannian metric on $\cH\rM$, $\sfG_{\cQ}$ is only dependent on the values of $\sfg$ and $\hsfg$ evaluated on tangent vectors of $\rM$.

The bundle $\cV\cH\rM$ is the vertical bundle, and $\cQ\cH\rM$ is the horizontal bundle of $\cH\rM$ under the Riemannian submersion $d\qq_{\cH\rM}$. Let $\ttQ$ be the idempotent map defining $\cQ\cH\rM\subset\cT\cH\rM$ in \cref{prop:QHM}, then $\ttQ$ is the projection from $\cT\cH\rM$ to $\cQ\cH\rM$ under $\sfGQ$, thus $\ttQ$ is the restriction of $\ttH_{\sfGQ}$ to $\cT\cH\rM$.
\end{proposition}
\begin{proof}We will apply \cref{lem:two_func}, with $f_1(\omega_{\frm}, \omega_{\frt}) = \omega_{\frm}, f_2(\omega_{\frm}, \omega_{\frt}) = \Gammaa[v]_x\omega_{\frm} + \omega_{\frt}$, $\cT_2 = \cQ_x\cH\rM$, $\cT = \cH_x\rM$. Restricting to $\cT_2 = \cQ_x\cH\rM$, $f_1 = (d\pi_{|\cH\rM})_{(x, v)}$ and $f_2 = \rCa_{(x,v)}$, as by construction $\rB(\eta, v)_x = 0$ for a horizontal vector $\eta$. This gives us the statements that $\sfGQ$ defines a metric under \cref{eq:Hproj}.

Let us show that $\cV_x\cH\rM$ and $\cQ_{(x, v)}\cH\rM$ are orthogonal. Consider $(\epsilon, \rB(\epsilon, v)_x)$ in $\cV_x\cH\rM$ where $\epsilon\in\cV_x\rM$ is a vertical vector, and $\teta = (\eta_{\frm}, \eta_{\frt}) \in \cQ_{(x, v)}\cH\rM$
$$\begin{gathered}\langle (\epsilon, \rB(\epsilon, v)_x), (\sfG_{\cQ})_{(x, v)}\teta\rangle_{\cE^2} =  \langle (\epsilon, \Gammaa(\epsilon, v)_x +\rB(\epsilon, v)_x), (\sfg_x\eta_{\frm}, \hsfg_{(x, v)}(\Gammaa(\eta_{\frm}, v)_x +\eta_{\frt}))\rangle_{\cE^2} \\
=  \langle (\epsilon, 0), (\sfg_x\eta_{\frm}, \hsfg_{(x, v)}(\Gammaa(\eta_{\frm}, v)_x +\eta_{\frt}))\rangle_{\cE^2} = 0
\end{gathered}  $$
as we have constructed $\Gammaa$ such that $\Gammaa(\epsilon, v)_x +\rB(\epsilon, v)_x = 0$, and $\epsilon$ and $\eta_{\frm}$ are orthogonal by the assumption that $\eta_{\frm}$ is a horizontal vector. By construction, $d\qq_{|\cH\rM}$ is an isometry from $\cQ_{(x, v)}\cH\rM$ to $\cT_{(b, w)}\cT\cB$, hence $d\qq_{|\cH\rM}$ is a Riemannian submersion. The statements about the independence of the metric on $\cB$ with respect to the lift follow by simple checks, based on the assumptions of $\sfg$ and $\hsfg$. The statements about $\ttQ$ and $\cQ\cH\rM$ follow from the submersion property of $\sfGQ$.  
\end{proof}
We will recall the lift $\rmb$ in \cref{fig:HMB} in \cref{eq:rmb} and describe the lifts $\rmq, \rmp$.
\begin{definition}\label{def:pqb} For $(x, v)\in \cH\rM$, let $\eta$ be a horizontal vector at $x\in\rM$. Define the $\pi$-horizontal lift $\eta^{\rmq}:= (\eta,-\Gammaa(\eta, v))\in\cQ_{(x, v)}\cH\rM$ , the $\pi$-vertical lift $\eta^{\rmp}:= (0, \eta)\in\cQ_{(x, v)}\cH\rM$. For a vertical vector $\epsilon\in\cV_x\rM$, recall the $\qq$-vertical lift $\epsilon^{\rmb} =(\epsilon, (\rB_v\epsilon)_x) = (\epsilon, \rB(\epsilon, v)_x)\in \cV_{(x, v)}\cH\rM$. We define $\pi$-horizontal and $\pi$-vertical lifts of horizontal vector fields, as well as $\qq$-vertical lifts of vertical vector fields on $\rM$ similarly. We have $\rmq$ is a bijection between $\cH_x\rM$ and the nullspace of $(\rCa)_{|\cQ_{(x, v)}\cH\rM}$, $\rmp$ is a bijection between $\cH_x\rM$ and the nullspace of $d\pi_{_{|\cQ_{(x, v)}\cH\rM}}$, and $\rmb_{|\cV_x\rM}$ is a bijection between $\cV_x\rM$ and $\cV_{(x, v)}\cH\rM$.
\end{definition}
We have the following bracket formulas for lifts of horizontal vector fields.
\begin{lemma} Let $(\rM,\qq, \cB, \sfg, \cE)$ be a submersed ambient structure of the  Riemannian submersion $\qq:\rM\to\cB$  with horizontal bundle $\cH\rM$. For two horizontal vector fields $\ttX, \ttY$ on $\rM$ we have
  \begin{equation}\label{eq:Dom1}[\ttX^{\rmq}, \ttY^{\rmq}] = (\ttH[\ttX, \ttY])^{\rmq} +
(\RcH_{\ttX, \ttY}\rU)^{\rmp} + (\ttV[\ttX, \ttY])^{\rmb} 
  \end{equation}
  \begin{equation}[\ttX^{\rmq}, \ttY^{\rmp}] = (\nabla_{\ttX}\ttY)^{\rmp}
\end{equation}    
  \begin{equation}[\ttX^{\rmp}, \ttY^{\rmp}] = [\ttX, \ttY]^{\rmp}
\end{equation}    
\end{lemma}
\begin{proof}
Similar to \cref{rem:rmv_rmh}, the first component of $\rD_{\ttX^{\rmq}}\ttY^{\rmq}$ is $\rD_{\ttX}\ttY$, from here the first component of $[\ttX^{\rmq}, \ttY^{\rmq}]$ is $[\ttX, \ttY]$. The second (tangent) component of  $\rD_{\ttX^{\rmq}}\ttY^{\rmq}$ is
  $$-(\rD_{\ttX}\Gammaa)(\ttY, \rU) - \Gammaa(\rD_{\ttX}\ttY, \rU)
  + \Gammaa(\ttY, \Gammaa(\ttX, \rU))
  $$
where $\rU$ is the canonical vector field. Since $\ttX$ is horizontal, $\Gammaa(\ttX, \omega) = \GammaH(\ttX, \omega)$ for $\omega\in \cE$. Changing the role of $\ttY$ and $\ttX$, the tangent component of $[\ttX^{\rmq}, \ttY^{\rmq}]$ is
$$-(\rD_{\ttX}\Gammaa)(\ttY, \rU) -\Gammaa([\ttX, \ttY], \rU) +\GammaH(\ttY, \GammaH(\ttX, \rU)) +
(\rD_{\ttY}\Gammaa)(\ttX, \rU) -\GammaH(\ttX, \GammaH(\ttY, \rU))
$$
We split the left-hand side of \cref{eq:Dom1} to 
$([\ttX, \ttY], -\Gammaa([\ttX, \ttY], \rU))$ and $(0, -(\rD_{\ttX}\Gammaa)(\ttY, \rU) \\ +\GammaH(\ttY, \GammaH(\ttX, \rU)) +(\rD_{\ttY}\Gammaa)(\ttX, \rU) -\GammaH(\ttX, \GammaH(\ttY, \rU))$. As $[\ttX, \ttY]$ is a vector field, using the definition of $\rmq$ and $\rmb$,
$$\begin{gathered}
  ([\ttX, \ttY], -\Gammaa([\ttX, \ttY], \rU)) = (\ttV[\ttX, \ttY]+\ttH[\ttX, \ttY], -\GammaH(\ttH[\ttX, \ttY], \rU)+\rB([\ttX, \ttY], \rU))\\
  =(\ttH[\ttX, \ttY])^{\rmq} + (\ttV[\ttX, \ttY])^{\rmb}
\end{gathered}$$
as $\rB(\ttH[\ttX, \ttY], \rU)=0$. Expand
$$(\rD_{\ttX}\Gammaa)(\ttY, \rU)=(\rD_{\ttX}\GammaH)(\ttY, \rU) +\GammaH((\rD_{\ttX}\ttH)\ttY, \rU)_x - (\rD_{\ttX}\rB)(\ttY, \rU)$$
we need to show the remaining terms below is $\RcH_{\ttX, \ttY}\rU$:
$$\begin{gathered}
- (\rD_{\ttX}\GammaH)(\ttY, \rU)-\GammaH(\ttX, \GammaH(\ttY, \rU))
+(\rD_{\ttY}\GammaH)(\ttX, \rU)  +\GammaH(\ttY, \GammaH(\ttX, \rU))\\
-\GammaH((\rD_{\ttX}\ttH)\ttY, \rU)_x +\
\GammaH((\rD_{\ttY}\ttH)\ttX, \rU)_x  + (\rD_{\ttX}\rB)(\ttY, \rU) -
(\rD_{\ttY}\rB)(\ttX, \rU)
\end{gathered}$$
In the second line, $((\rD_{\ttX}\ttH)\ttY - (\rD_{\ttY}\ttH)\ttX)_x = \ttV_x[\ttX, \ttY]_x$. Compared with \cref{eq:cursubmer}, we need to show
\begin{equation}\label{eq:needtoshow}-\GammaH(\ttV_x[\ttX, \ttY]_x, v)_x -
(\rD_{\ttY}\rB)(\ttX, \rU)_x + (\rD_{\ttX}\rB)(\ttY, \rU)_x= -\GammaH(v, \ttV_x[\ttX, \ttY]_x)_x\end{equation}
as the rightmost expression is $(\rAd_{v} \ttV[\ttX, \ttY])_x$. Note that for $\phi\in\cT_{(x, v)}\cH\rM$, $\rB(\phi, v)_x = (\rD_{\ttV\phi}\ttH)_x v - (\rD_v\ttH)_x\ttV\phi$. Hence, set $\xi = \ttX_x, \eta = \ttY_x$
$$\begin{gathered} - (\rD_{\ttY}\rB)_x(\xi, v) + (\rD_{\ttX}\rB)_x(\eta, v) =\\
-  \rD_\eta\{(\rD_{\ttV\xi}\ttH)v - (\rD_{v}\ttH)\ttV\xi)\} + \rD_\xi\{(\rD_{\ttV\eta}\ttH)v - (\rD_{v}\ttH)\ttV\eta)\}  \\ 
=  -(\rD_{\ttV[\ttX, \ttY]}\ttH)_xv + ((\rD_{v}\ttH)\ttV[\ttX, \ttY])_x
  \end{gathered}
$$
Where we have used $(\rD_{\eta}\ttV)_x\xi - (\rD_{\xi}\ttV)_x\eta = (\rD_{\xi}\ttH)_x\eta - (\rD_{\eta}\ttH)_x\xi = 2\rA_{\xi}\eta = (\ttV[\ttX, \ttY])_x$. Substitute in \cref{eq:needtoshow} and expand $-\GammaH(\ttV_x[\ttX, \ttY]_x, v)_x = (\rD_{(\ttV[\ttX, \ttY])}\ttH)_x v -\ttH\mrGamma(\ttV[\ttX, \ttY], v)_x$ on the left-hand side, we finally get the right-hand side. The rest of the lemma is clear.
\end{proof}
In the following theorem, expressions are evaluated at one point $(x, v)\in\cH\rM$ under consideration and we will omit the point to keep the expressions compact. We write $\langle\rangle_{\sfg}, \langle\rangle_{\hsfg}, \langle\rangle_{\sfGQ}$ for the inner products using the corresponding operators.
\begin{theorem}\label{theo:NaturalSubmerse}Let $\tnabla$ be the Levi-Civita covariant derivative with respect to $\sfGQ$, $\ttQ$ be the horizontal projection from $\cT\cH\rM$ to $\cQ\cH\rM$ in \cref{prop:QHM} and $\nabla^{\cH}=\ttH\nabla$. At $(x, v)\in \cH\rM$, for horizontal vector fields $\ttX$ and $\ttY$ on $\rM$, we have
  \begin{equation}\label{eq:tnablaLift}
  \begin{gathered}
    \ttQ \tnabla_{\ttX^{\rmq}}\ttY^{\rmq} = (\nabla^{\cH}_{\ttX} \ttY)^{\rmq} + \frac{1}{2}(\RcH_{\ttX, \ttY} v)^{\rmp}\\
\ttQ\tnabla_{\ttX^{\rmq}}\ttY^{\rmp} = (\nabla^{\cH}_{\ttX}\ttY)^{\rmp} -\frac{\alpha}{2}(\RcH_{v, \ttY}\ttX)^{\rmq}\\    
    \ttQ\tnabla_{\ttX^{\rmp}}\ttY^{\rmq} = -\frac{\alpha}{2}(\RcH_{v, \ttX}\ttY)^{\rmq}\\
    \ttQ\tnabla_{\ttX^{\rmp}}\ttY^{\rmp} =(\frac{\alpha'}{\alpha}(\sfg( v, \ttX) \ttY + \sfg( v, \ttY)\ttX) +\rF v)^{\rmp}\\
\rF := \frac{(\beta-\alpha')\langle\ttX, \ttY\rangle_{\sfg} + (\beta'-2\alpha'\beta/\alpha)\langle v, \ttX\rangle_{\sfg}\langle v, \ttY\rangle_{\sfg}}{\alpha + \| v\|_{\sfg}^2\beta}
\end{gathered}
  \end{equation}
  where the scalar functions $\alpha, \alpha', \beta, \beta'$ are evaluated at $\|v\|_{\sfg}^2$. Let $\txi=(\xi_{\frm}, \xi_{\frt}),  \teta =(\eta_{\frm}, \eta_{\frt})$ be two tangent vectors in $\cQ_{(x, v)}\cH\rM$. A horizontal Christoffel function $\Gamma_{\sfGQ}^{\cH}$ corresponding to the horizontal component $\ttQ\tnabla$ of the Levi-Civita connection for the metric $\sfGQ$ is given by
  \begin{equation}\label{eq:GammaHsfGDec}\begin{gathered}
    \Gamma_{\sfGQ}^{\cH}(\txi, \teta) =
    \Gamma_{\sfGQ}^{\cH}(\xi_{\frm}^{\rmq}, \eta_{\frm}^{\rmq}) +
    \Gamma_{\sfGQ}^{\cH}(\xi_{\frm}^{\rmq}, (\rCa\teta)^{\rmp}) +\\
\Gamma_{\sfGQ}^{\cH}((\rCa\txi)^{\rmp}, (\eta_{\frm})^{\rmq}) +
\Gamma_{\sfGQ}^{\cH}((\rCa\txi)^{\rmp}, (\rCa\teta)^{\rmp})
\end{gathered}
  \end{equation}
  \begin{equation}\label{eq:GammaHsfGparts}
 \begin{gathered}
       \Gamma_{\sfGQ}^{\cH}(\xi_{\frm}^{\rmq}, \eta_{\frm}^{\rmq}) = (\GammaH(\xi_{\frm}, \eta_{\frm}), -\Gammaa(\GammaH(\xi_{\frm},\eta_{\frm}), v) + (\rD_{\xi_{\frm}}\Gammaa)( \eta_{\frm}, v)-\\\GammaH( \eta_{\frm}, \GammaH(\xi_{\frm}, v)) + \frac{1}{2}\RcH_{\xi_{\frm}, \eta_{\frm}} v)\\
\Gamma_{\sfGQ}^{\cH}(\xi_{\frm}^{\rmq}, (\rCa\teta)^{\rmp}) =
(-\frac{\alpha}{2}\RcH_{ v, \rCa\teta}\xi_{\frm}, \frac{\alpha}{2}\GammaH(\RcH_{ v, \rCa\teta}\xi_{\frm}, v) +
\GammaH(\xi_{\frm}, \rCa\teta))\\
\Gamma_{\sfGQ}^{\cH}((\rCa\txi)^{\rmp}, \eta_{\frm}^{\rmq}) = (-\frac{\alpha}{2}\RcH_{ v,\rCa\txi}\eta_{\frm}, \frac{\alpha}{2}\GammaH(\RcH_{ v,\rCa\txi}\eta_{\frm}, v)
+\GammaH(\rCa\txi, \eta_{\frm}))\\
\Gamma_{\sfGQ}^{\cH}((\rCa\txi)^{\rmp}, (\rCa\teta)^{\rmp}) = (0,\frac{\alpha'}{\alpha}\{ \langle v, \rCa\txi\rangle_{\sfg}\rCa\teta + \langle v, \rCa\teta\rangle_{\sfg} \rCa\txi\} +\rF v)
    \end{gathered}
  \end{equation}
  where $\rF$ is evaluated from \cref{eq:tnablaLift} with $\rCa\txi, \rCa\teta$ in place of $\ttX$, $\ttY$.
\end{theorem}
Again, the case $\alpha = 1, \beta = 0$ is the case of the Sasaki metric on $\cT\cB$, the case $\alpha=\beta = (1+t)^{-1}$ is that of the Cheeger-Gromoll metric.
\begin{proof}First, we have the following relations for three horizontal vector fields $\ttX, \ttY, \ttZ$ on $\rM$. We will not repeat the proof (identical to that of Lemma 6.2 of \cite{GudKap}, with the opposite sign convention for $\RcH$, using the Koszul formula):
  \begin{equation}\label{eq:qqq}\langle \tnabla_{\ttX^{\rmq}}\ttY^{\rmq}, \ttZ^{\rmq}\rangle_{\sfGQ} = \langle \nabla^{\cH}_{\ttX}\ttY, \ttZ\rangle_{\sfg}
  \end{equation}
  \begin{equation}\label{eq:qqp}2\langle \tnabla_{\ttX^{\rmq}}\ttY^{\rmq}, \ttZ^{\rmp}\rangle_{\sfGQ} = \langle (\RcH_{\ttX, \ttY} v)^{\rmp}, \ttZ^{\rmp}\rangle_{\sfGQ}
    \end{equation}
  \begin{equation}\label{eq:qpq}2\langle \tnabla_{\ttX^{\rmq}}\ttY^{\rmp}, \ttZ^{\rmq}\rangle_{\sfGQ} = -\langle (\RcH_{\ttX, \ttZ} v)^{\rmp}, \ttY^{\rmp}\rangle_{\sfGQ}
    \end{equation}
  \begin{equation}\label{eq:qpp}2\langle \tnabla_{\ttX^{\rmq}}\ttY^{\rmp}, \ttZ^{\rmp}\rangle_{\sfGQ} = \rD_{\ttX^{\rmq}}\langle\ttY^{\rmp}, \ttZ^{\rmp}\rangle_{\sfGQ} - \langle\ttY^{\rmp}, (\nabla^{\cH}_{\ttX}\ttZ)^{\rmp} \rangle_{\sfGQ} +\langle\ttZ^{\rmp}, (\nabla^{\cH}_{\ttX}\ttY)^{\rmp} \rangle_{\sfGQ}
  \end{equation}
      \begin{equation}2\label{eq:pqq}\langle\tnabla_{\ttX^{\rmp}}\ttY^{\rmq}, \ttZ^{\rmq}\rangle_{\sfGQ} = -\langle (\RcH_{\ttY, \ttZ}v)^{\rmp}, \ttX^{\rmp}\rangle_{\sfGQ}
    \end{equation}
      \begin{equation}2\label{eq:pqp}\langle\tnabla_{\ttX^{\rmp}}\ttY^{\rmq}, \ttZ^{\rmp}\rangle_{\sfGQ} = \rD_{\ttY^{\rmq}}\langle \ttZ^{\rmp}, \ttX^{\rmp} \rangle_{\sfGQ} - \langle\ttZ^{\rmp}, (\nabla^{\cH}_{\ttY}\ttX)^{\rmp}\rangle_{\sfGQ} - \langle\ttX^{\rmp}, (\nabla^{\cH}_{\ttY}\ttZ)^{\rmp}\rangle_{\sfGQ}
    \end{equation}
    \begin{equation}\label{eq:ppq}2\langle\tnabla_{\ttX^{\rmp}}\ttY^{\rmp}, \ttZ^{\rmq}\rangle_{\sfGQ} = -\rD_{\ttZ^{\rmq}}\langle\ttX^{\rmp}, \ttY^{\rmp} \rangle_{\sfGQ} + \langle\ttY^{\rmp}, (\nabla^{\cH}_{\ttZ}\ttX)^{\rmp}\rangle_{\sfGQ} +\langle\ttX^{\rmp}, (\nabla^{\cH}_{\ttZ}\ttY)^{\rmp} \rangle_{\sfGQ}
    \end{equation}
  \begin{equation}\label{eq:ppp}2\langle\tnabla_{\ttX^{\rmp}}\ttY^{\rmp}, \ttZ^{\rmp}\rangle_{\sfGQ} = \rD_{\ttX^{\rmp}}\langle\ttY^{\rmp}, \ttZ^{\rmp}\rangle_{\sfGQ} +\rD_{\ttY^{\rmp}}\langle\ttZ^{\rmp}, \ttX^{\rmp}\rangle_{\sfGQ} -\rD_{\ttZ^{\rmp}}\langle\ttX^{\rmp}, \ttY^{\rmp}\rangle_{\sfGQ}
  \end{equation}
  The first equality of \cref{eq:tnablaLift} follows from \cref{eq:qqq,eq:qqp} as $\sfGQ$ and $\sfg$ are related by a submersion. The remaining equalities are proved in a similar way to Proposition 8.2 of \cite{GudKap} and of theorem 2 of \cite{Abbassi2005}, which we present below. From \cref{eq:qpq}, with $\alpha, \beta$ evaluated at $\| v\|^2_{\sfg}$
  $$2\langle \tnabla_{\ttX^{\rmq}}\ttY^{\rmp}, \ttZ^{\rmq}\rangle_{\sfGQ} = \langle \RcH_{\ttZ, \ttX} v, \ttY\rangle_{\hsfg} = \alpha\langle \RcH_{\ttZ, \ttX} v, \ttY \rangle_{\sfg} +\beta\langle\ttY, v \rangle_{\sfg}\langle \RcH_{\ttZ, \ttX} v, v \rangle_{\sfg}
  $$
  which is $-\alpha\langle \RcH_{v, \ttY}\ttX, \ttZ \rangle_{\sfg}$, using Bianchi's identities. This gives us the $\rmq$-component of $\tnabla_{\ttX^{\rmq}}\ttY^{\rmp}$. From metric compatibility of $\nabla$, \cref{eq:hsfg_compa} and property of projection
  $$\rD_{\ttX^{\rmq}}\langle\ttY^{\rmp}, \ttZ^{\rmp}\rangle_{\sfGQ} =
  \langle(\nabla_{\ttX}\ttY)^{\rmp}, \ttZ^{\rmp}\rangle_{\sfGQ} + \langle\ttY^{\rmp}, (\nabla_{\ttX}\ttZ)^{\rmp}\rangle_{\sfGQ} =
\langle(\nabla^{\cH}_{\ttX}\ttY)^{\rmp}, \ttZ^{\rmp}\rangle_{\sfGQ} + \langle\ttY^{\rmp}, (\nabla^{\cH}_{\ttX}\ttZ)^{\rmp}\rangle_{\sfGQ}$$
From here and \cref{eq:qpp}, we get $\rmp$-component of $\tnabla_{\ttX^{\rmq}}\ttY^{\rmp}$ because
  $$
  2\langle \tnabla_{\ttX^{\rmq}}\ttY^{\rmp}, \ttZ^{\rmp}\rangle_{\sfGQ} =\langle(\nabla^{\cH}_{\ttX}\ttY)^{\rmp}, \ttZ^{\rmp}\rangle_{\sfGQ} + \langle\ttY^{\rmp}, (\nabla^{\cH}_{\ttX}\ttZ)^{\rmp}\rangle_{\sfGQ}
  - \langle\ttY^{\rmp}, (\nabla^{\cH}_{\ttX}\ttZ)^{\rmp} \rangle_{\sfGQ} +\langle\ttZ^{\rmp}, (\nabla^{\cH}_{\ttX}\ttY)^{\rmp} \rangle_{\sfGQ}
  $$
We will skip the calculation of $\tnabla_{\ttX^{\rmp}}\ttY^{\rmq}$ as it is similar. Expanding \cref{eq:ppq}, using $\sfGQ$-metric compatibility then use the just proved expressions for $\tnabla_{\ttZ^{\rmq}}\ttX^{\rmv}, \tnabla_{\ttZ^{\rmq}}\ttY^{\rmv}$, note that the $\rmq$ and $\rmv$ components are orthogonal
$$\begin{gathered}2\langle \tnabla_{\ttX^{\rmp}}\ttY^{\rmp}, \ttZ^{\rmq}\rangle_{\sfGQ} = -\langle(\nabla^{\cH}_{\ttZ}\ttX)^{\rmp}, \ttY^{\rmp} \rangle_{\sfGQ}
  +\frac{\alpha}{2}\langle(\RcH_{v, \ttX}\ttZ)^{\rmq}, \ttY^{\rmp} \rangle_{\sfGQ}
-\langle\ttX^{\rmp}, (\nabla^{\cH}_{\ttZ}\ttY)^{\rmp} \rangle_{\sfGQ}\\
+\frac{\alpha}{2}\langle \ttX^{\rmp}, (\RcH_{v, \ttY}\ttZ)^{\rmq} \rangle_{\sfGQ}  + \langle\ttY^{\rmp}, (\nabla^{\cH}_{\ttZ}\ttX)^{\rmp}\rangle_{\sfGQ} +\langle\ttX^{\rmp}, (\nabla^{\cH}_{\ttZ}\ttY)^{\rmp} \rangle_{\sfGQ} =0
  \end{gathered}$$
Next, for any real function $f$, $\rD_{\ttX^{\rmp}} f(\|\rU\|^2_{\sfg}) = 2f'(\|\rU\|^2_{\sfg})\langle \ttX, \rU\rangle_{\sfg}$. Write $\alpha, \beta, \alpha', \beta'$ for their values at $\|v\|^2_{\sfg}$ at the horizontal tangent point $(x, v)$:
$$\begin{gathered}\rD_{\ttX^{\rmp}}\langle\ttY^{\rmp}, \ttZ^{\rmp}\rangle_{\sfGQ} =
\rD_{\ttX^{\rmp}}\{\alpha(\|v\|^2_{\sfg})\langle\ttY, \ttZ\rangle_{\sfg} +
\beta(\|v\|^2_{\sfg})\langle\ttY, v\rangle_{\sfg}\langle \ttZ, v\rangle_{\sfg}
\} =\\
\{2\alpha'\langle\ttX, v \rangle_{\sfg}\langle\ttY, \ttZ\rangle_{\sfg} + 
2\beta'\langle\ttX, v \rangle_{\sfg}\langle\ttY, v\rangle_{\sfg}\langle \ttZ, v\rangle_{\sfg} +
\beta\{\langle\ttY, \ttX\rangle_{\sfg}\langle \ttZ, v\rangle_{\sfg} +
\langle\ttY, v\rangle_{\sfg}\langle \ttZ, \ttX\rangle_{\sfg}\}
\end{gathered}
$$
By \cref{eq:hsfginv}, $\hsfg^{-1}\sfg\omega = \frac{1}{\alpha}\omega - \frac{\beta}{\alpha(\alpha +\|v\|^2_{\sfg}\beta)}\langle v,\sfg\omega \rangle_{\cE}v$, in particular, $\hsfg^{-1}\sfg v = \frac{1}{\alpha + \|v\|^2_{\sfg}\beta}v$. Thus, from \cref{eq:ppp}, with two permutations of the above equality, the $\rmp$-component of $\ttQ\tnabla_{\ttX^{\rmp}}\ttY^{\rmp}$ is
$$\begin{gathered}\frac{1}{2}\hsfg^{-1}\sfg\{
  2\alpha'\langle\ttX, v \rangle_{\sfg}\ttY + 
2\beta'\langle\ttX, v \rangle_{\sfg}\langle\ttY, v\rangle_{\sfg}v +
\beta\langle\ttY, \ttX\rangle_{\sfg} v +
\beta\langle\ttY, v\rangle_{\sfg}\ttX +\\
  2\alpha'\langle\ttY, v \rangle_{\sfg}\ttX + 
2\beta'\langle\ttY, v \rangle_{\sfg}\langle\ttX, v\rangle_{\sfg}v +
\beta\langle\ttX, \ttY\rangle_{\sfg} v +
\beta\langle\ttX, v\rangle_{\sfg}\ttY
-\\
2\alpha'\langle\ttY, \ttX\rangle_{\sfg}v  -
2\beta'\langle\ttY, v\rangle_{\sfg}\langle \ttX, v\rangle_{\sfg}v -
\beta\langle \ttX, v\rangle_{\sfg}\ttY -
\beta\langle\ttY, v\rangle_{\sfg}\ttX\}=\\
\hsfg^{-1}\sfg\{\alpha'\langle\ttX, v \rangle_{\sfg}\ttY + \alpha'\langle\ttY, v \rangle_{\sfg}\ttX + (\beta'\langle\ttX, v \rangle_{\sfg}\langle\ttY, v\rangle_{\sfg} +
\beta\langle\ttY, \ttX\rangle_{\sfg}-\alpha'\langle \ttY, \ttX\rangle_{\sfg})v
\}=\\
\hsfg^{-1}\sfg\{\alpha'\langle\ttX, v \rangle_{\sfg}\ttY + \alpha'\langle\ttY, v \rangle_{\sfg}\ttX\} + (\beta'\langle\ttX, v \rangle_{\sfg}\langle\ttY, v\rangle_{\sfg} +
\beta\langle\ttY, \ttX\rangle_{\sfg}-\alpha'\langle \ttY, \ttX\rangle_{\sfg})\hsfg^{-1}\sfg v\\
=\frac{\alpha'}{\alpha}(\langle\ttX, v \rangle_{\sfg}\ttY +  \langle\ttY, v \rangle_{\sfg}\ttX) -\frac{\beta\alpha'}{\alpha(\alpha +\|v\|^2_{\sfg}\beta)}(2\langle\ttX, v \rangle_{\sfg}\langle\ttY, v\rangle_{\sfg})v\\
+\frac{1}{\alpha + \|v\|^2_{\sfg}\beta}(\beta'\langle\ttX, v \rangle_{\sfg}\langle\ttY, v\rangle_{\sfg} +
(\beta-\alpha')\langle\ttX, \ttY\rangle_{\sfg})v\\
=\frac{\alpha'}{\alpha}(\langle\ttX, v \rangle_{\sfg}\ttY + \langle\ttY, v \rangle_{\sfg}\ttX) +
\frac{
 (\beta'-2\beta\alpha'/\alpha)\langle\ttX, v \rangle_{\sfg}\langle\ttY, v\rangle_{\sfg}
+ (\beta-\alpha')\langle\ttX, \ttY\rangle_{\sfg}}{\alpha +\|v\|^2_{\sfg}\beta}v
  \end{gathered}
$$
This completes the proof of \cref{eq:tnablaLift}. To compute $\Gamma_{\sfG}^{\cH}$, we have \cref{eq:GammaHsfGDec} by linearity. Fix $(x, v)\in \cH\rM$. Let $\qtxi, \rtxi, \qteta, \rteta$ be horizontal vector fields on $\rM$ by defining $\qtxi(y) = \ttH_y\xi_{\frm}, \rtxi(y)= \ttH_y(\xi_{\frt}+\GammaH(\xi_{\frm}, v)_x), \qteta(y) = \ttH_y\eta_{\frm}, \rteta(y)= \ttH_y(\eta_{\frt}+\GammaH(\eta_{\frm}, v)_x)$ for $y\in\rM$, then $\xi_{\frm}^{\rmh}, (\rCa_{(x, v)}\txi)^{\rmv}, \eta_{\frm}^{\rmh}, (\rCa_{(x, v)}\teta)^{\rmv}$ are $\qtxi^{\rmh}, \rtxi^{\rmv}, \qteta^{\rmh}, \rteta^{\rmv}$ evaluated at $(x, v)$. We have
$$\begin{gathered}(\rD_{\qtxi^{\rmh}}\qteta^{\rmh})_{(x, v)} = \rD_{\qtxi^{\rmh}}(\qteta, -\Gammaa(\qteta, \rU))_{(x, v)} =\\
  ((\rD_{\qtxi}\qteta)_x, -(\rD_{\xi_{\frm}}\Gammaa)_x(\eta_{\frm}, v)  -\Gammaa((\rD_{\qtxi}\qteta)_x, v)_x+\Gammaa( \eta_{\frm}, \Gammaa(\xi_{\frm}, v))_x)
\end{gathered}$$
Using \cref{eq:tnablaLift}, $(\ttQ\tilde{\nabla}_{\qtxi^{\rmh}}\qteta^{\rmh})_{(x, v)} = (\nabla^{\cH}_{\qtxi}\qteta, -\Gammaa(\nabla^{\cH}_{\qtxi}\qteta, v) +\frac{1}{2}\RcH_{\xi_{\frm}, \eta_{\frm}}v)_x$, thus
$$\begin{gathered}
\Gamma^{\cH}_{\sfGQ}(\xi_{\frm}^{\rmh}, \eta_{\frm}^{\rmh})_{(x, v)} =
  (\ttQ\tilde{\nabla}_{\qtxi^{\rmh}}\qteta^{\rmh} - \rD_{\qtxi^{\rmh}}\qteta^{\rmh})_{(x, v)}=
  ((\nabla^{\cH}_{\qtxi}\qteta - \rD_{\qtxi^{\rmh}}\qteta)_{(x, v)},\\
  -\Gammaa(\nabla^{\cH}_{\qtxi}\qteta - \rD_{\qtxi^{\rmh}}\qteta, v) +\frac{1}{2}\RcH_{\xi_{\frm}, \eta_{\frm}}v +
(\rD_{\xi_{\frm}}\Gammaa)(\eta_{\frm}, v)-\GammaH(\eta_{\frm}, \GammaH(\xi_{\frm}, v)))
  \\
  =(\GammaH(\xi_{\frm}, \eta_{\frm}),\\
  -\Gammaa(\GammaH(\xi_{\frm}, \eta_{\frm}), v) +
(\rD_{\xi_{\frm}}\Gammaa)(\eta_{\frm}, v)-\GammaH(\eta_{\frm}, \GammaH(\xi_{\frm}, v)) +
  \frac{1}{2}\RcH_{\xi_{\frm}, \eta_{\frm}}v)
\end{gathered}
$$
where expressions are evaluated at $(x, v)$. This gives us the first equation in \cref{eq:GammaHsfGparts}. Similarly, the expression for $\Gamma^{\cH}_{\sfGQ}(\xi_{\frm}^{\rmh}, (\rCa\teta)^{\rmv})_{(x,v)}$ follows from
$$(\ttQ\nabla_{\qtxi^{\rmh}}\rteta^{\rmv})_{(x, v)} =(-\frac{\alpha}{2}\rR_{v, \rCa\teta}\xi_{\frm},
 \frac{\alpha}{2}\GammaH(\RcH_{v, \rCa\teta}\xi_{\frm}, v) + (\nabla^\cH_{\qtxi}\rteta))_x$$
$$(\rD_{\qtxi^{\rmh}}\rteta^{\rmv})_{(x, v)} = (0, (\rD_{\qtxi}\rteta)_x)$$
 Then, the formula for $\Gamma^{\cH}_{\sfGQ}((\rCa\txi)^{\rmv}, \eta_{\frm}^{\rmh})$ follows from
 $$(\rD_{\rtxi^{\rmv}}\qteta^{\rmh})_{(x, v)} = (0, -\GammaH(\rCa\txi, \eta_{\frm}))_{(x, v)}$$
 $$(\ttQ\nabla_{\rtxi^{\rmv}}\qteta^{\rmh})_{(x, v)} = -\frac{\alpha}{2}(\RcH_{v,\rCa\xi_{\frm}}\teta)^{\rmh}_{(x, v)}$$
 and $(\rD_{\rtxi^{\rmv}}\rteta^{\rmv})_{(x, v)}=0$ gives us the formula for $\Gamma^{\cH}_{\sfGQ}((\rCa\txi)^{\rmv}, (\rCa\teta)^{\rmv})$.
\end{proof}  
\section{Application to Grassmann manifold}\label{sec:grass}
The Grassmann manifold could be considered as the simplest example of a flag manifold, that we have realized as a quotient of the orthogonal group. Following \cite{Edelman_1999}, we will construct it as a quotient of the Stiefel manifold. The horizontal projection and Levi-Civita connection are well-known and will be reviewed briefly. We will show the computation of curvature, Jacobi fields and horizontal bundle metric in this section.

Let $n > p > 0$ be two positive integers. Recall a Stiefel manifold could be considered as a submanifold of $\R^{n\times p}$ of matrices satisfying the equation $Y^{\sfT}Y = \dI_p$, $Y\in \R^{n\times p}$. We will use the embedded metric on the Stiefel manifold $\sfg\omega = \omega$ for $\omega\in\cE=\R^{n\times p}$.
For an ambient vector $\omega \in \cE = \R^{n\times p}$, the projection to the tangent space of Stiefel manifold at $Y$ is $\Pi\omega = \omega-\frac{1}{2}(YY^{\ft}\omega + Y\omega^{\ft}Y)$ \cite{Edelman_1999}, and $\mrGamma = 0$.

The Grassmann manifold $\Gr{p}{n}$ could be considered as the quotient of $\St{p}{n}$ by a right action of the orthogonal group $\OO(p)$, that is under the equivalence $YQ\sim Y$ for $Y\in \St{p}{n}$ and $Q \in \OO(p)$. We will use the notation $\lb Y\rb$ to denote the equivalent class of $Y$. Thus, in our convention, $\rM = \St{p}{n}$ and $\cB= \Gr{p}{n}$, and we have a submersion $\qq:\rM\to\cB$. In this submersion, the vertical space consists of vectors $Yb$, where $b = -b^{\sfT}\in \oo(p)$. The vertical projection is therefore $\ttV\omega = \frac{1}{2}Y(Y^{\sfT}\omega-\omega^{\sfT}Y)$ using \cref{lem:projprop}, with the constant metric $\sfg\omega = \omega$ and the map $\rN:\oo(p)\to\R^{n\times p}=:\cE$, $\rN:b\mapsto Yb$ for $b\in\oo(p)$, and $\rN^{\sfT}\omega = \frac{1}{2}(Y^{\sfT}\omega - \omega^{\sfT}Y)$. Hence, the projection to the horizontal space is $\ttH\omega = (\Pi - \ttV)\omega = \omega-YY^{\ft}\omega$. A horizontal vector $\eta$ satisfies $Y^{\sfT}\eta = 0$. We try to keep the formulas compact and omit explicit subscripting $Y$ for $\Pi, \ttH$ and $\ttV$.

If $\xi$ is tangent to the Stiefel manifold and $\omega \in \cE$, $\GammaH(\xi, \omega) = -(\rD_{\xi}\ttH)\omega$
$$\GammaH(\xi, \omega) = Y\xi^{\ft}\omega +\xi Y^{\ft}\omega$$
and for horizontal vectors $\xi, \eta, \phi$, $Y^{\ft}\xi = Y^{\ft}\eta= Y^{\ft}\phi =0$, from \cref{eq:rAV}
\begin{equation}\label{eq:rA_grass}\rA_{\xi}\eta =-(\rD_{\xi}\ttV)\eta = -\frac{1}{2}(Y\xi^{\ft}\eta +\xi Y^{\ft}\eta -\xi\eta^{\ft} Y - Y\eta^{\ft}\xi)
=-\frac{1}{2}Y(\xi^{\ft}\eta - \eta^{\ft}\xi)\end{equation}
$$\rAd_{\phi}\rA_{\xi}\eta = -\GammaH(\phi,-\frac{1}{2}Y(\xi^{\ft}\eta - \eta^{\ft}\xi))=
(Y\phi^{\ft} +\phi Y^{\ft})\{\frac{1}{2}Y(\xi^{\ft}\eta - \eta^{\ft}\xi)\}=\frac{1}{2}(\phi\xi^{\ft}\eta-\phi\eta^{\ft}\xi)
$$
$$(\rD_{\xi}\GammaH)(\eta, \phi) = \xi\eta^{\ft}\phi +\eta\xi^{\ft}\phi$$
Thus $-(\rD_{\xi}\GammaH)(\eta, \phi) + (\rD_{\eta}\GammaH)(\xi, \phi) =0$ and
$\GammaH(\xi, \GammaH(\eta, \phi)) = (Y\xi^{\ft} +\xi Y^{\ft})(Y\eta^{\ft}\phi +\eta Y^{\ft}\phi)= \xi\eta^{\ft}\phi$. We now get the classical curvature formula for $\Gr{p}{n}$:
\begin{proposition} 
Let $\xi, \eta, \phi$ be three horizontal vectors at $Y\in\St{p}{n}$ as horizontal lifts of tangent vectors at $\lb Y\rb \in \Gr{p}{n}$. The lift of the Riemannian curvature tensor of $\Gr{p}{n}$ to $\St{p}{n}$ is given by:
\begin{equation}\label{eq:r13_grass}\RcH_{\xi, \eta}\phi= -\xi\eta^{\ft}\phi + \eta\xi^{\ft}\phi +\phi\xi^{\ft}\eta-\phi\eta^{\ft}\xi\end{equation}
If $(\Yperp | Y) \in \OO(n)$ and $\xi = \Yperp B_1, \eta = \Yperp B_2$, we have $\htK(\xi, \eta) = \langle\hcR_{\xi, \eta}\xi, \eta\rangle_{\sfg}$ is
  \begin{equation}\label{eq:sec_grass_cur}
    \begin{gathered}
    \htK_{\cH}(\xi, \eta) = \TrR(B_1B_1^{\ft}B_2B_2^{\ft} + B_2B_1^{\ft}B_1B_2^{\ft}  - 2B_1B_2^{\ft}B_1B_2^{\ft}) \\= ||B_2^{\ft}B_1 - B_1^{\ft}B_2||_F^2 + ||B_1B_2^{\ft} - B_2B_1^{\ft}||_F^2
\end{gathered}    
\end{equation}  
\end{proposition}
Note that the expression for $\htK_{\cH}$ is dependent on $\Yperp\Yperp^{\sfT} = \dI_n - YY^{\sfT}$, not on the choice of $\Yperp$. Without using \cref{eq:cursubmer} or \cref{eq:flag_curv}, the curvatures could be derived from the theory of symmetric spaces, where the above expression comes from the Lie bracket $[\tilde{B}_3[\tilde{B}_1\tilde{B}_2]]$ in the embedding $B_i\mapsto \tilde{B}_i=\begin{bmatrix}0 & -B_i^{\ft}\\B_i & 0\end{bmatrix}\in \oo(n)$.
\begin{proof}Equation (\ref{eq:r13_grass}) follows from the preceding calculation and \cref{eq:cursubmer}, which reduces to $2\rAd_{\phi}\rA_{\xi}\eta -\GammaH(\xi\GammaH(\eta, \phi)) + \GammaH(\eta, \GammaH(\xi, \phi))$, or $\phi(\xi^{\ft}\eta - \eta^{\ft}\xi) -\xi\eta^{\sfT}\phi + \eta\xi^{\sfT}\phi$.
  
  Let $\xi=Y_{\perp}B_1$, $\eta=Y_{\perp}B_2, \phi=Y_{\perp}B_3$, $B_1, B_2, B_3\in \R^{(n-p)\times p}$. Then
  $$Y_{\perp}\RcH_{\xi, \eta} \phi= -B_1B_2^{\ft}B_3 + B_2B_1^{\ft}B_3 + B_3B_1^{\ft}B_2-B_3B_2^{\ft}B_1$$
  The sectional curvature numerator in \cref{eq:sec_grass_cur} follows from a substitution.  
\end{proof}
We now describe the horizontal bundle $\cH\St{p}{n}\subset \cT\St{p}{n}$ of the submersion $\qq:\St{p}{n}\to\Gr{p}{n}$ and its structure as in \cref{sec:subm_tangent}. $\cH\St{p}{n}$ could be identified with a submanifold of $(\R^{n\times p})^2$ of pairs of matrices $(Y, \eta)$ satisfying $Y^{\sfT}Y=\dI_p$, $Y^{\sfT}\eta = 0$. Its tangent bundle could be considered as a quadruple $(Y, \eta, \Delta_{\frm}, \Delta_{\frt})$ with $\eta$ a horizontal vector, $\Delta_{\frm}$ a Stiefel-tangent vector, thus $Y^{\sfT}\Delta_{\frm} + \Delta^{\sfT}_{\frm}Y = 0$, and $\Delta_{\frm}^{\sfT} \eta + Y^{\sfT}\Delta_{\frt} = 0$. For $b\in\oo(p)$, $(\rD_{Yb}\ttH)\eta=0$ and the operator $\rB$ in \cref{prop:QHM} is
$$\rB(Yb, \eta) = (\rD_{Yb}\ttH)\eta - (\rD_{\eta}\ttH) Yb = \eta Y^{\sfT}Y b =   \eta b$$
as expected. The extension $\rB(\phi, \eta) = \eta Y^{\sfT}\phi$ satisfies $\rB(\phi, \eta) = 0$ if $\phi$ is horizontal and $\rB(Yb, \eta)= \eta b$, and we will use this expression to extend $\rB$ to $\cE^2$. The vertical bundle $\cV\cH\rM$ consists of quadruples $(Y, \eta, Yb, \eta b)$, while $\cQ\cH\rM$ consists of tuples $(Y, \eta, \delta_{\frm}, \delta_{\frt})$ with the same relation as $(\Delta_{\frm}, \Delta_{\frt})$ but now $\delta_{\frm}$ is horizontal. The connection map sends ($\delta_{\frm}, \delta_{\frt})$ to $\delta_{\frc} = \delta_{\frt} + Y\delta_{\frm}^{\sfT}\eta$. A horizontal vector in $\cQ\cH\rM$ is thus of the form $(Y, \eta, \delta_{\frm}, \delta_{\frc} -Y\delta_{\frm}^{\sfT}\eta)$ for three tangent vectors $\eta, \delta_{\frm}, \delta_{\frc}$.

In this Stiefel coordinate, from \cref{prop:frjH} and \cref{eq:rA_grass}, the canonical flip maps $(Y, \eta, \delta_{\frm}, \delta_{\frt})$ to $(Y, \delta_{\frm}, \eta, \delta_{\frt} + Y(\delta_{\frm}^{\sfT}\eta-\eta^{\sfT}\delta_{\frm}))$. From \cref{theo:jacsub}
\begin{proposition} Let $\csr(z)$ and $\ssr z$ be analytic continuations of $\cos z^{1/2}$ and $z^{-1/2} \sin z^{1/2}$ to entire functions. For $Y\in\St{p}{n}$ and $\eta$ a tangent vector at $Y$, a geodesics for the Grassmann manifold $\Gr{p}{n}$ lifts to a horizontal geodesics $\gamma(t)$ on the Stiefel manifold $\St{p}{n}$, with $(\gamma(0), \dot{\gamma}(0)) = (Y, \eta)$ as:
  \begin{equation}\label{eq:exp_grass}
\gamma(t) = \Exp_{Y}t\eta = Y \csr t^2\eta^{\ft}\eta + t\eta\ssr t^2 \eta^{\ft}\eta
  \end{equation}
The tangent component of the horizontal lift $\frJ^{\cH}(t)=(\gamma(t), J^{\cH}(t))$ of a Jacobi field on $\Gr{p}{n}$ to $\St{p}{n}$ with initial data $\dot{\frJ}^{\cH}(0) = (Y, \nu_{\frm}, \eta, \nu_{\frt})\in\cQ\cH\St{p}{n}$ is
  \begin{equation}
  \begin{gathered}
J^{\cH}(t) =  (\dI_n - \gamma(t)\gamma(t)^{\ft})\{
  \nu_{\frm}\csr t^2\eta^{\ft}\eta +
t\delta_{\frt}\ssr(t^2\eta^{\ft}\eta)
+\\ t^2 Y \frL_{\csr}(t^2\eta^{\ft}\eta, \eta^{\ft}\delta_{\frt} + \delta_{\frt}^{\ft}\eta) + t^3\eta \frL_{\ssr}(t^2\eta^{\ft}\eta, \eta^{\ft}\delta_{\frt} + \delta_{\frt}^{\ft}\eta)\}
  \end{gathered}
  \end{equation}
where $\delta_{\frt} = \nu_{\frt} - Y(\nu_{\frm}^{\sfT}\eta - \eta^{\sfT}\nu_{\frm})$ and for an analytic function $f(z) = \sum_{i=0}^\infty f_i z^i$, $\frL_f(A, E)$ denotes the Fr{\'e}chet derivative $\sum_{i=1}^\infty f_i (\sum_{b+c=i-1}A^bEA^c)$.
\end{proposition}
We will review Fr{\'e}chet derivatives in \cref{sec:frechet}. As mentioned, when $f(x)=\exp(x)$, $\frL_{\exp}(A, E) = \exp(tA)((1-\exp(-x))/x)_{x=t\ad_A}E$. Fr{\'e}chet derivatives have the advantage that it is defined for all differentiable functions, and for analytic functions it has about three times the computational complexity of evaluating $f(A)$, so for any practical purpose it could be considered as a closed-form expression. We hope the expression of Jacobi fields in terms of Fr{\'e}chet derivatives is also useful theoretically. The Fr{\'e}chet derivatives $\frL_{\csr}$ and $\frL_{\ssr}$ are not available in numerical packages but are simple to implement. It is easy to verify the initial condition $J^{\cH}(0) = (\dI_n-YY^{\sfT})\nu_{\frm} = \nu_{\frm}$, and $\dot{J}^{\cH}(0) = -(Y\eta^{\sfT} +\eta Y^{\sfT})\nu_{\frm} + (\dI_n-YY^{\sfT})\delta_{\frt} =\delta_{\frc} - Y\eta^{\sfT}\nu_{\frm} = \delta_{\frt} + Y\nu_{\frm}^{\sfT}\eta - Y\eta^{\sfT}\nu_{\frm}=\nu_{\frt}$.

\begin{proof}The formula for $\gamma(t)$ is proved in \cite{NguyenGeodesic}, or follows from either direct substitution of \cref{eq:exp_grass} to the geodesic equation for the lift of the Grassmann, which is $\ddot{\gamma} + \gamma\dot{\gamma}^\sfT\dot{\gamma} = 0$ (a calculation similar to verification the geodesic equation of a sphere), or from the first $p$ columns of $U\exp t\hat{\eta}$ (the geodesic for $SO(n)$),  where $U = (Y | \Yperp^{\sfT})$, $\hat{\eta} = \begin{pmatrix} 0 & -\eta^{\sfT}\\ \eta & 0\end{pmatrix}$, breaking $\exp t\hat{\eta}$ to even and odd powers (note the horizontal geodesics are the same for the induced metric $\sfg_i\eta = \eta$ and the canonical metric $\sfg_c\eta = \eta - \frac{1}{2}YY^{\sfT}\eta$, as $Y^{\sfT}\eta = 0$).

The expression for $J^{\cH}(t)$ is just the horizontal projection of the directional derivative of $\gamma(Y, \eta; t)$ in the direction $(\delta_{\frm}, \delta_{\frt})$ defined by the canonical flip.
\end{proof}

For the natural metric on $\cH\St{p}{n}$ corresponding to the submersion to the Grassmann manifold, at a point $(Y, V)\in \cH\St{p}{n}$, for $\omega\in \R^{n\times p}$ we set $\Gammaa(\omega, V)_Y = \GammaH(\ttH_Y\omega, V) - \rB(\omega, V) = Y\omega^{\sfT}V - V Y^{\sfT}\omega$ from \cref{eq:Hproj} and extend it to an operator on $\cE^2$. This expression could be used to evaluate the metric $\sfGQ$ in \cref{eq:Hproj}, the projection in \cref{eq:bundl_projH}. To evaluate the Christoffel function at $(Y, V)\in \cH\St{p}{n}$ for horizontal vectors $\txi=(\xi_{\frm}, \xi_{\frt}), \teta=(\eta_{\frm}, \eta_{\frt})\in \cQ_{(Y, V)}\cH\St{p}{n}$, in \cref{eq:GammaHsfGparts}, with the curvature known, we can use
\begin{equation}(\rD_{\xi_{\frm}}\Gammaa)( \eta_{\frm}, V)=\xi_{\frm}\eta_{\frm}^{\sfT}V - V \xi_{\frm}^{\sfT}\eta_{\frm}\end{equation}
\begin{equation}\Gammaa(\GammaH(\xi_{\frm},\eta_{\frm}), V) = - V \xi_{\frm}^{\sfT}\eta_{\frm}\end{equation}
\section{Discussion} We have demonstrated differential geometric measures of a Riemannian manifold could be computed effectively using a metric operator if the manifold is embedded in a Euclidean space, or if it is a submersed image of such manifold, and have derived several new results using this approach. We believe the approach could be effective for other types of geometries, for example, Finsler geometry or generalized complex geometry. Jacobi field and tangent bundle metrics appear in the problem of geodesic regression in computer vision \cite{MachaLeite,Niethammer2011,Fletcher2013}, thus our present work presents an approach to evaluate them for common manifolds. We hope researchers, both applied and pure mathematics will find the approach useful in their future works.
\begin{appendix}
\section{Fr{\'e}chet Derivative}\label{sec:frechet}Recall \cite{Higham,Mathias,Havel} if $f(A)=\sum_{i=0}^{\infty}f_iA^i$ is a power series with scalar coefficients and $A$ is a square matrix, then the Fr{\'e}chet derivative $\frL_f(A, E) = \lim_{h\to 0}\frac{1}{h}(f(A+hE) - f(A))$ in direction $E$ could be expressed under standard convergence condition as
$$\frL_f(A, E) = \sum_{i=0}^{\infty} f_i \sum_{a+b=i-1}A^aEA^b$$
If $\hat{A} = \begin{pmatrix} A & E \\ 0 & A\end{pmatrix}$ then $f(\hat{A}) = \begin{pmatrix} f(A) & \frL_f(A, E) \\ 0 & f(A)\end{pmatrix}$, this could be used to show $\frL_f(A, E)$ and $f(A)$ could be computed together with a computational complexity of around three times the complexity of $f(A)$. There exist routines to compute Fr{\'e}chet derivatives of the exponential function in open source or commercial packages. We have mentioned $\frL_{\exp}(A, E) = \exp A\sum_{n=0}^{\infty} \frac{(-1)^n}{(n+1)!}\ad_A^n E$. For Jacobi fields of the Grassmann manifold, we need to evaluate $\frL_{\csr}$ and $\frL_{\ssr}$, where we recall $\csr z$ and $\ssr$ are analytic continuations of $\cos z^{1/2}$ and $z^{-1/2}\sin z^{1/2}$. Based on the ideas in \cite{HighamCosine}, the evaluation for $A$ with small eigenvalues could be done by Pad{\'e} approximant, then use functional equations for $\csr$ and $\ssr$ (based on equations for {\it cosine} and {\it sine}) for $A$ with large eigenvalues. To evaluate time derivatives of Jacobi fields, the following formula is handly. While it is easy to prove, we could not find a reference.
\begin{lemma} If $f(x) = \sum_{i=0}^{\infty} f_i x^i$ is an analytic function near zero and $f'(x) = \sum_{i=0}^{\infty} (i+1)f_{i+1} x_i$, then
  \begin{equation}
    \frac{d}{dt}L_f(tA, tE) = AL_{f'}(tA, tE) + Ef'(tA)
\end{equation}    
\end{lemma}  
\begin{proof} We only need to prove this for monomials $f(x) = x^n$. This follows from
  $$  \frac{d}{dt}\sum_{a+b=n-1}t^nA^aEA^b =  A(t^{n-1}\sum_{a+b=n-2}nA^a E A^b) + E(n  t^{n-1} A^{n-1})
  $$
\end{proof}
\end{appendix}
\bibliographystyle{amsplain}
\bibliography{RiemannianCurvature}
\end{document}